\documentclass{article}
\usepackage{setspace}
\usepackage{titlesec}
\usepackage[utf8]{inputenc}
\usepackage{tabularx}	
\usepackage{etoolbox}
\usepackage{mathtools}
\usepackage{amsmath}
\usepackage{amssymb} 
\usepackage{amsfonts}
\usepackage{mathrsfs} 
\usepackage{amsthm,pifont}
\usepackage{graphicx}
\usepackage{bbm}
\usepackage[colorlinks=true, pdfstartview=FitV, linkcolor=blue,
citecolor=blue, urlcolor=blue, pagebackref=false]{hyperref}
\usepackage{cleveref}
\usepackage[numbers,sort&compress]{natbib}
\usepackage{xcolor}
\usepackage{tikz,pgfplots}
\usepackage{setspace}
\usepackage{braket}
\usepackage{leftidx,tensor}
\usepackage{dsfont}
\usepackage{calc}
\usepackage{caption}
\usepackage{subcaption}
\usepackage{color}
\usepackage{comment}
\usepackage{thm-restate}
\usetikzlibrary{decorations.markings}
\usetikzlibrary{shapes.geometric}
\usetikzlibrary{patterns}
\tikzstyle{vertex}=[circle, draw, inner sep=0pt, minimum size=6pt]
\definecolor{blobteal}{RGB}{66, 200, 184}
\definecolor{blobred}{RGB}{200, 75, 75}
\definecolor{blobpink}{RGB}{235, 164, 164}
\newcommand{\vertex}{\node[vertex]}
\newcommand{\sss}[1]{{\scriptscriptstyle #1}}
\newcommand{\ssup}[1] {{\scriptscriptstyle{({#1}})}}

\usepackage[utf8]{inputenc}
\usepackage[english]{babel}
\usepackage{enumitem}

\let\OLDthebibliography\thebibliography
\renewcommand\thebibliography[1]{
  \OLDthebibliography{#1}
  \setlength{\parskip}{2.5pt}
  \setlength{\itemsep}{2.5pt plus 0.3ex}
}

\allowdisplaybreaks

\newtheorem{theorem}{Theorem}[section]
\makeatletter
\patchcmd{\ttlh@hang}{\parindent\z@}{\parindent\z@\leavevmode}{}{}
\patchcmd{\ttlh@hang}{\noindent}{}{}{}
\makeatother
\titleformat*{\section}{\large\bfseries}
\titleformat*{\subsection}{\small\bfseries}
\titleformat*{\subsubsection}{\small\bfseries}
\titleformat*{\paragraph}{\small\bfseries}
\titleformat*{\subparagraph}{\small\bfseries}

\newcommand{\N}{\mathbb{N}}
\newcommand{\R}{\mathbb{R}}

\newcommand{\E}{\mathbb{E}}

\newcommand{\p}{\mathbb{P}}
\newcommand{\md}{\ensuremath{\mathrm{d}}}
\newcommand{\eps}{\varepsilon}

\newcommand{\cA}{\mathcal{A}}
\newcommand{\cB}{\mathcal{B}}

\newcommand{\cE}{\mathcal{E}}
\newcommand{\cF}{\mathcal{F}}
\newcommand{\cG}{\mathcal{G}}

\newcommand{\cU}{\mathcal{U}}

\newcommand{\pa}{\mathrm{pa}}

\newcommand{\hs}{\widehat{\sigma}}
\newcommand{\rt}{\mathrm{t}}
\newcommand{\rd}{\mathrm{d}}

\newcommand{\Ind}[1]{\mathbbm{1}{\{#1\}}}

\DeclareMathOperator*{\argmin}{arg\,min}

\newtheorem{lemma}[theorem]{Lemma}
\newtheorem{proposition}[theorem]{Proposition}
\newtheorem{definition}[theorem]{Definition}
\newtheorem{corollary}[theorem]{Corollary}
\newtheorem{claim}[theorem]{Claim}

\newtheorem{problem}[theorem]{Open problem}

\theoremstyle{definition}
\newtheorem{remark}[theorem]{Remark}
\numberwithin{equation}{section}

\usepackage{geometry}
\title{History estimation in random recursive trees:\\Pointwise approach via iterated Jordan centralities}
\author{Johannes Bäumler\footnote{Mathematical Institute,
University of Koblenz,
Germany.} \and Simon Briend\footnote{Department of Mathematics and Computer Science,
UniDistance,
Brig, Switzerland.} \and Joost Jorritsma\footnote{Department of Statistics,
University of Oxford,
United Kingdom.\\
\indent\hspace{15pt} \emph{Email: }\href{mailto:jbaeumler@uni-koblenz.de}{jbaeumler@uni-koblenz.de}, 
\href{mailto:simon.briend@unidistance.ch}{simon.briend@unidistance.ch}, 
\href{mailto:joost.jorritsma@stats.ox.ac.uk}{joost.jorritsma@stats.ox.ac.uk}.}}

\begin{document}

\maketitle

\begin{abstract}
We study the problem of estimating the arrival times of vertices in a uniform random recursive tree from its unlabeled structure. We adopt a pointwise perspective and analyze the distribution of the relative estimation error, and derive tail bounds that are uniform in both the vertex and the tree size. For the ranking induced by Jordan centrality, the probability that the estimate exceeds the true arrival time by a factor $S$ decays on the order of $1/S$, while the probability of underestimating the arrival time by a factor $1/S$ decays exponentially in $S$. We introduce a refined centrality measure whose overestimation tail decays on the order of $(\log S)/S^{2}$, at the cost of a heavier lower tail of order $1/S^{2}$.
These results reveal a tradeoff between upper- and lower-tail performance in arrival-time estimation that is invisible to the previously studied risk functional. Nevertheless, the refined centrality measure attains the optimal order of the risk for all its parameter values.
\end{abstract}

 \medskip
\noindent 
{\small

}
\hypersetup{linkcolor=black}
\hypersetup{linkcolor=blue}


\section{Introduction and main results}
Network archaeology asks what can be inferred about the growth history of a network from its final structure. A central instance is root-finding: given a graph generated by a growth process but observed without arrival-time labels, identify the vertex that arrived first. This problem models, for example, the detection of patient zero in infection networks \cite{hens2012robust, keeling2005networks}, and has been studied extensively for random tree models, including uniform random recursive trees \cite{BuDeLu17,addarioberry2024leafstrippinguniformattachment,addarioberry2024optimalrootrecoveryuniform,BRRS26,JoLo17a}, preferential attachment trees \cite{contat2024eve,dereich2009random, BuDeLu17}, nearest neighbor trees \cite{brandenberger2024finding}, and Bienaymé–Galton–Watson trees \cite{BrDeGo21}. Related questions have also been explored for random graphs beyond trees, see \cite{briend2023archaeology, banerjee2023degree, CrXu21a}.
Root-finding algorithms commonly proceed either via Bayesian approaches \cite{CrXu21, young2019phase} or by ranking vertices according to a centrality statistic. The latter builds upon the idea that the root is among the most `central' nodes in the graph, and a variety of centrality measures have been considered, including degree, rumor centrality, and Jordan centrality.

\smallskip 
Motivated by these results, recent work has asked whether centrality-based rankings can be used to recover more detailed temporal information. Rather than identifying only the earliest vertex, the goal is to estimate the arrival time of an arbitrary vertex $v$, given only the unlabeled tree. This problem was first investigated for preferential attachment trees in \cite{magner2018times, young2019phase}, and more recently for uniform random recursive trees in \cite{briend2024estimating}. In these works, the arrival time of $v$ is estimated by $\hs(v)$, defined as the rank of $v$ induced by a given centrality measure. In \cite{briend2024estimating}, the performance of $\hs$ is evaluated through the risk functional
 \begin{equation}\label{eq:RiskDef2}
    R_{\alpha}(\hs):=\E\Bigg[\sum_{v\in V_n} \frac{|\hs(v)-v|}{v^{\alpha}}\bigg], \qquad \alpha>0.
 \end{equation}

\smallskip 
\noindent\textbf{Our contribution.\ }
In this paper, we take an alternative, pointwise perspective. For a vertex $v\in[n]$ in a uniform random recursive tree, we study the distribution of the relative error $\hs(v)/v$ and derive tail bounds that are uniform in both $v$ and~$n$. 
Proposition~\ref{prop:lower_bound} shows that every label-invariant estimator incurs errors of order $v$ with constant probability, making relative error the natural scale for arrival-time estimation.
Moreover,  this perspective distinguishes between overestimation and underestimation. While this distinction is absent in root finding and only weakly reflected by the risk functional~\eqref{eq:RiskDef2}, it reveals qualitative differences between  estimators that are invisible from the risk alone. At the same time, pointwise bounds still determine the asymptotic order of the global risk.
We summarize our main findings informally.

\begin{itemize}[itemsep=0pt, leftmargin=*]
    \item[-] \emph{Jordan-2 estimator}.\ We introduce a new centrality measure $\phi^{\ssup{2}}$ and study the ranking $\hs_2$ based on this centrality measure. Theorem \ref{thm:main} shows that the upper tail decays faster than the ranking based on Jordan centrality, as 
    \begin{equation}\label{eq:main asymptotic}
        \p\big(\hs_2(v)\ge Sv\big)\asymp \frac{1+\log S}{S^2},\qquad S\to\infty.
    \end{equation}
    This improvement comes at the cost of a heavier lower tail, as
    $\p(\hs_2(v)\le v/S)\asymp S^{-2}$ for this estimator. Furthermore, the tail bounds \eqref{eq:main asymptotic} allow us to determine the asymptotic growth of the risk for any value of $\alpha >0$ (Corollary \ref{cor:new method}), which turns out to be optimal up to constant factors (Corollary \ref{cor:optimal}).
    
    \item[-] \emph{Jordan estimator}.\  For the estimator \(\hs_J\) induced by Jordan centrality, we obtain matching upper and lower bounds, up to constants, for the upper tail uniformly in $v$ and $n$; see Theorem~\ref{thm:jordan}. That is,
    \[
        \p\big(\hs_J(v)\ge Sv\big)\asymp \frac{1}{S}, \qquad S\to\infty.
    \]
    We also show that for the estimator induced by Jordan centrality, the lower tail decays much faster than the upper tail, as $\log\p(\hs_J(v)\le v/S)\asymp -S$.
    This upper-tail estimate allows us to identify the order of the parametrized risk \eqref{eq:RiskDef2} for Jordan centrality for all values of $\alpha > 0$, determining regimes in which this order is suboptimal relative to that of $\hs_2$ (Corollary~\ref{cor:new method}). 
\end{itemize}

Our analysis of these two estimators suggests a qualitative tradeoff that we expect to hold more broadly (for instance, for rankings based on rumor centrality): improvements in the upper-tail behavior of the estimation error come at the expense of worse lower-tail behavior, and vice versa. From this perspective, the Jordan-2 estimator may be viewed as comparatively balanced, in that both its upper and lower tails exhibit quadratic decay, up to logarithmic factors.  
Next, we formalize our model, the estimators, and our results.

\subsection{The model: uniform random recursive trees}\label{sec:model}

We study uniform random recursive trees (RRT). We denote by $T_n=(V_n,E_n)$ a RRT of size $n$, where the vertex set is $V_n=[n]$, and the random edge set $E_n$ is sampled recursively as follows. Start with $T_1$, consisting of a single vertex labeled $1$, the root, and no edges. Given $T_n$, the tree $T_{n+1}$ is constructed by first choosing $u\in \{1,\ldots,n\}$ uniformly at random, independent of $T_1,\ldots,T_n$, and then adding a new vertex labeled $n+1$ and the undirected edge $\{u,n+1\}$ to the tree. We also write $\mathrm{pa}(n+1)=u$ and say that $u$ is the \emph{parent} of $n+1$. 

\smallskip
The tree $T_n$ is \emph{recursive}: its vertex labels are increasing on any simple (non-backtracking) path starting at the root. We sometimes relabel trees by applying a permutation to their vertex labels. When the vertex labels of a labeled tree $\rt$ are permuted by $\tau$, we denote the resulting tree by $\rt^{\tau}$. In general, $\rt^{\tau}$ is not a recursive tree.

\medskip 
\noindent\textbf{History estimation.}\ Our goal is to estimate the history of the tree, that is, to estimate the arrival time of each vertex using only the unlabeled tree structure. Formally,  given a vertex-labeled tree $\rt_n$ on $n$ vertices, an estimator $\hs$ is a (possibly random) bijection $\hs:V_n \to [n]$. Throughout, we only consider \emph{label-invariant estimators}, which satisfy, for any fixed labeled tree $\rt$, any permutation of its vertices $\tau$, and any vertex $u$,

\begin{equation}\label{eq:def_label-invariant}
    \hs\left(u,\rt\right)
    \quad 
    \overset{\mathcal{L}}{=} \quad
    \hs\left(\tau(u),\rt^{\tau}\right),
\end{equation}
where the equality would be deterministic if $\hs$ itself were deterministic given $\rt$.

\subsection{Jordan and Jordan-$2$ centrality}
We next introduce the Jordan centrality and the Jordan-2 centrality, and the orderings introduced by them.
Given a tree $\rt$ and a vertex $v$ in this tree, let $(\rt, v)$ denote the tree rooted at $v$, and, for a vertex $u$, let $(\rt, v)_{u\downarrow }$ be the subtree of $u$  in this rooted tree.
We define the \emph{Jordan centrality}\footnote{
The more common convention defines Jordan centrality by $\psi_{\rt}(v)=\max_{u\sim v}|(\rt,v)_{u\downarrow}|=|\rt|-\phi_\rt(v)$. We use $\phi_{\rt}$ because it typically coincides with the size of the fringe subtree $(t,1)_{v\downarrow}$ and extends directly to the Jordan-$k$ centrality defined in Section \ref{sec:discussion}.} 
of a vertex $v$ by 
\begin{equation}\label{eq:neg-jordan}
    \phi_\rt(v) := \min_{u\sim v}  \big|(\rt, u)_{v\downarrow}\big|,
\end{equation}
where the minimum is taken over all neighbors of $v$ in $\rt$. 
Ordering vertices decreasingly by $\phi_\rt$ (with random tie-breaking) yields the \emph{Jordan ordering} $\hs_J$ (with random tie-breaking). Thus, rank \(1\) is assigned to the vertex with largest centrality, and smaller values of \(\hs_J(v)\) correspond to earlier estimated arrival times.
For vertices $v$ with $|(\rt,1)_{v\downarrow}|\leq |\rt|/2$, the minimum in \eqref{eq:neg-jordan} is attained at the parent of $v$, and $\phi_\rt(v)=|(\rt,1)_{v\downarrow}|$ is the size of the \emph{fringe tree} of $v$. For $v\in\rt$, we denote the vertex attaining the minimum in \eqref{eq:neg-jordan}, an estimate of the parent of $v$, by 
\begin{equation}\label{eq:v1}
v^{\ssup{1}}:=\argmin_{u\in \rt: u\sim v} \big|(\rt, u)_{v\downarrow}\big|
\end{equation} 
with tie-breaking discussed in Definition \ref{def:jordan2}.

\begin{definition}[Jordan-$2$ centrality]\label{def:jordan2}
    For a finite tree $\rt=(V_n,E_n)$, we define the \emph{Jordan-$2$ centrality} $\phi_{\rt}^{\ssup{2}} : V_n \to \R_{\geq 0}$ by
\begin{equation}
    \phi_\rt^{\ssup{2}}(v):=  \phi_\rt(v) \cdot \big(\phi_\rt(v^{\ssup{1}}) \vee  \phi_\rt(v) \big).\label{eq:NNBJordanDef}
\end{equation}
If $v^{\ssup{1}}$ is not uniquely defined, we let $v^{\ssup{1}}$ be such that the resulting choice minimizes $\phi_\rt^{\ssup{2}}$.  The Jordan-2 ordering $\widehat\sigma_2$ is given by ordering vertices decreasingly by $\phi_\rt^{\ssup{2}}$, and breaking ties uniformly at random.
\end{definition}
For the uniform recursive tree $T_n$, we also write $\phi_{n}$ and $\phi_{n}^{\ssup{2}}$ for $\phi_{T_n}$ and $\phi_{T_n}^{\ssup{2}}$, respectively.
Note that the Jordan-$2$ centrality also satisfies
\begin{equation*}
    \phi_\rt^{\ssup{2}}(v) =  \phi_\rt(v) \cdot \big(\phi_\rt(v^{\ssup{1}}) \vee  \phi_\rt(v) \big) = \big|(\rt, v^{\ssup{1}})_{v\downarrow}\big|\cdot\left(\big|(\rt, v^{\ssup{2}})_{v^{\ssup{1}}\downarrow}\big| \vee \big|(\rt, v^{\ssup{1}})_{v\downarrow}\big|\right),
\end{equation*}
where $v^\ssup{2}$ is an estimator for the parent of $v^\ssup{1}$, that is,
\begin{equation}\label{eq:v2}
v^{\ssup{2}}:= 
\argmin_{u\in \rt: u\sim v^{\ssup{1}}} \big|(\rt,\, u)_{v^\ssup{1}\downarrow}\big|.
\end{equation}
Figure~\ref{fig:NBJ} presents two examples of the relative positions of $v, v^{\ssup{1}}$, and $v^{\ssup{2}}$. 
\begin{figure}[t]
    \centering
     \begin{minipage}[t]{0.48\textwidth}
        \centering
        \begin{tikzpicture}[scale=0.6,
		vertex/.style={circle, fill=black, minimum size=7pt, inner sep=0pt},
		edge/.style={thick, black}
		]
		
		
		\node[vertex, label={[right, xshift=1pt, yshift=-1pt] $v^{(2)}$}] (v2) at (0, 0) {};
		
		\node[vertex, label={[right, xshift=1pt, yshift=-1pt] $v^{(1)}$}] (v1) at (-3.5, -2.5) {};
		\node[vertex] (m) at (0.8, -2.5) {};
		\node[vertex] (r) at (4.2, -2.5) {};
		
		\node[vertex, label={[right, xshift=1pt] $v$}] (v) at (-4.8, -4.5) {};
		\node[vertex] (v1r) at (-2.2, -4.5) {};
		
		\node[vertex] (v_1) at (-5.4, -6.5) {};
		\node[vertex] (v_2) at (-4.2, -6.5) {};
		
		\node[vertex] (v1r_1) at (-3.1, -6.5) {};
		\node[vertex] (v1r_2) at (-2.2, -6.5) {};
		\node[vertex] (v1r_3) at (-1.3, -6.5) {};
		
		\node[vertex] (mm) at (0.8, -4.5) {};
		\node[vertex] (mr) at (2.5, -4.5) {};
		
		\node[vertex] (mm_1) at (-0.1, -6.5) {};
		\node[vertex] (mm_2) at (0.5, -6.5) {};
		\node[vertex] (mm_3) at (1.1, -6.5) {};
		\node[vertex] (mm_4) at (1.7, -6.5) {};
		
		\node[vertex] (rc) at (4.2, -4.5) {};
		
		\node[vertex] (rc_1) at (3.6, -6.5) {};
		\node[vertex] (rc_2) at (4.8, -6.5) {};

		\pgfdeclarelayer{bg}
		\pgfsetlayers{bg,main}
		
		\begin{pgfonlayer}{bg}
			\fill[blobteal] plot [smooth cycle, tension=0.5] coordinates {
				(-3.0, -1.2)
				(-1.0, -3.8)
				(-0.7, -6.5)
				(-2.5, -7.4)
				(-5.4, -7.4)
				(-6.3, -6.0)
				(-5.2, -3.2)
				(-4.2, -1.5)
			};
			\fill[blobpink] plot [smooth cycle, tension=0.5] coordinates {
				(-4.8, -4)
				(-4.2, -4.5)
				(-3.8, -6.6)
				(-4.8, -7.1)
				(-5.8, -6.7)
				(-5.2, -4.5)
			};
		\end{pgfonlayer}
		
		
		\draw[edge] (v2) -- (v1);
		\draw[edge] (v2) -- (m);
		\draw[edge] (v2) -- (r);
		
		\draw[edge] (v1) -- (v);
		\draw[edge] (v1) -- (v1r);
		\draw[edge] (v) -- (v_1);
		\draw[edge] (v) -- (v_2);
		\draw[edge] (v1r) -- (v1r_1);
		\draw[edge] (v1r) -- (v1r_2);
		\draw[edge] (v1r) -- (v1r_3);
		
		\draw[edge] (m) -- (mm);
		\draw[edge] (m) -- (mr);
		\draw[edge] (mm) -- (mm_1);
		\draw[edge] (mm) -- (mm_2);
		\draw[edge] (mm) -- (mm_3);
		\draw[edge] (mm) -- (mm_4);
		
		\draw[edge] (r) -- (rc);
		\draw[edge] (rc) -- (rc_1);
		\draw[edge] (rc) -- (rc_2);
		
	\end{tikzpicture}
        \caption*{}
    \end{minipage}
    \hfill
    \begin{minipage}[t]{0.48\textwidth}
        \centering
        \begin{tikzpicture}[scale=0.6,
		vertex/.style={circle, fill=black, minimum size=7pt, inner sep=0pt},
		edge/.style={thick, black}
		]
		
		
		\node[vertex] (1) at (2,6) {};
		
		\node[vertex, label={[left, xshift=0pt, yshift=1pt] $v^{(1)}$}] (21) at (-1,4) {};
		\node[vertex] (22) at (5,4) {};
		
		\node[vertex] (31) at (-2.5,2) {};
		\node[vertex, label={[left, xshift=-1pt, yshift=1pt] $v^{(2)}=v$}] (32) at (0.5,2) {};
		\node[vertex] (33) at (3,2) {};
		\node[vertex] (34) at (5,2) {};
		\node[vertex] (35) at (6,2) {};
		
		\node[vertex] (41) at (-3,0) {};
		\node[vertex] (42) at (-2,0) {};
		\node[vertex] (43) at (-1,0) {};
		\node[vertex] (44) at (0,0) {};
		\node[vertex] (45) at (1,0) {};
		\node[vertex] (46) at (2,0) {};
		\node[vertex] (47) at (3,0) {};
		\node[vertex] (48) at (4,0) {};
		\node[vertex] (49) at (5,0) {};
		
		\node[vertex] (51) at (0,-2) {};
		\node[vertex] (52) at (1,-2) {};
		
		\draw[edge] (21) -- (1);
		\draw[edge] (22) -- (1);
		
		\draw[edge] (31) -- (21);
		\draw[edge] (32) -- (21);
		\draw[edge] (33) -- (21);
		\draw[edge] (34) -- (22);
		\draw[edge] (35) -- (22);
		
		\draw[edge] (41) -- (31);
		\draw[edge] (42) -- (31);
		\draw[edge] (43) -- (32);
		\draw[edge] (44) -- (32);
		\draw[edge] (45) -- (32);
		\draw[edge] (46) -- (32);
		\draw[edge] (47) -- (33);
		\draw[edge] (48) -- (33);
		\draw[edge] (49) -- (34);
		
		\draw[edge] (51) -- (44);
		\draw[edge] (52) -- (45);

		\pgfdeclarelayer{bg}
		\pgfsetlayers{bg,main}
		
		\begin{pgfonlayer}{bg}
			\fill[blobpink] plot [smooth cycle, tension=0.5] coordinates {
				(0.5, 2.5)
				(2, 1.2)
				(2.2, -1.5)
				(1, -2.5)
				(0, -2.5)
				(-1.3, -1.5)
				(-1.5, 1)
			};
			
			\fill[blobteal] plot [smooth cycle, tension=0.7] coordinates {
				(2, 6.5)
				(5, 4.5)
				(6.5, 2)
				(6, -0.5)
				(3.4, -0.5)
				(2.5, 0)
				(2.5, 1.5)
				(1.5, 2.5)
				(0, 3)
				(-1.75, 3.5)
				(-1, 5)
			};

		\end{pgfonlayer}
	\end{tikzpicture}
        \caption*{}
    \end{minipage}
    \vspace{-0.4cm}
    \caption{The Jordan-$2$ centrality of a vertex $v$ in two trees $\rt$. The light-red tree  represents $(\rt, v^{\ssup{1}})_{v\downarrow}$.
    The tree contained in the turquoise area is $(\rt, v^{\ssup{2}})_{v^{\ssup{1}}\downarrow}$. The Jordan-$2$ centrality $\phi_\rt^\ssup{2}(v)$ is given by the product of the number of vertices contained in the light-red and the turquoise areas for these examples. Typically, the turquoise tree is a superset of the light-red tree, but exceptions are possible, as shown in the right tree where $v^\sss{(2)}=v$. When the light-red tree is larger than the turquoise tree, the Jordan-2 centrality is given by the square of the size of the light-red tree.
    }
    \label{fig:NBJ}
    \vspace{-0.2cm}
\end{figure}

\smallskip 
Typically, the Jordan-2 is obtained by multiplying the Jordan centrality of a vertex by the Jordan centrality of an estimator of the parent of this vertex. This allows for more robust history estimation, as stated in Theorems \ref{thm:main} and \ref{thm:jordan} below. One problem with the Jordan centrality defined in \eqref{eq:neg-jordan} is that it is very sensitive to $v$ remaining a leaf for a long time. Considering also the ancestral line of the vertex leads to $\phi_n^{\ssup{2}}$ being less sensitive to the event that $v$ remains a leaf for a long time. A refined heuristic for the Jordan-2 centrality is given at the beginning of Section~\ref{sec:NBjordan}.
\smallskip

The maximum $(\phi_\rt(v^{\ssup{1}}) \vee  \phi_\rt(v) )$ in \eqref{eq:NNBJordanDef} is useful to handle the exceptional configuration where $v^\ssup{1} \neq \mathrm{pa}(v)$ (see e.g.\ Lemma~\ref{lem:minima}) and to ensure $\phi_n^{\ssup{2}}(v) \geq \phi_n(v)^2$. Additionally, it simplifies some calculations throughout our analysis. To get intuition, we make an observation. For each vertex $v\ge 2$, 
\begin{equation}\label{eq:jordan-upper}
    \phi_\rt(v) = \min_{u\sim v}\big|(\rt,u)_{v\downarrow}\big| \le  \big|(\rt,\pa(v))_{v\downarrow}\big|=\big|(\rt,1)_{v\downarrow}\big|.
\end{equation}
As a result, regardless if $v^\ssup{1}$ is the parent or a child of $v$, we also have $\phi_\rt(v^\ssup{1})\le \big|(\rt,1)_{\pa(v)\downarrow}\big|$, so 
\begin{equation}\label{eq:upper-bound-jordan2}
\begin{aligned}
\phi^\ssup{2}_{\rt}(v)  = \phi_\rt(v) \cdot \big(\phi_\rt(v^\ssup{1})\vee \phi_\rt(v)\big) 
&
\le \big|(\rt,1)_{v\downarrow}\big| \cdot \Big( \big|(\rt,1)_{\pa(v)\downarrow}\big| \vee \big|(\rt,1)_{v\downarrow}\big| \Big)\\& \leq \big|(\rt,1)_{v\downarrow}\big|\cdot \big|(\rt,1)_{\pa(v)\downarrow}\big|, \qquad \qquad\qquad  v\ge 2.
\end{aligned}
\end{equation}
This upper bound is sharp if $\big|(\rt,1)_{\pa(v)\downarrow}\big|\le |\rt|/2$.
Lemma \ref{lem:minima} below establishes a lower bound.


\subsection{Main results}

We start with a lower bound on the performance of arbitrary label-invariant estimators, which we prove in Section \ref{sec:lower bounds}. This lower bound implies that $|\hs(v)-v|$ is of the same order as $v$ (with positive probability), for all label-invariant estimators~$\hs$.

\begin{proposition}[Pointwise estimation for arbitrary estimators]\label{prop:lower_bound}
    Let $\hs$ be a label-invariant estimator of the arrival times in $T_n$. Then, for any $n\ge 3$, and any $v\in [n]$,
    \begin{equation*}
        \p\Big( |\hs(v)-v| \geq \frac{v}{56} \Big) \geq \frac{1}{224}.
    \end{equation*}
\end{proposition}

This proposition implies that one should focus on the relative error of $\hs(v)/v$. Unlike root-finding, arrival-time estimation naturally distinguishes between two directions of error: an arrival time may be estimated substantially later or earlier than it truly is. This distinction is only weakly reflected by the risk functional \eqref{eq:RiskDef2}: estimators with markedly different pointwise behavior may nevertheless have the same leading-order risk; see Section \ref{sec:risk}. 
We therefore distinguish between \emph{overestimation} and \emph{underestimation}, and control the upper tail (the event $\left\{ \hs(v) \geq Sv \right\}$) and lower tail (the event $\left\{ \hs(v) \leq v/S \right\}$) separately, for large values of $S$, and uniformly over $v \in [n]$. The next theorem describes our results for the Jordan-2 ordering.

\begin{theorem}[Pointwise estimation of the Jordan-$2$ ordering]\label{thm:main}
    For every $\eps > 0$, there exist constants $c, C > 0$ such that for any $n \ge 1$, any $v \in [n]$ and any $S\in \left[ 1,c (n/v)^{1-\eps} \right]$, the upper tail satisfies
    \begin{equation}\label{jordan2:overest_combined}
        c\frac{1+\log S}{S^2} 
        \ \leq\ \p \big( \hs_{2}(v) \geq S v \big) 
        \ \leq\ C\frac{ 1+\log S}{S^2}.
    \end{equation}
    The upper bound holds for all $S\ge  1$.

    Moreover, there exist constants $c', C', \delta > 0$ such that for any $n \ge 1$, any $v \in [n/4]$, and any $S\in[1,\delta\sqrt{v}]$, the lower tail satisfies
    \begin{equation}\label{jordan2:underest_combined}
        \frac{c'}{S^2} 
        \ \leq\  \p\Big(\hs_2(v) \leq \frac{1}{S}v\Big) 
        \ \leq\ \frac{C'}{S^2}.
    \end{equation}
    The upper bound holds for all $S \geq 1$.
\end{theorem}
For $v=1$, the upper bound on the probability of overestimating the label of $v$, \eqref{jordan2:overest_combined}, directly implies the following corollary for root-finding with the Jordan-2 ordering.

\begin{corollary}[Root detection using Jordan-2 centrality]
    There exists a constant $C>0$ such that for all $\varepsilon\in(0,1)$, the set $\mathcal{S}_2$ containing the $C\sqrt{\log(1/\varepsilon)/\varepsilon}$ vertices with largest Jordan-2 centrality contains vertex~$1$ with probability at least $1-\varepsilon$.
\end{corollary}
For small values of $\eps>0$, the required set $\mathcal{S}_2$ defined above is much larger compared to the optimal confidence set at confidence level $1-\eps$, whose size is of order $\exp\big(\Theta\big(\sqrt{\log(1/\varepsilon)}\big)\big)$,  see \cite{addarioberry2024optimalrootrecoveryuniform}. On the other hand, when using the Jordan ordering, the required set has size that scales as $1/\varepsilon$, see \cite[Theorem 3]{BuDeLu17}, which also follows from the next theorem.

\begin{theorem}[Pointwise estimation of the Jordan ordering]\label{thm:jordan}
    For the Jordan-based estimator $\hs_{J}$, there exist constants $c, C, \delta > 0$ such that for all $n \ge 1$, any $v \in [n]$ and any $S \in [1, \delta n/v]$, the upper tail satisfies
    \begin{equation}\label{jordan:overest_combined}
        \frac{c}{S} 
        \ \leq\ \p \big( \hs_{J}(v) \geq S v\big) 
        \ \leq\ \frac{C}{S}.
    \end{equation}
    The upper bound holds for all $S \geq 1$.

    Moreover, there exist constants $c', c'', C', C'', \delta' > 0$ such that for all $n \ge 1$, any vertex $v \in [n/2]$, and any $S \in [1, \min(\sqrt{v},\delta'v)]$, the lower tail satisfies
    \begin{equation}\label{jordan:underest_combined}
        c''\exp(-C''S) 
        \ \leq\ \p \Big( \hs_{J}(v) \leq \frac{1}{S} v\Big) 
        \ \leq\ C'\exp(-c'S).
    \end{equation}
    The upper bound holds for all $v \in [n]$ and $S \in [1, \sqrt{v}]$; the lower bound holds for all $v\in[n/2]$ and $S\in[1,\delta'v]$.
\end{theorem}

Theorems \ref{thm:main} and  \ref{thm:jordan} have straightforward consequences when plugged into the definition of the risk \eqref{eq:RiskDef2} introduced by Briend \emph{et al.}\ \cite{briend2024estimating}, improving on their results. We state these consequences in Section \ref{sec:risk}. Before that, we present a numerical illustration and a heuristic for the different behavior of the Jordan-2 ordering vs.\ the Jordan ordering.

\medskip 
\noindent\textbf{Numerical illustration.\ }Figure \ref{fig:Simulation:underestimation} illustrates the tail behaviors predicted by Theorems \ref{thm:main} and~\ref{thm:jordan}. The simulations are based on $2\times 10^4$ independent uniform random recursive trees of size $n=10^4$. When evaluating the empirical probability of the event $\{\hs(v) \geq Sv\}$, we averaged over the vertices $v\in(21,50]$, while for the plots of the lower tail corresponding to the event $\{\hs(v)\le v/S\}$, we averaged over the vertices $v \in (8000, 9000]$. These choices reflect the finite-size constraints: large overestimation factors are only visible when $v$ is small compared to $n$, whereas large underestimation factors require $v$ to be sufficiently large. 

\begin{figure}[h]
	{
    \begin{tikzpicture}[scale=1]
	\begin{loglogaxis}[
		xlabel={$S$},
        xlabel style={yshift=8pt},
		ylabel={},
		title={Upper tail},
        title style={yshift=-5pt},
		grid=both,
		major grid style={gray},
		minor grid style={gray!20},
		xtick distance=10,
		ytick distance=10,
		width=0.35\textwidth,
		height=5.5cm,
        legend pos = south west, 
        legend cell align=left,
		]
		
		\addplot[
		thick, red
		]
		coordinates {
			(5, 0.21460666666666667)
			(6, 0.1755)
			(7, 0.14832)
			(8, 0.12851666666666667)
			(9, 0.11309333333333334)
			(10, 0.10096166666666667)
			(11, 0.091115)
			(12, 0.08291833333333333)
			(13, 0.07618333333333334)
			(14, 0.070445)
			(15, 0.06554333333333333)
			(16, 0.061121666666666664)
			(17, 0.05733333333333333)
			(18, 0.05401333333333333)
			(19, 0.05102166666666667)
			(20, 0.04835)
			(21, 0.04586833333333333)
			(22, 0.04365)
			(23, 0.041705)
			(24, 0.03983)
			(25, 0.03807166666666666)
			(26, 0.036625)
			(27, 0.03511333333333334)
			(28, 0.03378833333333333)
			(29, 0.032595)
			(30, 0.03145)
			(31, 0.030336666666666668)
			(32, 0.029248333333333335)
			(33, 0.02827)
			(34, 0.027443333333333333)
			(35, 0.026485)
			(36, 0.025705)
			(37, 0.02492)
			(38, 0.024213333333333333)
			(39, 0.023545)
			(40, 0.022971666666666668)
			(41, 0.022353333333333333)
			(42, 0.021716666666666665)
			(43, 0.021148333333333335)
			(44, 0.020611666666666667)
			(45, 0.02007)
			(46, 0.01963)
			(47, 0.019175)
			(48, 0.018708333333333334)
			(49, 0.018306666666666666)
			(50, 0.017918333333333335)
			(51, 0.017513333333333332)
			(52, 0.01704)
			(53, 0.016658333333333334)
			(54, 0.016233333333333332)
			(55, 0.015896666666666667)
			(56, 0.015538333333333333)
			(57, 0.015201666666666667)
			(58, 0.014903333333333333)
			(59, 0.01461)
			(60, 0.014375)
			(61, 0.014166666666666666)
			(62, 0.013953333333333333)
			(63, 0.013731666666666666)
			(64, 0.013513333333333334)
			(65, 0.01334)
			(66, 0.013163333333333334)
			(67, 0.012978333333333333)
			(68, 0.01274)
			(69, 0.012485)
			(70, 0.012236666666666667)
			(71, 0.011995)
			(72, 0.011793333333333333)
			(73, 0.011553333333333334)
			(74, 0.011378333333333334)
			(75, 0.011168333333333334)
			(76, 0.010993333333333334)
			(77, 0.010806666666666666)
			(78, 0.010661666666666667)
			(79, 0.010491666666666667)
			(80, 0.010333333333333333)
			(81, 0.010171666666666667)
			(82, 0.010028333333333334)
			(83, 0.009915)
			(84, 0.009816666666666666)
			(85, 0.009706666666666667)
			(86, 0.009591666666666667)
			(87, 0.009456666666666667)
			(88, 0.009346666666666666)
			(89, 0.00922)
			(90, 0.009118333333333334)
			(91, 0.009011666666666666)
			(92, 0.008925)
			(93, 0.00886)
			(94, 0.00876)
			(95, 0.008706666666666666)
			(96, 0.008633333333333333)
			(97, 0.008533333333333334)
			(98, 0.00847)
			(99, 0.008411666666666666)
			(100, 0.008353333333333334)
		};
		\addlegendentry{Jordan}
		
		\addplot[
		thick, blue
		]
		coordinates {
			(5, 0.14660333333333334)
			(6, 0.10718333333333334)
			(7, 0.08226)
			(8, 0.06528833333333334)
			(9, 0.05309666666666667)
			(10, 0.04389666666666667)
			(11, 0.036925)
			(12, 0.031428333333333336)
			(13, 0.0271)
			(14, 0.023705)
			(15, 0.020768333333333333)
			(16, 0.018361666666666665)
			(17, 0.01623)
			(18, 0.014493333333333334)
			(19, 0.013075)
			(20, 0.011833333333333333)
			(21, 0.010745)
			(22, 0.009788333333333333)
			(23, 0.008973333333333333)
			(24, 0.008196666666666666)
			(25, 0.00749)
			(26, 0.006908333333333333)
			(27, 0.006355)
			(28, 0.005923333333333333)
			(29, 0.0054733333333333335)
			(30, 0.005098333333333333)
			(31, 0.004758333333333333)
			(32, 0.004401666666666667)
			(33, 0.00412)
			(34, 0.0038283333333333333)
			(35, 0.0035666666666666668)
			(36, 0.0033566666666666667)
			(37, 0.0031316666666666667)
			(38, 0.002948333333333333)
			(39, 0.00279)
			(40, 0.002628333333333333)
			(41, 0.0025083333333333333)
			(42, 0.0023583333333333334)
			(43, 0.002225)
			(44, 0.0020633333333333333)
			(45, 0.0019433333333333334)
			(46, 0.001835)
			(47, 0.0017266666666666667)
			(48, 0.001665)
			(49, 0.0015833333333333333)
			(50, 0.0015)
			(51, 0.001425)
			(52, 0.0013483333333333333)
			(53, 0.0012733333333333333)
			(54, 0.0012033333333333334)
			(55, 0.0011516666666666667)
			(56, 0.00109)
			(57, 0.0010466666666666667)
			(58, 0.000995)
			(59, 0.0009633333333333333)
			(60, 0.0009266666666666667)
			(61, 0.0008983333333333333)
			(62, 0.0008516666666666667)
			(63, 0.00081)
			(64, 0.0007716666666666667)
			(65, 0.0007366666666666667)
			(66, 0.0007116666666666667)
			(67, 0.0006833333333333333)
			(68, 0.0006666666666666666)
			(69, 0.0006433333333333333)
			(70, 0.0006116666666666667)
			(71, 0.0005933333333333333)
			(72, 0.0005716666666666667)
			(73, 0.0005433333333333334)
			(74, 0.0005283333333333333)
			(75, 0.0005116666666666667)
			(76, 0.0004933333333333334)
			(77, 0.00048)
			(78, 0.00046333333333333334)
			(79, 0.00044833333333333335)
			(80, 0.00042666666666666667)
			(81, 0.0004133333333333333)
			(82, 0.00040166666666666665)
			(83, 0.00039)
			(84, 0.00038)
			(85, 0.000365)
			(86, 0.000355)
			(87, 0.00033)
			(88, 0.000325)
			(89, 0.00031)
			(90, 0.00030333333333333335)
			(91, 0.00028333333333333335)
			(92, 0.0002716666666666667)
			(93, 0.00026)
			(94, 0.000255)
			(95, 0.00024333333333333333)
			(96, 0.00023)
			(97, 0.00022333333333333333)
			(98, 0.00022333333333333333)
			(99, 0.00021166666666666667)
			(100, 0.00021)
		};
		\addlegendentry{Jordan-2}
		
		\addplot[
		dashed, black
		]
		coordinates {
			(5, 0.2)
			(100, 0.01)
		};
		
		\addplot[
		dashed, black
		]
		coordinates {
			(5, 4/5^2)
			(100, 4/100^2)
		};

	\end{loglogaxis}

    \hspace{0.33\textwidth}
		
		\begin{semilogyaxis}[
			xlabel={$S$},
			ylabel={},
            xlabel style={yshift=8pt},
			title={Lower tail: Jordan},
            title style={yshift=-5pt},
			grid=both,
			major grid style={gray},
			minor grid style={gray!20},
			xtick distance=2,
			ytick distance=10,
			width=0.35\textwidth,
			height=5.5cm,
			]

			\addplot[
			thick, red
			]
			coordinates {
				(2, 1.0)
				(3, 0.14998195)
				(4, 0.0232965)
				(5, 0.00347325)
				(6, 0.00031255)
				(7, 3.05e-5)
				(8, 3.95e-6)
				(9, 5.0e-7)
			};

		\end{semilogyaxis}

		\hspace{0.33\textwidth}
		
		\begin{loglogaxis}[
			xlabel={$S$},
            xlabel style={yshift=8pt},
			ylabel={},
			title={Lower tail: Jordan-2},
            title style={yshift=-5pt},
			grid=both,
			major grid style={gray},
			minor grid style={gray!20},
			xtick distance=10^0.5,
			ytick distance=10,
			width=0.35\textwidth,
			height=5.5cm,
			]

			\addplot[
			thick, blue
			]
			coordinates {
				(3, 0.10121575)
				(4, 0.03584185)
				(5, 0.0185427)
				(6, 0.0113976)
				(7, 0.00770945)
				(8, 0.0055491)
				(9, 0.0041935)
				(10, 0.0032759)
				(11, 0.00262995)
				(12, 0.00215515)
				(13, 0.00179665)
				(14, 0.0015202)
				(15, 0.0013023)
				(16, 0.00112605)
				(17, 0.00098065)
				(18, 0.0008605)
				(19, 0.00076335)
				(20, 0.00067995)
				(21, 0.00060795)
				(22, 0.0005462)
				(23, 0.00048955)
				(24, 0.0004376)
				(25, 0.00039105)
				(26, 0.00035025)
				(27, 0.0003111)
				(28, 0.000275)
				(29, 0.00024205)
				(30, 0.0002096)
				(31, 0.00017985)
				(32, 0.00015205)
				(33, 0.000129)
				(34, 0.00010395)
				(35, 8.435e-5)
				(36, 6.76e-5)
				(37, 5.595e-5)
				(38, 4.655e-5)
				(39, 3.935e-5)
				(40, 3.405e-5)
				(41, 2.98e-5)
				(42, 2.53e-5)
				(43, 2.225e-5)
				(44, 1.97e-5)
				(45, 1.745e-5)
				(46, 1.47e-5)
				(47, 1.24e-5)
				(48, 1.05e-5)
				(49, 8.95e-6)
				(50, 7.05e-6)
				(51, 6.2e-6)
				(52, 4.45e-6)
				(53, 3.8e-6)
				(54, 3.25e-6)
				(55, 2.9e-6)
				(56, 2.65e-6)
				(57, 2.2e-6)
				(58, 1.9e-6)
				(59, 1.5e-6)
				(60, 1.25e-6)
				(61, 1.15e-6)
				(62, 1.0e-6)
				(63, 6.5e-7)
				(64, 5.0e-7)
			};

			\addplot[
			dashed, black
			]
			coordinates {
				(3, 1/3^2)
				(64, 1/64^2)
			};
		\end{loglogaxis}
		\end{tikzpicture}

    }
	\vspace{-0.2cm}
	\caption{
    The left panel shows the empirical probabilities of $\{\hs(v)\ge Sv\}$ for the Jordan and Jordan-2 orderings. The observed slopes are consistent with expected asymptotics ($S^{-1}$ for Jordan vs.\ $S^{-2\pm o(1)}$ for Jordan-2). The dashed lines are the lines $1/S$ and $4/S^2$, respectively. The middle panel shows the lower tail (empirical probability of $\{\hs_J(v) \leq v/S\}$) for the Jordan ordering on a semi-logarithmic scale, illustrating the exponential decay in Theorem \ref{thm:jordan}. The right panel illustrates the polynomial decay for the lower tail of the Jordan-2 ordering, with the dashed line corresponding to $1/S^2$.  
      For small and intermediate values of $S$, the decay is close to $1/S^2$, while for larger $S$, the probability drops faster, in line with the range $S\lesssim\sqrt{v}$ in the lower bound in Theorem \ref{thm:main}.
    }
	\label{fig:Simulation:underestimation}
    \vspace{-0.3cm}
\end{figure}
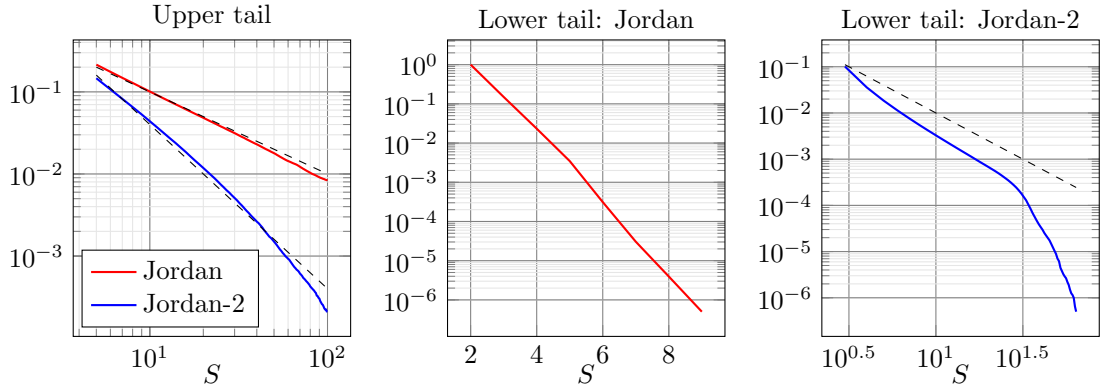

\medskip
\noindent\textbf{Heuristics and comparison.\ }
The different tail behaviors of the two orderings can be understood from the growth of the fringe trees. For the Jordan ordering, the upper tail is governed by the event that vertex $v$ remains a leaf until time of order $Sv$. This event has probability of order $1/S$; on it, the fringe tree of $v$ looks like the fringe tree of a vertex that arrived around time $Sv$, leading to an overestimation by a factor of order $S$. By contrast, underestimation requires the fringe tree of $v$ to be atypically large, of order $Sn/v$, an event exponentially unlikely in $S$.

\smallskip 
For the Jordan-2 ordering, overestimation requires not only the fringe tree of $v$, but also the fringe tree of its parent, to be small. Heuristically, 
\[ 
\Big\{\hs_2(v)\ge Sv\Big\} \approx \bigg\{
\mathrm{Fr}_n(v)\cdot\mathrm{Fr}_n(\mathrm{pa}(v))
\lesssim \left(\frac{n}{Sv}\right)^2
\bigg\} = \bigg\{\mathrm{Fr}_n(v)\frac{v}{n} \cdot \mathrm{Fr}_n(\mathrm{pa}(v))\frac{v}{n} \lesssim \frac{1}{S^2} \bigg\},
\]
where $\mathrm{Fr}_n(u) = |(T_n,1)_{u\downarrow}|$ denotes the size of the fringe tree of $u$.
The probability of this event is of order $(\log S)/S^2$. The lower tail has a different mechanism: $v$ can be estimated as much older than it is when its parent arrived much earlier than $v$. If the parent arrived before time $v/S^2$, an event of probability of order $1/S^2$, then the product of the two fringe tree sizes is typically larger by a factor $S^2$. This causes the Jordan-2 ordering to rank $v$ as if it had arrived around time $v/S$. 
We see that, compared to the ordering $\hs_J$ based on the Jordan centrality, the ordering $\hs_2$  based on the Jordan-2 centrality improves the upper tail from order $1/S$ to essentially $1/S^2$. However, this gain comes at the cost of a thicker lower tail, which is exponentially small for the Jordan ordering, but polynomially small for the ordering based on the Jordan-2 centrality.

\subsection{Risk}\label{sec:risk}
The pointwise tail bounds established in Proposition \ref{prop:lower_bound} and Theorems~\ref{thm:main} and~\ref{thm:jordan} readily yield bounds on the risk $R_\alpha$ defined in \eqref{eq:RiskDef2}, which measures the global performance of an estimator~\cite{briend2024estimating}. 
In particular, the Jordan-2 ordering attains the optimal order of the risk for every $\alpha>0$. 
\begin{corollary}[Lower bounds for arbitrary estimators]\label{cor:optimal}
    For each $\alpha>0$, there exists a constant $c_\alpha$ such that for any label-invariant estimator $\hs$, and any $n\ge 2$,
    \begin{equation*}
        R_{\alpha}(\hs) \ge 
        \begin{dcases}
            c_\alpha n^{2-\alpha}, &\text{if }\alpha\in(0,2),\\
            c_\alpha\log n,&\text{if }\alpha=2, \\
            c_\alpha,&\text{if }\alpha>2.
        \end{dcases}
    \end{equation*}
\end{corollary}
Corollary \ref{cor:optimal} improves the lower bounds obtained in \cite[Theorem 1]{briend2024estimating} for the risk $R_\alpha$ when $\alpha=2$. We write $f_\alpha(n)\asymp g(n)$ if  $f_\alpha(n)/g(n)\in [\tfrac{1}{C_\alpha}, C_\alpha]$ for  some constant $C_\alpha>0$ and all $n\ge 1$.

\begin{corollary}[Jordan-$2$ and Jordan ordering]\label{cor:new method}
   For each $\alpha>0$, 
    \begin{equation*}
        R_\alpha(\hs_2) \asymp 
        \begin{dcases}
            n^{2-\alpha},& \text{if }\alpha\in(0,2),\\
            \log n,&\text{if }\alpha=2,\\
            1,&\text{if }\alpha>2,
        \end{dcases}
        \qquad \qquad 
        R_\alpha(\hs_J) \asymp
        \begin{dcases}
            n^{2-\alpha},& \text{if }\alpha\in(0,2),\\
            (\log n)^2,&\text{if }\alpha=2,\\
            \log n,&\text{if }\alpha>2.
        \end{dcases}
    \end{equation*}
\end{corollary}
\begin{proof}[Proofs] We give the proof of the upper bound on {$R_\alpha(\hs_2)$} in Corollary \ref{cor:new method}; the other bounds follow analogously.
    The  risk defined in \eqref{eq:RiskDef2} can be rewritten as follows:
    \begin{align}
    \notag R_\alpha(\hs_2)=\sum_{u=1}^nu^{-\alpha}\E\big[|\hs_2(u)-u|\big]&= \sum_{u=1}^n u^{-\alpha}\int_0^n\p\left(|\hs_2(u)-u|\ge x\right)\ \rd x\\
    & \label{eq:coro5.1insert} =  \sum_{u=1}^n u^{1-\alpha}\int_0^{n/u}\p\left(|\hs_2(u)-u|\ge u y\right)\ \rd y.
    \end{align}
     When $y > 1$, $|\hs_2(u)-u| \geq u y$  implies that $\hs_2(u) \geq uy$, since $\hs_2(u) \geq 0$. Thus, we get that 
     \begin{align*}
         \int_0^{n/u}\p\left(|\hs_2(u)-u|\ge u y\right)\ \rd y \leq 1 + \int_1^{\infty}\p\left(\hs_2(u)\ge u y\right)\ \rd y
         \overset{\eqref{jordan2:overest_combined}}{\leq }
         1 + \int_1^{\infty} C \frac{1+\log(y)}{y^2} \rd y \eqqcolon C^\star,
     \end{align*}
     where $C^\star>0$ is a constant not depending on $u$. Inserting this bound into \eqref{eq:coro5.1insert} finishes the proof of the upper bound. The other bounds in Corollaries~\ref{cor:optimal} and \ref{cor:new method} can be obtained analogously.
\end{proof}

For $\alpha \geq 2$, Corollary \ref{cor:new method} improves the risk upper bound on the Jordan ordering of $(\log n)^4$ from   \cite[Theorem~4]{briend2024estimating}. It was already noted below  \cite[Theorem 4]{briend2024estimating} that the risk of the Jordan ordering should be at least of order $\log n$ for all $\alpha\ge 2$. This bound is sharp for $\alpha>2$; for $\alpha= 2$ the risk is of order $(\log n)^2$ instead. When $\alpha<2$, the bounds are of the optimal order,  as also noted in \cite{briend2024estimating}.

 \smallskip 
The proof of Corollary~\ref{cor:new method} highlights a limitation of the risk $R_\alpha$. Indeed,  controlling $R_\alpha$ only requires bounds on the \emph{upper tail} of the pointwise estimation error: if an estimator $\hs$ satisfies $\p\left(\hs(v)\ge Sv \right)\le f(S)$ uniformly in $v$ and $n$ for an integrable function $f$, the risk is always of the optimal order. 
Thus, once the upper tail is integrable, further improvements do not affect the leading-order asymptotics of the risk, while the \emph{lower tail} is invisible altogether. In particular, estimators with qualitatively different lower-tail behavior may nevertheless have the same asymptotic risk. This motivates the pointwise perspective developed throughout the paper.

\subsection{Further directions}\label{sec:discussion}
The tradeoff between upper and lower tails described above raises the question whether there are estimators that can strictly improve upon the Jordan-2 ordering in one tail without sacrificing performance in the other, or whether the observed phenomenon is inevitable.
\begin{problem}\label{pr:outperform-jordan2}
    Find a label-invariant estimator $\hs_\ast$ such that there exist constants $a,b > 2$, such that for all $\varepsilon>0$ there exists a constant $C>0$ such that for all $n\ge 1$, $v\in[n]$, and $S\ge 1$,
    \[
    \p\left(\hs_\ast(v)\le v/S\right)\le C S^{-a+\varepsilon}, \qquad \text{and }\qquad \p\left(\hs_\ast(v)\ge vS\right)\le CS^{-b+\varepsilon}.
    \]
\end{problem}
We have no evidence that such estimators exist; proving non-existence would be equally interesting. 
Such an estimator would be genuinely better at the pointwise level, even though it would not improve the risk $R_\alpha$ (defined in~\eqref{eq:RiskDef2}) by more than a constant factor, since $R_\alpha$ is insensitive, to first order, to the lower tail. Capturing this improvement would require a risk functional that distinguishes upper- and lower-tail errors in a more refined way.

The Jordan ordering fails the desired upper-tail bound by Theorem~\ref{thm:jordan}. More generally, orderings that incorporate more information about the ancestral line of a vertex may improve the upper tail, but we expect them to do so at the expense of the lower tail. This motivates the following extension of the Jordan-$2$ construction.
Adopting the convention $v^\ssup{0}=v$, define recursively
\[
v^\ssup{j}:=\argmin_{u\in \rt: u\sim v^\ssup{j-1}} \big|(\rt,\, u)_{v^\ssup{j-1}\downarrow}\big|, \qquad j\ge 1.
\]
Thus, $v^\ssup{j}$ estimates the $j$-th ancestor of $v$, with ties resolved analogously to the case $j=1$ in Definition \ref{def:jordan2}. For $k\ge 3$, we define the Jordan-$k$ centrality by 
\[
\phi^{\ssup{k}}_\rt(v):= \left( \prod_{j=0}^{k-1} \big|(\rt, v^{\ssup{j+1}})_{v^{\ssup{j}}\downarrow}\big| \right) \vee \left( \phi^{\ssup{k-1}}_\rt(v) \right)^{\frac{k}{k-1}}.
\]
The Jordan-$k$ ordering is obtained by ordering vertices decreasingly with respect to $\phi^{\ssup{k}}$. Note that $k=1$ recovers the standard Jordan ordering introduced by $\phi$ in \eqref{eq:neg-jordan}.

We expect that increasing $k$ makes the estimator less likely to overestimate the arrival time, with upper tails of order $S^{-k\pm o(1)}$, but more likely to underestimate it, with lower tails of order $S^{-k/(k-1) \pm o(1)}$. The mechanism behind the lower tail is that young vertices attached close to the root may be estimated as much older than they actually are, an effect that becomes stronger as more ancestral information is incorporated. In this sense, Jordan-2 appears to be the most balanced member of the family, as both tails exhibit quadratic decay up to logarithmic factors.

\smallskip 
A separate question concerns the far lower tail. The lower-tail estimates in Theorems~\ref{thm:main} and~\ref{thm:jordan} describe a polynomial regime and an exponential regime, respectively, valid for small and intermediate values of $S$. For the Jordan-$2$ ordering, this polynomial decay cannot persist for $S\gg\sqrt v$. Indeed, Remark~\ref{rem:stretched-exp} shows that $\p\big(\hs_2(v) \leq v/S\big)$ is already stretched-exponentially small when $S\gg \sqrt{v}$.
Figure~\ref{fig:Simulation:underestimation} provides empirical evidence for this transition for large values of $S$. We conjecture that the lower tail of the Jordan-2 ordering undergoes a  transition from polynomial decay to stretched-exponential decay, and that a similar phenomenon occurs for the Jordan ordering. 

\medskip\noindent\textbf{Organisation.\ }The core technical contribution of the paper is the proof of Theorem \ref{thm:main} on the Jordan-2 centrality in Section \ref{sec:NBjordan}. Along the way, we prove almost all prerequisite lemmas needed to prove Theorem \ref{thm:jordan} on the Jordan centrality in Section \ref{sec:Jordan}, following a similar strategy. Afterwards, in Section \ref{sec:lower bounds}, we prove Proposition \ref{prop:lower_bound} on arbitrary estimators. Appendix \ref{sec:appendix} contains proofs of technical lemmas.

\section{Performance of the Jordan-$2$ ordering}\label{sec:NBjordan}
In this section, we study the Jordan-$2$ centrality measure $\phi^{\ssup{2}}$ and its associated ordering $\hs_2$, proving the four inequalities of Theorem \ref{thm:main} in Lemmas \ref{lem:Jordan2_uppertail_UB}, 
\ref{lem:Jordan2_upper_tail_lower_bound},
 \ref{lem:Jordan2_lower_tail}, and \ref{lem:Jordan2_lower_tail_lower_bound}, respectively.
We start with an outline.

\medskip 
\noindent \textbf{Outline: upper tail -- upper bound.\ }
Since $\hs_2(\cdot)$ ranks vertices in decreasing order of Jordan-2 centrality, the event  $\{\hs_2(v)\ge Sv\}$ implies that at least $Sv$ vertices have Jordan-2 centrality at least  $\phi_n^\ssup{2}(v)$. 
The proof splits the event $\{\hs_2(v)\ge Sv\}$ into two possibilities: the Jordan-2 centrality of the fixed vertex $v$ is unusually small, or there are unusually many vertices whose Jordan-2 centrality exceeds an explicit threshold. We use that, for some constant $C>0$,
\begin{equation}\label{eq:hs2 split}
    \big\{\hs_2(v)\ge Sv\big\}\subseteq  \Big\{\phi_n^\ssup{2}(v)\le \left(C\tfrac{n}{Sv}\right)^2\Big\}\,\cup\, \Big\{\big|\big\{u\in[n]: \phi_n^\ssup{2}(u) \geq \left(C\tfrac{n}{Sv}\right)^2\big\}\big| \ge Sv\Big\}.
\end{equation}
Our cutoff value for $\phi_n^\ssup{2}(v)$ is motivated by its typical scale: recall from \eqref{eq:NNBJordanDef} that 
\begin{equation}\label{eq:outline-before-intuition-phi2}
\phi_n^\ssup{2}(v)=\phi_n(v)\left(\phi_n(v)\vee\phi_n(v^\ssup{1})\right),
\end{equation}
where $\phi_n(v)$ and $\phi_n(v^\ssup{1})$ estimate the sizes of the related subtrees $(T_n, 1)_{v\downarrow}$ and $(T_n, 1)_{\pa(v)\downarrow}$. These subtree sizes follow a beta-binomial distribution and are typically of order $n/v$ each.

\smallskip

The probability of the first event in the union in \eqref{eq:hs2 split} gives the dominant contribution to the upper bound on $\p(\hs(v)\ge Sv)$. This event depends on the local structure around $v$, i.e.,  the fringe trees of $v$ and $\mathrm{pa}(v)$.
We derive a deterministic lower bound on $\phi_n^\ssup{2}(v)$ as the product of two subtree sizes of vertices that arrive before time $v$, but where we only count descendants arriving after time $v$. These restricted subtree sizes follow a beta-binomial distribution, and because the subtrees involved are disjoint, their sizes are negatively associated. A tail bound for the product of two negatively associated beta-binomial variables then yields the decay of order  $(\log S)/S^2$.

\smallskip 
The second event in the union in \eqref{eq:hs2 split}---that many vertices have Jordan-2 centrality at least $\big(\tfrac{n}{Sv}\big)^2$---requires information about subtree sizes throughout the tree, and is substantially more difficult.
To control it, we use a construction of the Ulam--Harris embedding of the tree. In this construction---called fringe-tree splitting---the root begins with mass $n$, and its children receive random integer masses that sum up to $n-1$; the procedure continues recursively for every vertex with mass at least $2$. This construction provides enough independence to count vertices with large Jordan-2 centrality and forms the other main technical ingredient of the proof.

\medskip 
\noindent\textbf{Outline: upper tail -- lower bound.\ }
For $\{\hs_2(v)\geq Sv\}$ to hold, it suffices that at least $Sv$ vertices have Jordan-2 centrality exceeding $\phi_n^\ssup{2}(v)$. Using a similar cutoff for $\phi_n^\ssup{2}(v)$ to \eqref{eq:hs2 split},
$$ 
\big\{\hs_2(v)\geq Sv\big\}\, \supseteq\,
\Big\{\phi_n^\ssup{2}(v)\leq \left(\tfrac{1}{64S}\tfrac{n}{v}\right)^2\Big\} 
\, \setminus \,
\Big\{\big|\big\{u\in[n]: \phi_n^\ssup{2}(u) > \left(\tfrac{1}{64S}\tfrac{n}{v}\right)^2\big\}\big|< Sv\Big\}.
$$
We establish a lower bound on the first probability on the right-hand side analogously to the reasoning below \eqref{eq:outline-before-intuition-phi2}. This reduces lower bounding $\p\big(\hs_2(v)\geq Sv\big)$ to establishing an upper bound on the probability of the second event on the right-hand side.

To have few vertices with large Jordan-2 centrality, it is required that few vertices in the tree have a fringe tree of size at least $n/(64Sv)$. 
Since the tree $T_n$ typically contains order $Sv$ vertices whose fringe trees are at least $n/(Sv)$,
having substantially fewer such vertices is unlikely.

\medskip 
\noindent\textbf{Outline: lower tail -- upper and lower bounds.\ }
For both the upper and lower bound on the probability that the Jordan-2 estimator is unlikely small, we use similar decomposition as for the upper tail. The main difference is that we instead need to control $\p\big(\phi_n^\ssup{2}(v)\ge (Sn/v)^2\big)$, which is dominated by the probability that the parent of $v$ arrived before time $v/S^2$.

\medskip 
\noindent\textbf{Organisation of the section.\ }
Thus, to prove the four inequalities of Theorem \ref{thm:main}, we always decompose the relevant event into a local event on the Jordan-2 centrality of $v$, and a global event ensuring that the Jordan-2 centrality throughout the tree behaves regularly. 
The probability of the local event provides the dominant contribution for all of these tail bounds, whereas the probability of the global event is negligible.
In Section~\ref{subsec:fringe}, we introduce the fringe-splitting construction and develop the counting lemmas used to control the number of vertices with Jordan-$2$ centrality at least $r>0$. The remaining subsections complete the proofs of the four inequalities in Theorem \ref{thm:main}. That is, in Section \ref{sec:Upper_tail_UB_J2}, we finish the proof of the upper bound on the upper tail by combining these results. In Section \ref{sec:Upper_tail_LB_J2}, we prove a lower bound on the upper tail. In Section \ref{sec:Lower_tail_UB-J2}, we analyze the lower tail relying on the auxiliary lemmas developed before.

\subsection{The Jordan-2 centrality of a fixed vertex}\label{sec:jordan2-v}
In this section, we analyze the Jordan-2 centrality of a fixed vertex $v$. The probability that it is unlikely small or large forms the main contribution to the bounds in Theorem \ref{thm:main}. The lower tail on the centrality corresponds to an upper tail on the ordering and vice versa. We postpone the proofs of technical (relatively standard) lemmas to the appendix. 

\subsubsection{Lower tail}\label{subsec:jordan2c-lower}
The main goal of this section is to prove the following lemma.

\begin{lemma}[Lower tail of the Jordan-2 centrality]\label{lem:boosted small}
    There exists a constant $C>0$ such that for any $v\in [n]$ and for any $t \geq 1$, the random variable $\phi_n^{\ssup{2}}(v)$ satisfies 
    \begin{equation}\label{eq:lem2.1 upper bound}
        \p \left( \phi_n^{\ssup{2}}(v) \leq \left( \frac{n}{t v} \right)^2 \right) \leq C \frac{1+\log t}{t^2},
    \end{equation}
    Furthermore, for every $\eps > 0$, there exists a constant $c>0$ so that for any $n\geq 3$, $v\in [n/4]$ and $t\in\big(\sqrt{2},\tfrac{n}{v \sqrt{3}} \wedge (n/v)^{1-\eps}\big)$,
    \begin{equation}\label{eq:lem2.1 lower bound}
        \p \left( \phi_n^{\ssup{2}}(v)
    \leq
    \left( \frac{n}{t v} \right)^2 \right) \geq c \frac{\log t}{t^2}.
    \end{equation}
    Moreover, there exists a constant $c^\prime>0$ so that for any $n\geq 3$, $v\in [n/4]$ and $t\in\big(1, \frac{1}{\sqrt{3}} (n/v) \big)$,
    \begin{equation}\label{eq:lem2.1 lower bound version2}
        \p \left( \phi_n^{\ssup{2}}(v)
    \leq
    \left( \frac{n}{t v} \right)^2 \right) \geq c \frac{1}{t^2}.
    \end{equation}
\end{lemma}
We start by proving the upper bound. The key step is to lower bound the Jordan-2 centrality $\phi_n^\ssup{2}(v)$ by the product of suitable subtree sizes. We first introduce fringe trees witnessed after time $m$ and then state some of their basic distributional properties. These sizes admit a simple P\'olya-urn description and form a family of negatively associated random variables, which will be crucial for controlling the lower tail of $\phi_n^\ssup{2}(v)$.

\begin{definition}[Fringe trees]\label{def:fringe}
    For a vertex $v\in [n]$, we call the subtree of $v$ in the rooted tree $(T_n, 1)$ the fringe tree attached to $v$, or the fringe tree of $v$. For its size, we write 
\begin{equation}
    \mathrm{Fr}_n(v):=|(T_n, 1)_{v\downarrow}|=|(T_n, \pa(v))_{v\downarrow}|,
\end{equation}
    where the last equality holds for all $v \geq 2$.
    Furthermore, given $m \in \N$ and $v \in [n]$, we define the fringe tree of $v$ witnessed after time $m$ as the set of vertices in $(T_n, 1)_{v\downarrow}$ whose unique path to $v$ in $T_n$ has no edges  that arrived before time $m+1$.  We write $\mathrm{Fr}_n^{>m}(v)$ for its size. 
\end{definition}
\begin{figure}[t]
    \begin{subfigure}{0.47 \textwidth}
    \centering
        \begin{tikzpicture}[scale=1]

    \node[draw, circle, minimum size=4mm, inner sep = 0 pt] (1) at (0,0) {$1$};
    \node[draw, circle, minimum size=4mm, inner sep = 0 pt, blue, thick] (2) at (-1,-1) {$2$};
    \node[draw, circle, minimum size=4mm, inner sep = 0 pt, blue, thick] (3) at (-2,-2) {$3$};
    \node[draw, circle, minimum size=4mm, inner sep = 0 pt] (4) at (0,-1) {$4$};
    \node[draw, circle, minimum size=4mm, inner sep = 0 pt, blue, thick] (5) at (-1,-2) {$5$};
    \node[draw, circle, minimum size=4mm, inner sep = 0 pt, blue, thick] (6) at (-2.5,-3) {$6$};
    \node[draw, circle, minimum size=4mm, inner sep = 0 pt, blue, thick] (7) at (-1.5,-3) {$7$};
    \node[draw, circle, minimum size=4mm, inner sep = 0 pt] (8) at (0,-2) {$8$};
    \node[draw, circle, minimum size=4mm, inner sep = 0 pt, blue, thick] (9) at (-0.5,-3) {$9$};
    \node[draw, circle, minimum size=4mm, inner sep = 0 pt] (10) at (1,-1) {$10$};
    \node[draw, circle, minimum size=4mm, inner sep = 0 pt] (11) at (1,-2) {$11$};

    \draw (2) edge (1);
    \draw (3) edge[blue, thick] (2);
    \draw (4) edge (1);
    \draw (5) edge[blue, thick] (2);
    \draw (6) edge[blue, thick] (3);
    \draw (7) edge[blue, thick] (5);
    \draw (8) edge (4);
    \draw (9) edge[blue, thick] (5);
    \draw (10) edge (1);
    \draw (11) edge (10);

\end{tikzpicture}
    \end{subfigure}
    \begin{subfigure}{0.47 \textwidth}
    \centering
        \begin{tikzpicture}[scale=1]

    \node[draw, circle, minimum size=4mm, inner sep = 0 pt] (1) at (0,0) {$1$};
    \node[draw, circle, minimum size=4mm, inner sep = 0 pt, red, thick] (2) at (-1,-1) {$2$};
    \node[draw, circle, minimum size=4mm, inner sep = 0 pt] (3) at (-2,-2) {$3$};
    \node[draw, circle, minimum size=4mm, inner sep = 0 pt] (4) at (0,-1) {$4$};
    \node[draw, circle, minimum size=4mm, inner sep = 0 pt, red, thick] (5) at (-1,-2) {$5$};
    \node[draw, circle, minimum size=4mm, inner sep = 0 pt] (6) at (-2.5,-3) {$6$};
    \node[draw, circle, minimum size=4mm, inner sep = 0 pt, red, thick] (7) at (-1.5,-3) {$7$};
    \node[draw, circle, minimum size=4mm, inner sep = 0 pt] (8) at (0,-2) {$8$};
    \node[draw, circle, minimum size=4mm, inner sep = 0 pt, red, thick] (9) at (-0.5,-3) {$9$};
    \node[draw, circle, minimum size=4mm, inner sep = 0 pt] (10) at (1,-1) {$10$};
    \node[draw, circle, minimum size=4mm, inner sep = 0 pt] (11) at (1,-2) {$11$};

    \draw (2) edge (1);
    \draw (3) edge (2);
    \draw (4) edge (1);
    \draw (5) edge[red, thick] (2);
    \draw (6) edge (3);
    \draw (7) edge[red, thick] (5);
    \draw (8) edge (4);
    \draw (9) edge[red, thick] (5);
    \draw (10) edge (1);
    \draw (11) edge (10);

\end{tikzpicture}
    \end{subfigure}
    \caption{The fringe tree of $2$ (blue, left panel), and the fringe tree  of $2$ witnessed after time $m\in\{3,4\}$ (red, right panel). We have $\mathrm{Fr}_{11}(2)=6=\mathrm{Fr}_{11}^{>2}(2)$ and $\mathrm{Fr}_{11}^{>3}(2)=\mathrm{Fr}_{11}^{>4}(2)=4$. The fringe tree of $2$ witnessed after time $m\ge 5$ consists only of vertex $2$ itself. 
    }
    \label{fig:fringe witness}
\end{figure}
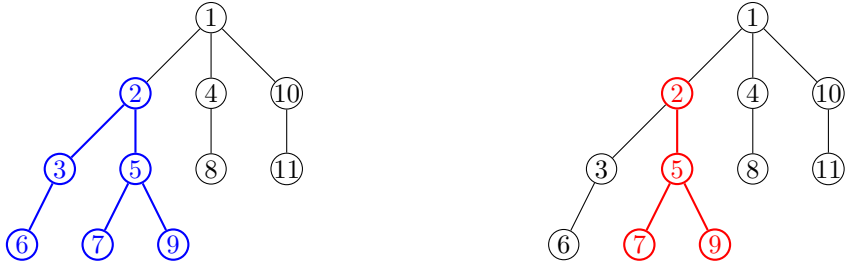
For example, if $m \leq v$, then $\mathrm{Fr}_n^{> m}(v) = \mathrm{Fr}_n(v)$. On the other hand, if, say, $v=1$ and $m=2$, then $\mathrm{Fr}_n^{>2}(1)$ counts the size of the fringe tree attached to $1$ that can be witnessed after time $2$. That is, we do not include the fringe tree of $2$ and thus get that $\mathrm{Fr}_n^{> 2}(1) = \mathrm{Fr}_n(1) - \mathrm{Fr}_n(2)$. Figure~\ref{fig:fringe witness} visualises the fringe tree and the fringe tree witnessed after time $m$.

The size of the fringe tree is directly related to the Jordan centrality, since by \eqref{eq:jordan-upper} and \eqref{eq:upper-bound-jordan2}
\begin{equation*}
    \phi_n(v) \leq \mathrm{Fr}_n(v), \quad \text{ and for $v\ge 2$,} \quad \phi_n^{\ssup{2}}(v) \leq \mathrm{Fr}_n(v) \mathrm{Fr}_n(\mathrm{pa}(v)).
\end{equation*}
Further, if $\mathrm{Fr}_n(v) \leq \frac{n}{2}$, we also have
\begin{equation*}
    \phi_n(v) = \mathrm{Fr}_n(v) \quad \text{ and } \quad \phi_n^{\ssup{2}}(v) \geq (\phi_n(v))^2 = (\mathrm{Fr}_n(v))^2 . 
\end{equation*}

The witnessed fringe-tree sizes $\left(\mathrm{Fr}_n^{>m}(1),\ldots, \mathrm{Fr}_n^{>m}(m)\right)$ behave like the counts of colors in a Pólya urn started with $m$ colors. As a consequence, they are exchangeable, negatively associated, and have explicit one-dimensional tail bounds.
We give the proof of this lemma in the Appendix \ref{sec:appendix} below.

\begin{restatable}[Joint fringe-tree-size distribution]{lemma}{betabinfringe}
    \label{lem:beta-bin-fringe} The following statements hold:
\begin{enumerate}[itemsep=0pt]
    \item The variables $\left(\mathrm{Fr}_n^{>m}(1),\ldots, \mathrm{Fr}_n^{>m}(m)\right)$ are exchangeable for all $n\ge m$.
    \item The variables $(\mathrm{Fr}_n^{> m}(u) :u\in[m])$ are negatively associated for all $n\ge m$.
    \item Let $n\in\N$, $m\in[n]$ and $M\ge 1$. For any $u\in[m]$, 
    \begin{equation}\label{eq:fringe-ub}
    \p\left(\mathrm{Fr}_n^{>m}(u)\le \frac{n}{Mm}\right) = \p\left(\mathrm{Fr}_n(m)\le \frac{n}{M m}\right) \le \frac{2}{M}.
    \end{equation}
    For $M\ge 4$, $m\in[2, n/M]$, and $u\in[m]$, 
    \begin{equation}\label{eq:fringe-lb}
    \p\left(\mathrm{Fr}_n^{>m}(u)\le \frac{n}{Mm}\right) = \p\left(\mathrm{Fr}_n(m)\le \frac{n}{Mm}\right) \ge \frac{1}{8M}.
    \end{equation}
    \end{enumerate}
\end{restatable}

\smallskip 
The next lemma gives a deterministic lower bound on $\phi_n(v)$ and $\phi_n^{\ssup{2}}(v)$ in terms of the sizes of (witnessed) fringe trees. Note that the lemma is deterministic and therefore holds for all recursive trees of size $n$.

\begin{lemma}[Deterministic lower bound on the Jordan-2 centrality]\label{lem:minima}
    Let $v\geq 2$. Then
    \begin{align}
    \notag\phi_n(v) &\ge \min\left(\mathrm{Fr}^{>v}_n(v), \mathrm{Fr}^{>v}_n(1)\right), \quad\text{and }\\
    \phi_n^\ssup{2}(v)&\ge \min\left(\mathrm{Fr}_n^{>v}(v)\mathrm{Fr}_n^{>v}(\pa(v)),  \,\,\,  \mathrm{Fr}_n^{>v}(v) \mathrm{Fr}_n^{>v}(1),   \,\,\,  \mathrm{Fr}_n^{>v}(v) \mathrm{Fr}_n^{>v}(2),   \,\,\,  \mathrm{Fr}_n^{>v}(1)\mathrm{Fr}_n^{>v}(2)\right).\label{eq:what happens above 2}
    \end{align}
    Further, for $v=2$, one has
    \begin{equation}\label{eq:what happens at 2}
        \phi_n^\ssup{2}(2) \ge \min\left(\mathrm{Fr}_n^{>3}(2)\mathrm{Fr}_n^{>3}(1),  \,\,\,  \mathrm{Fr}_n^{>3}(2)\mathrm{Fr}_n^{>3}(3),   \,\,\,  \mathrm{Fr}_n^{>3}(1)\mathrm{Fr}_n^{>3}(3)\right).
    \end{equation}
\end{lemma}
The special case for $v=2$ is needed, since we need the pairs of vertices in the products to be different. For $v=2$, the third term in \eqref{eq:what happens above 2} equals $\mathrm{Fr}_n(2)^2$.
\begin{proof}
 By the definition of $\phi_n(v)$ and $v^\ssup{1}$ in \eqref{eq:v1} and \eqref{eq:neg-jordan}, we have
 \begin{equation}\label{eq:phi-recap}
 \phi_n(v)=\min_{u\sim v}\big|(T_n, u)_{v\downarrow}\big|, \qquad v^\ssup{1}=\argmin_{u\sim v}\big|(T_n, u)_{v\downarrow}\big|, \qquad
 \phi_n(v^\ssup{1})=\min_{u\sim v^\ssup{1}}\big|(T_n, u)_{v^\ssup{1}\downarrow}\big|,
 \end{equation}
 where $(T_n,u)_{v\downarrow}$ denotes the subtree of $v$ in the tree $T_n$ when rooted at $u$. 

 \smallskip 
 For the lower bound on $\phi_n(v)$, we distinguish two cases, see also Figure \ref{fig:two cases}. 

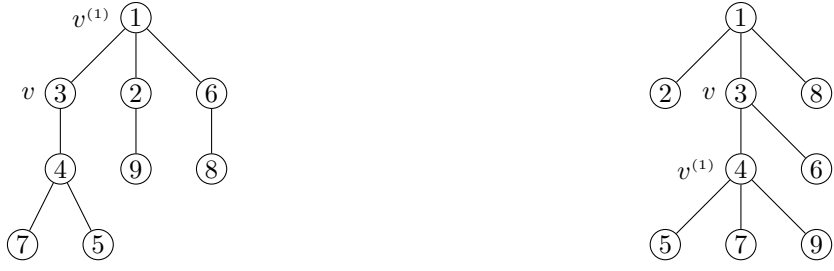
\begin{figure}[h]
		
		\centering
	
	\begin{subfigure}{0.45\textwidth}
\centering
\begin{tikzpicture}[scale=1]

    \node[draw, circle, minimum size=4mm, inner sep = 0 pt, label=left:{$v^\ssup{1}$}] (1) at (0,-0) {$1$};
    \node[draw, circle, minimum size=4mm, inner sep = 0 pt] (2) at (0,-1) {$2$};
    \node[draw, circle, minimum size=4mm, inner sep = 0 pt, label=left:{$v$}] (3) at (-1,-1) {$3$};
    \node[draw, circle, minimum size=4mm, inner sep = 0 pt] (4) at (-1,-2) {$4$};
    \node[draw, circle, minimum size=4mm, inner sep = 0 pt] (5) at (-0.5,-3) {$5$};
    \node[draw, circle, minimum size=4mm, inner sep = 0 pt] (6) at (1,-1) {$6$};
    \node[draw, circle, minimum size=4mm, inner sep = 0 pt] (7) at (-1.5,-3) {$7$};
    \node[draw, circle, minimum size=4mm, inner sep = 0 pt] (8) at (1,-2) {$8$};
    \node[draw, circle, minimum size=4mm, inner sep = 0 pt] (9) at (0,-2) {$9$};

    \draw (2) -- (1);
    \draw (3) -- (1);
    \draw (4) -- (3);
    \draw (5) -- (4);
    \draw (6) -- (1);
    \draw (7) -- (4);
    \draw (8) -- (6);
    \draw (9) -- (2);

\end{tikzpicture}
\end{subfigure}
\hfill
\begin{subfigure}{0.45\textwidth}
\centering
\begin{tikzpicture}[scale=1]

    \node[draw, circle, minimum size=4mm, inner sep = 0 pt] (1a) at (0,0) {$1$};
    \node[draw, circle, minimum size=4mm, inner sep = 0 pt] (2a) at (-1,-1) {$2$};
    \node[draw, circle, minimum size=4mm, inner sep = 0 pt, label=left:{$v$}] (3a) at (0,-1) {$3$};
    \node[draw, circle, minimum size=4mm, inner sep = 0 pt, label=left:{$v^\ssup{1}$}] (4a) at (0,-2) {$4$};
    \node[draw, circle, minimum size=4mm, inner sep = 0 pt] (5a) at (-1,-3) {$5$};
    \node[draw, circle, minimum size=4mm, inner sep = 0 pt] (6a) at (1,-2) {$6$};
    \node[draw, circle, minimum size=4mm, inner sep = 0 pt] (7a) at (0,-3) {$7$};
    \node[draw, circle, minimum size=4mm, inner sep = 0 pt] (8a) at (1,-1) {$8$};
    \node[draw, circle, minimum size=4mm, inner sep = 0 pt] (9a) at (1,-3) {$9$};

    \draw (2a) -- (1a);
    \draw (3a) -- (1a);
    \draw (4a) -- (3a);
    \draw (5a) -- (4a);
    \draw (6a) -- (3a);
    \draw (7a) -- (4a);
    \draw (8a) -- (1a);
    \draw (9a) -- (4a);

\end{tikzpicture}
\end{subfigure}
	
	\caption{For $v=3$, we have $v^{\ssup{1}}=1=\mathrm{pa}(v)$ for the left tree, but $v^{\ssup{1}}=c=4 \neq \mathrm{pa}(v)$ for the right tree.}\label{fig:two cases}
\end{figure}
 
If $\pa(v)$ attains the minimum in the first equation in \eqref{eq:phi-recap}, then $\phi_n(v)=\big|(T_n, \mathrm{pa}(v) )_{v\downarrow}\big|=\mathrm{Fr}_n(v)=\mathrm{Fr}^{>v}_n(v)$. Assume next that the minimum is instead attained at a child $c$ of $v$, so $\phi_n(v)=|(T_n, c)_{v\downarrow}|$. We claim that the tree $(T_n, c)_{v\downarrow}$ contains the fringe tree of vertex $1$ witnessed after time $v$. Indeed, if we pick a vertex $w$ in this witnessed fringe tree, then the unique path from $w$ to $c$ has to contain the edge $\{v, \pa(v)\}$ which arrived before time $v+1$, and thus $w \in (T_n, c)_{v\downarrow}$. Thus, we get that $\phi_n(v) = |(T_n, c)_{v\downarrow}| \ge  \mathrm{Fr}_n^{>v}(1)$ in this case. 
 
 As a result, we see that in either case we have $\phi_n(v)\ge \min(\mathrm{Fr}_n^{>v}(v), \mathrm{Fr}_n^{>v}(1))$ when $v\ge 2$, proving the first inequality of Lemma \ref{lem:minima}. \\


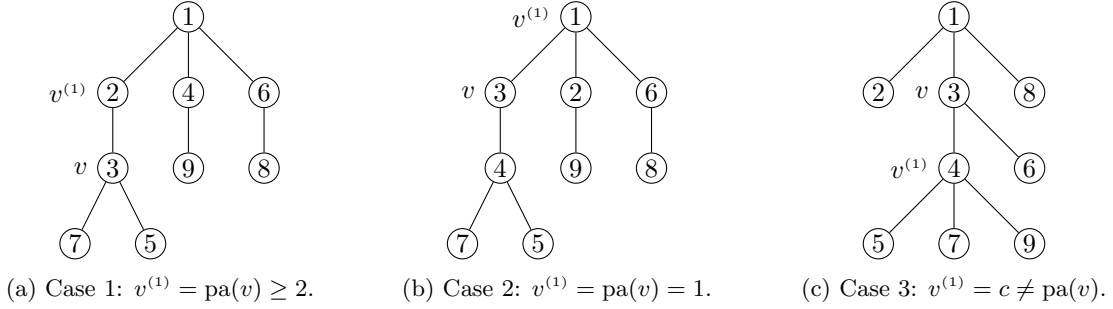
\begin{figure}[h]
\centering

\begin{subfigure}{0.3\textwidth}
\centering
\begin{tikzpicture}[scale=1]

    \node[draw, circle, minimum size=4mm, inner sep = 0 pt] (1) at (0,0) {$1$};
    \node[draw, circle, minimum size=4mm, inner sep = 0 pt] (4) at (0,-1) {$4$};
    \node[draw, circle, minimum size=4mm, inner sep = 0 pt, label=left:{$v^\ssup{1}$}] (2) at (-1,-1) {$2$};
    \node[draw, circle, minimum size=4mm, inner sep = 0 pt, label=left:{$v$}] (3) at (-1,-2) {$3$};
    \node[draw, circle, minimum size=4mm, inner sep = 0 pt] (5) at (-0.5,-3) {$5$};
    \node[draw, circle, minimum size=4mm, inner sep = 0 pt] (6) at (1,-1) {$6$};
    \node[draw, circle, minimum size=4mm, inner sep = 0 pt] (7) at (-1.5,-3) {$7$};
    \node[draw, circle, minimum size=4mm, inner sep = 0 pt] (8) at (1,-2) {$8$};
    \node[draw, circle, minimum size=4mm, inner sep = 0 pt] (9) at (0,-2) {$9$};

    \draw (2) -- (1);
    \draw (3) -- (2);
    \draw (4) -- (1);
    \draw (5) -- (3);
    \draw (6) -- (1);
    \draw (7) -- (3);
    \draw (8) -- (6);
    \draw (9) -- (4);

\end{tikzpicture}
\caption{Case 1: $v^\ssup{1}=\pa(v)\ge 2$.}
\end{subfigure}
\hfill
\begin{subfigure}{0.3\textwidth}
\centering
\begin{tikzpicture}[scale=1]

    \node[draw, circle, minimum size=4mm, inner sep = 0 pt, label=left:{$v^\ssup{1}$}] (1) at (0,0) {$1$};
    \node[draw, circle, minimum size=4mm, inner sep = 0 pt] (2) at (0,-1) {$2$};
    \node[draw, circle, minimum size=4mm, inner sep = 0 pt, label=left:{$v$}] (3) at (-1,-1) {$3$};
    \node[draw, circle, minimum size=4mm, inner sep = 0 pt] (4) at (-1,-2) {$4$};
    \node[draw, circle, minimum size=4mm, inner sep = 0 pt] (5) at (-0.5,-3) {$5$};
    \node[draw, circle, minimum size=4mm, inner sep = 0 pt] (6) at (1,-1) {$6$};
    \node[draw, circle, minimum size=4mm, inner sep = 0 pt] (7) at (-1.5,-3) {$7$};
    \node[draw, circle, minimum size=4mm, inner sep = 0 pt] (8) at (1,-2) {$8$};
    \node[draw, circle, minimum size=4mm, inner sep = 0 pt] (9) at (0,-2) {$9$};

    \draw (2) -- (1);
    \draw (3) -- (1);
    \draw (4) -- (3);
    \draw (5) -- (4);
    \draw (6) -- (1);
    \draw (7) -- (4);
    \draw (8) -- (6);
    \draw (9) -- (2);

\end{tikzpicture}
\caption{Case 2: $v^\ssup{1}=\pa(v)=1$.}
\end{subfigure}
\hfill
\begin{subfigure}{0.3\textwidth}
\centering
\begin{tikzpicture}[scale=1]

    \node[draw, circle, minimum size=4mm, inner sep = 0 pt] (1b) at (0,0) {$1$};
    \node[draw, circle, minimum size=4mm, inner sep = 0 pt] (2b) at (-1,-1) {$2$};
    \node[draw, circle, minimum size=4mm, inner sep = 0 pt, label=left:{$v$}] (3b) at (0,-1) {$3$};
    \node[draw, circle, minimum size=4mm, inner sep = 0 pt, label=left:{$v^\ssup{1}$}] (4b) at (0,-2) {$4$};
    \node[draw, circle, minimum size=4mm, inner sep = 0 pt] (5b) at (-1,-3) {$5$};
    \node[draw, circle, minimum size=4mm, inner sep = 0 pt] (6b) at (1,-2) {$6$};
    \node[draw, circle, minimum size=4mm, inner sep = 0 pt] (7b) at (0,-3) {$7$};
    \node[draw, circle, minimum size=4mm, inner sep = 0 pt] (8b) at (1,-1) {$8$};
    \node[draw, circle, minimum size=4mm, inner sep = 0 pt] (9b) at (1,-3) {$9$};

    \draw (2b) -- (1b);
    \draw (3b) -- (1b);
    \draw (4b) -- (3b);
    \draw (5b) -- (4b);
    \draw (6b) -- (3b);
    \draw (7b) -- (4b);
    \draw (8b) -- (1b);
    \draw (9b) -- (4b);

\end{tikzpicture}
\caption{Case 3: $v^\ssup{1} = c \neq \pa(v)$.\ }
\end{subfigure}

\caption{The three cases we distinguish for the lower bound on $\phi_n^{\ssup{2}}(v)$.}
\label{fig:three cases}
\end{figure}

 \smallskip 
 We proceed to the proof of inequality \eqref{eq:what happens above 2}.
 Recall from~\eqref{eq:NNBJordanDef} that 
 \[
 \phi_n^\ssup{2}(v) = (\phi_n(v))^2\vee\left(\phi_n(v)\cdot \phi_n(v^\ssup{1})\right).
 \]
 For the lower bound on $\phi_n^{\ssup{2}}$, we distinguish three cases; see Figure \ref{fig:three cases}.

 \smallskip\noindent\emph{Case 1: $v^\ssup{1}=\pa(v)\ge 2$.\ }
 The parent of $v$ attains the minimum in the first equation in \eqref{eq:phi-recap}, so $\phi_n(v)=\mathrm{Fr}_n(v)=\mathrm{Fr}^{>v}_n(v)$. If $v^\ssup{1}\ge 2$, it follows by the lower bound on $\phi_n$ that 
 \begin{equation}\nonumber 
     \phi_n^\ssup{2}(v)\ge \phi_n(v)\cdot\phi_n(v^\ssup{1}) = \mathrm{Fr}_n^{>v}(v) \cdot\phi_n(\mathrm{pa}(v)) \ge \mathrm{Fr}_n^{>v}(v) \cdot \min\left(\mathrm{Fr}^{>v}_n(\pa(v)), \mathrm{Fr}^{>v}_n(1)\right).
 \end{equation}
 \noindent\emph{Case 2: $v^\ssup{1}=\pa(v)=1$.\ }
Similar to case 1, $\phi_n(v)=\mathrm{Fr}_n(v)=\mathrm{Fr}^{>v}_n(v)$. Since $1$ has no parent, the minimum in \eqref{eq:phi-recap} is attained at one of its children, say $c$.

If $c\ge 3$, the tree $(T_n, c)_{1\downarrow}$ must contain the entire fringe tree of $2$: indeed, since $c\geq 3$ and $2$ are both children of $1$, the trees $(T_n,1)_{2 \downarrow}$ and $(T_n,1)_{c \downarrow}$ are disjoint and thus $(T_n,1)_{2 \downarrow} \subseteq T_n \setminus (T_n,1)_{c\downarrow} = (T_n,c)_{1\downarrow}$.


If $c=2$,  $(T_n, 2)_{1\downarrow}$ contains the entire fringe tree of $1$ witnessed after time $v\ge 2$: the unique path in $T_n$ from a vertex in $(T_n,1)_{2\downarrow}$ to a vertex in the fringe tree of $1$ witnessed after time $v$ must traverse the edge $\{1,2\}$ which is not in the fringe tree of $1$ witnessed after time $v\ge 2$. Thus, $\phi_n(1) = |(T_n,2)_{1\downarrow}| \ge \mathrm{Fr}_n^{>v}(1)$ if $c = 2$. 
Combining these subcases, if $v^\ssup1=\pa(v)=1$, then
\[\phi_n^\ssup{2}(v)\ge \phi_n(v)\phi_n(1)\ge \mathrm{Fr}_n^{>v}(v)\cdot\min\left(\mathrm{Fr}_n^{>v}(1), \mathrm{Fr}_n^{>v}(2)\right).\]

\smallskip
\noindent\emph{Case 3: $v^\ssup{1}\neq \pa(v)$.\ }
We use the bound $\phi_n^\ssup{2}(v)\ge (\phi_n(v))^2$. The minimum in the first equation in \eqref{eq:phi-recap} is attained at a child $c$ of $v$. Since $c>v\ge 2$, the subtree $(T_n, c)_{v\downarrow}$ contains vertices $1$ and~$2$. Therefore, the subtree $(T_n, c)_{v\downarrow}$ also contains the fringe trees of $1$ and $2$ witnessed after time $v$. 
Thus, 
\[
\phi_n^\ssup{2}(v)\ge (\phi_n(v))^2\ge \max \left(\mathrm{Fr}_n^{>v}(1) , \mathrm{Fr}_n^{>v}(2)\right)^2 \ge \mathrm{Fr}_n^{>v}(1)\mathrm{Fr}_n^{>v}(2).
\]
Since exactly one of the three cases we discussed needs to occur, the second inequality of Lemma~\ref{lem:minima} follows.

\smallskip
We are left to show inequality \eqref{eq:what happens at 2} for $v=2$. We distinguish 4 cases.

\noindent\emph{Case 1: $2^\ssup{1}=1, 3 \sim 2$.\ } The fringe trees of $2$ and $3$ witnessed after time $3$ and are contained in the fringe tree of $2$. Further, since $2^\ssup{1}=1$, one has $\phi_n(2) = \mathrm{Fr}_n(2)$, implying that
\begin{equation*}
	\phi_n^\ssup{2}(2) \geq \phi_n(2)^2 = \mathrm{Fr}_n(2)^2 \geq \left(\mathrm{Fr}_n^{>3}(2) + \mathrm{Fr}_n^{>3}(3)\right)^2 \geq \mathrm{Fr}_n^{>3}(2) \mathrm{Fr}_n^{>3}(3) .
\end{equation*}

\noindent\emph{Case 2: $2^\ssup{1} = c \geq 3, 3 \sim 2$.\ } If $c=3$, then the fringe tree of $2$ witnessed after time $3$ is contained in the tree $(T_n, c)_{2 \downarrow}$. If $c>3$, then the fringe tree of $3$ witnessed after time $3$ is contained in the tree $(T_n, c)_{2 \downarrow}$. This implies that $\phi_n(2) = |(T_n, c)_{2\downarrow}| \geq \min(\mathrm{Fr}_n^{>3}(2), \mathrm{Fr}_n^{>3}(2))$. Further, in either case, the fringe tree of $1$ witnessed after time $3$ is contained in the tree $(T_n,c)_{2\downarrow}$. Thus we see that
\begin{equation*}
	\phi_n^\ssup{2}(2) \geq \phi_n(2)^2 \geq \mathrm{Fr}_n^{>3}(1)\cdot \min(\mathrm{Fr}_n^{>3}(2), \mathrm{Fr}_n^{>3}(3))
	= 
	\min(\mathrm{Fr}_n^{>3}(1) \mathrm{Fr}_n^{>3}(2), \mathrm{Fr}_n^{>3}(1) \mathrm{Fr}_n^{>3}(3)) .
\end{equation*}

\noindent\emph{Case 3: $2^\ssup{1} = 1, 3 \sim 1$.\ } If $a \coloneqq 1^{\ssup{1}} \in \{2,3\}$, then the fringe tree of $1$ witnessed after time $3$ is contained in the tree $(T_n,a)_{1\downarrow}$. If $a = 1^{\ssup{1}} > 3$, then the fringe tree of $3$ witnessed after time $3$ is contained in the tree $(T_n,a)_{1\downarrow}$. This implies that $\phi_n(1) = |(T_n,a)_{1\downarrow}| \geq \min(\mathrm{Fr}_n^{>3}(1), \mathrm{Fr}_n^{>3}(3))$. Further, since $2^\ssup{1} = 1 = \mathrm{pa}(v)$, we always have that $\phi_n(2) = \mathrm{Fr}_n(2) \geq \mathrm{Fr}_n^{>3}(2)$. Combining these inequalities gives
\begin{equation*}
	\phi_n^{\ssup{2}}(2) \geq  \phi_n(2) \phi_n(1) \geq \mathrm{Fr}_n^{>3}(2) \cdot \min(\mathrm{Fr}_n^{>3}(1), \mathrm{Fr}_n^{>3}(3))
	=
	\min(\mathrm{Fr}_n^{>3}(2) \mathrm{Fr}_n^{>3}(1) , \mathrm{Fr}_n^{>3}(2) \mathrm{Fr}_n^{>3}(3)) .
\end{equation*}

\noindent\emph{Case 4: $2^\ssup{1} \geq 3, 3 \sim 1$.\ } Since $3 \nsim 2$, we have $a \coloneqq 2^{\ssup{1}} \geq 4$, and thus the tree $(T_n,a)_{2 \downarrow}$ contains both the fringe tree of $1$ witnessed after time $3$ as well as the fringe tree of $2$ witnessed after time $3$, implying that
\begin{equation*}
	\phi_n^\ssup{2}(2) \geq \phi_n(2)^2 =|(T_n,a)_{2\downarrow}| \geq \left(\mathrm{Fr}_n^{>3}(1) + \mathrm{Fr}_n^{>3}(3)\right)^2 \geq \mathrm{Fr}_n^{>3}(1) \mathrm{Fr}_n^{>3}(3) .
\end{equation*}
Since exactly one of the four preceding cases needs to occur, this shows inequality \eqref{eq:what happens at 2}.
\end{proof}

By Lemma~\ref{lem:minima}, the upper bound on the lower tail has been reduced to controlling products of witnessed fringe-tree sizes. Since these variables are negatively associated, the following elementary claim allows us to study the lower tail of their product.
\begin{claim}\label{claim:negative_corelation}
Let  $X,Y$ be two non-negative random variables that are negatively associated and for which there exists a constant $C > 0$ such that for all $s > 0$,
\begin{align*}
    \p \left( X \leq s\right),\  \p \left( Y \leq s \right) \leq Cs .
\end{align*}
Then, for any $s\in(0,1]$,
\begin{align*}
    \p \left( XY \leq s \right)  \leq 2 C^2  (\log_2(1/s)+1)  s + 2 C s.
\end{align*}
\end{claim}
\begin{proof}
We bound $
\p(XY\le s) \le \p(X\le s)+\p(Y\le s)
+ \p\left(XY\le s, X> s, Y> s\right).$ \\
The first two terms are each at most $Cs$. For the third term we use a dyadic decomposition for $X$: for $j\ge0$, define $I_j \coloneqq (s2^j, s2^{j+1}]$. If $X\in I_j$ and $XY\le s$, then \(Y\le 2^{-j}\). Hence,
\[
\begin{aligned}
\p\left(XY\le s,\ X> s,\ Y> s\right)
&
\le \sum_{j=0}^{\lfloor\log_2(1/s)\rfloor} \p \left( X \in I_j, Y \leq 2^{-j}  \right) \\
&\le \sum_{j=0}^{\lfloor\log_2(1/s)\rfloor} \p\left(X\le s2^{j+1}, Y\le 2^{-j}\right).
\end{aligned}
\]
By negative association,
\[
\p\left(X\le s2^{j+1}, Y\le 2^{-j}\right)
\le \p(X\le s2^{j+1})\cdot \p(Y\le 2^{-j}) \le Cs2^{j+1}\cdot C2^{-j} = 2C^2 s.
\]
Summing over $j=0,\dots,\lfloor\log_2(1/s)\rfloor$ yields the claimed bound.
\end{proof}

To obtain the matching lower bound in Lemma~\ref{lem:boosted small}, we next control the product $\mathrm{Fr}_n(v)\mathrm{Fr}_n(\pa(v))$. The following lemma provides the required estimate. Its proof is deferred to Appendix \ref{sec:appendix}.

\begin{restatable}{lemma}{appendixfringesize}
\label{lem:appendix_fringe_size}
    For every $\eps>0$ there exists a constant $c>0$ such that for any  $n\ge 3, v\in[3,n/4], s\in(3(v/n)^2 \vee (v/n)^{2-\eps},1/2)$, 
    \begin{equation}\label{eq:lem2.6 (1)}
        \p\Big(\mathrm{Fr}_n(v)\mathrm{Fr}_n(\pa(v)) \le s(n/v)^2\Big)\ge c s \log(1/s).
    \end{equation}
    Further, there exists a constant $c>0$ such that for any  $n\ge 3, v\in[3,n/4], s\in(3(v/n)^{2},1)$, 
    \begin{equation}\label{eq:lem2.6 (2)}
        \p\Big(\mathrm{Fr}_n(v)\mathrm{Fr}_n(\pa(v)) \le s(n/v)^2\Big)\ge c s.
    \end{equation} 
\end{restatable}

Next, we prove the main result of this Section, Lemma \ref{lem:boosted small}.

\begin{proof}[Proof of Lemma \ref{lem:boosted small}]
Since $\phi_n^\ssup{2}(1)$ and $\phi_n^\ssup{2}(2)$ have the same distribution, we assume without loss of generality that $v\ge 2$, at the cost of changing the constants $C,c$.
    We start by proving the upper bound. Let $\cG_v$ be the $\sigma$-algebra generated by the evolution of the tree up to time $v$, i.e., $\cG_v=\sigma\left( \{u, \mathrm{pa}(u) \} ; 2 \leq u \leq v \right)$. Note that the random variables $\left(\mathrm{Fr}_n^{>v}(j) : 1\leq j \leq n \right)$ are independent of $\cG_v$, whereas $\mathrm{pa}(v)$ is measurable with respect to $\cG_v$.
    By Lemma \ref{lem:minima} and a union bound,
    \begin{align*}
        \p \left( \phi_n^{\ssup{2}}(v) \leq \left( \frac{n}{t v} \right)^2 \right)
        &
        \le 
        \E \left[ \p \left( \min_{(x,y) \in \left\{ (v,\mathrm{pa}(v)), (v,1), (v,2), (1,2) \right\}} \mathrm{Fr}_n^{>v}(x) \mathrm{Fr}_n^{>v}(y) \leq \left( \frac{n}{t v} \right)^2 \,\Big|\, \cG_v \right)  \right]
        \\
        &
        \leq
        \E \left[ \sum_{(x,y) \in \left\{ (v,\mathrm{pa}(v)), (v,1), (v,2), (1,2) \right\}} \p \left(  \mathrm{Fr}_n^{>v}(x) \mathrm{Fr}_n^{>v}(y) \leq \left( \frac{n}{t v} \right)^2 \,\Big|\, \cG_v \right)  \right].
    \end{align*}   
    The pairs of witnessed fringe-trees in the sum are identically distributed by the exchangeability in Lemma \ref{lem:beta-bin-fringe}(1), and they are independent of $\cG_v$. Further, for $v \geq 3$, all pairs in $\{(v,\mathrm{pa}(v))$, $(v,1)$, $(v,2)$, $(1,2)\}$ consist of two different vertices. Thus, 
     \[
     \p\left( \phi_n^{\ssup{2}}(v) \leq \left( \frac{n}{t v} \right)^2 \right)\le 4\p \left( \frac{v}{n}\mathrm{Fr}_n^{>v}(v) \cdot \frac{v}{n}\mathrm{Fr}_n^{>v}(1) \leq \frac{1}{t^2} \right).
     \]
     The upper bound on the lower tail in Lemma \ref{lem:beta-bin-fringe}(3) implies that for all $M\ge 1$,
     \[
     \p\left(\frac{v}{n}\mathrm{Fr}_n^{>v}(1)\le \frac{1}{M}\right) = \p\left(\frac{v}{n}\mathrm{Fr}_n^{>v}(v)\le \frac{1}{M}\right) = \p\left(\mathrm{Fr}_n(v)\le \frac{n}{Mv}\right) \le \frac{2}{M}.
     \]
        Since $\tfrac{v}{n}\mathrm{Fr}_n^{>v}(1)$ and $\tfrac{v}{n}\mathrm{Fr}_n^{>v}(v)$ are negatively associated by Lemma \ref{lem:beta-bin-fringe}(2), Claim \ref{claim:negative_corelation} with $s=1/t^2$ implies that 
        \[
        \p \left( \phi_{n}^{\ssup{2}}(v) \leq \left( \frac{n}{t v} \right)^2 \right)
        \leq
        4\p \left( \frac{v}{n}\mathrm{Fr}_n^{>v}(v) \cdot \frac{v}{n}\mathrm{Fr}_n^{>v}(1) \leq \frac{1}{t^2} \right)
        \le 4\cdot 2\cdot 2^2\cdot \frac{1+\log_2 t^2}{t^2} +\frac{4\cdot2\cdot 2}{t^2},
        \]
        finishing the upper bound for $v \geq 3$. For $v=2$, we have, using similar arguments and \eqref{eq:what happens at 2}, that
        \begin{multline*}
        	\p \left( \phi_n^{\ssup{2}}(v) \leq \left( \frac{n}{t v} \right)^2 \right)
        	\le \sum_{(x,y) \in \left\{ (1,2), (2,3), (1,3) \right\}}
        	\p \left( \mathrm{Fr}_n^{>3}(x) \mathrm{Fr}_n^{>3}(y) \leq \left( \frac{n}{t v} \right)^2 \right)
        	\\
        	= 3
        	\p \left( \frac{3}{n} \mathrm{Fr}_n^{>3}(1) \frac{3}{n} \mathrm{Fr}_n^{>3}(2) \leq \frac{9}{4 t^2} \right) \leq C \frac{1 + \log t}{t^2},
        \end{multline*}   
        for some constant $C< \infty$ and all $t\geq 1$.

        We turn to the lower bound, starting with $v\ge 3$.
        The lower bound follows from inequality \eqref{eq:lem2.6 (1)} of Lemma \ref{lem:appendix_fringe_size}, since
        \begin{align*}
            \p \left( \phi_n^{\ssup{2}}(v) \leq \left( \frac{n}{tv} \right)^2 \right) \geq \p \left( \mathrm{Fr}_n(v) \mathrm{Fr}_n(\pa(v)) \leq t^{-2} \left( \frac{n}{v} \right)^2 \right) \geq c t^{-2} \log(t^2) = 2c \frac{\log t}{t^2},
        \end{align*}
        for some constant $c>0$ (depending on $\eps>0$) and all $t > 0$ with $3(v/n)^2 \vee (v/n)^{2-\eps} \leq t^{-2} \leq 1/2$, which is equivalent to $\sqrt{2} \leq t \leq \tfrac{1}{\sqrt{3}} (n/v) \wedge (n/v)^{1-\eps/2}$. Since $(n/v)^{1-\eps/2} \geq (n/v)^{1-\eps}$, \eqref{eq:lem2.1 lower bound} follows.
        The proof of \eqref{eq:lem2.1 lower bound version2} follows analogously, with \eqref{eq:lem2.6 (1)} replaced by \eqref{eq:lem2.6 (2)}.

        For $v=2$, observe that conditioned on the event $\mathrm{pa}(3)=1$, the tree $T_3$ consists of a line of length $3$ with vertices $2$ and $3$ at the ends. In particular, conditioned on $\mathrm{pa}(3)=1$, $\phi_n^{\ssup{2}}(2)$ has the same distribution as $\phi_n^{\ssup{2}}(3)$. Furthermore, $\phi_n^{\ssup{2}}(3)$ is independent of $\mathrm{pa}(3)$, since the tree $T_3$ always consists of a line of length $3$, with $3$ at one of its endpoints, implying that
        \begin{align*}
            \p \big( \phi_n^{\ssup{2}}(2) \leq s \big) &\geq \p(\pa(3)=1)\p \big( \phi_n^{\ssup{2}}(2) \leq s \mid \mathrm{pa}(3)=1 \big)  \\&= 
            \frac{1}{2}\p \big( \phi_n^{\ssup{2}}(3) \leq s \mid \mathrm{pa}(3)=1 \big) 
            =
            \frac{1}{2}\p \left( \phi_n^{\ssup{2}}(3) \leq s\right) 
        \end{align*}
        for all $s > 0$. The proofs of \eqref{eq:lem2.1 lower bound} and \eqref{eq:lem2.1 lower bound version2} for $v=2$ follow by decreasing the constant $c$.
\end{proof}

\subsubsection{Upper tail}\label{subsec:jordan2c-upper}
We next turn to the analysis of the upper tail of the Jordan-2 centrality $\phi^{\ssup{2}}$, which will be used to study the lower tail of the Jordan-2 ordering $\hs_2(v)$. We state the lower and upper bounds in the following two lemmas.
\begin{lemma}[Upper tail of the Jordan-2 centrality: lower bound]\label{lem:J2_LT_LB_main_term}
    There exists a constant $C>0$ such that for any $n\geq 3$, $v\in[n]$, and any $S \geq 1$,
    
    $$\p\left(\phi_n^\ssup{2}(v)\ge S\left(\frac{n}{v}\right)^2\right)  
        \geq \frac{1}{128S}-\frac{C}{v}.$$
\end{lemma}

\begin{restatable}[Upper tail of the Jordan-2 centrality: upper bound]{lemma}{Fringefringeparentlarge}
    \label{lem:Fringe_fringe_parent_large}
    There exists a constant $C>0$ such that for any $n\geq 2$, $v\in[n]$ and $S > 0$, 
    \[
    \p\Big(\phi_n^\ssup{2}(v) \ge S\left(\frac{n}{v}\right)^2 \Big) \le \frac{C}{S}.
    \] 
\end{restatable}

The proofs of these lemmas rely on the following auxiliary lemma on fringe tree sizes, complementing Lemma \ref{lem:beta-bin-fringe}. The lemma assumes $v\ge2$, as $\mathrm{Fr}_n(1)=n$ deterministically.

\begin{restatable}[Fringe tree: upper tail]{lemma}{fringetreedistrib}
\label{lem:proba_large_fringe}
    For any $n\geq 3$, $v\ge 2$, and any $k\geq 1$,
    \begin{equation}\label{eq:proba_fringe-k}\p\left(\mathrm{Fr}_n(v)=k\right)
    \leq
    2e^2\cdot\frac{v}{n}\cdot\exp\left(-k\frac{v}{n}\right).\end{equation}
    Moreover, for any $t>0$,

    \begin{equation}\label{eq:proba_fringe-large}\p\left(\mathrm{Fr}_n(v)\geq t\frac{n}{v} \right) 
    \leq 
    2e^2\exp (-t),\end{equation}
    and for all $v \leq \frac{n}{2}$, $n\ge 2$, and $t \in \left[1,\frac{v}{8}\right]$,
    \begin{equation}\label{eq:proba-parent-large-lower}
        \p \left( \mathrm{Fr}_n(v) > t \frac{n}{v} \right) \geq  \frac{1}{2} \exp\left( - 8t \right). 
    \end{equation}
    Furthermore, there exists a constant $C>0$ such that for all $v\ge 2$, $n\geq v$, and $t>0$,
    \begin{equation}\label{eq:proba-parent-large2}
        \p\left(\mathrm{Fr}_n(\pa(v))\ge t\frac{n}{v} \right) \le \frac{C}{t}, 
    \end{equation}
    and
    \begin{equation}\label{eq:proba-parent-large}
        \p\left(\mathrm{Fr}_n(\pa(v))\ge \frac{n}{2} \right) \le \frac{C}{v} .
    \end{equation}
\end{restatable}

The proof of this lemma follows from standard properties of P\'olya urns and is given in Appendix~\ref{sec:appendix}, along with the proof of Lemma \ref{lem:Fringe_fringe_parent_large}. We finish this section with a proof of the lower bound in Lemma~\ref{lem:J2_LT_LB_main_term}.

\begin{proof}[Proof of Lemma \ref{lem:J2_LT_LB_main_term}]
The statement of the lemma is clear for $S > \frac{v}{16}$ or for $v=1$, by taking the constant $C$ large enough. Thus, we can assume that $S \leq \frac{v}{16}$ for the rest of the proof.
    Fix some $n\geq1$ and $S\geq 1$.  We first relate the Jordan-2 centrality to the product of the fringe-tree size of $v$ and its parent, that is,

    \begin{align*}
        \p\left(\phi_n^\ssup{2}(v) \ge S\left(\frac{n}{v}\right)^2\right)
        \geq 
        \p\left(\left\{\phi_n^\ssup{2}(v) \ge S\left(\frac{n}{v}\right)^2\right\}\cap \left\{ \mathrm{Fr}_n(\pa(v))\leq \frac{n}{2} \right\} \right).
    \end{align*}
    If $\mathrm{Fr}_n(\pa(v))\leq n/2$, then $\mathrm{Fr}_n(v)\leq n/2$ and hence $\phi^\ssup{2}_n(v)=\mathrm{Fr}_n(\pa(v))\mathrm{Fr}_n(v)$. Therefore,

    \begin{equation}\nonumber 
    \begin{aligned}
        \p\left(\phi_n^\ssup{2}(v) \ge S \left(\frac{n}{v}\right)^2\right)
        &\geq 
        \p\Big(\mathrm{Fr}_n (\pa(v))\mathrm{Fr}_n(v) \ge S\left(\frac{n}{v}\right)^2 \Big)-\p\Big( \mathrm{Fr}_n(\pa(v)) > \frac{n}{2} \Big).
    \end{aligned}
    \end{equation}
    By \eqref{eq:proba-parent-large}, there exists a constant $C>0$ such that for all $v\ge 2$, and $n\ge 3$, the probability of the event $\{\mathrm{Fr}_n(\pa(v)) > n/2\}$ is at most $C/v$. So, 
    \begin{equation}\label{eq:Jordan2_lower_tail_LB_decompo}
    \begin{aligned}
        \p\left(\phi_n^\ssup{2}(v) \ge S\left(\frac{n}{v}\right)^2\right)
        &\geq 
        \p\Big(\mathrm{Fr}_n (\pa(v))\mathrm{Fr}_n(v) \ge S\left(\frac{n}{v}\right)^2 \Big)-C/v.
    \end{aligned}
    \end{equation}
    In the remainder, we establish a lower bound on the probability on the right-hand side of this inequality. We intersect with the likely events that both $v$ and its parent have a fringe-tree size of the right order. That is,
    \begin{equation*}
    \begin{aligned}
        &\p\left( \mathrm{Fr}_n(\pa(v))\mathrm{Fr}_n(v) \ge S\left(\frac{n}{v}\right)^2 \right)
        \\
        &
        \hspace{2cm}
        \geq 
        \p\left( \left\{\mathrm{Fr}_n(\pa(v))\mathrm{Fr}_n(v) \ge S\left(\frac{n}{v}\right)^2\right\} \cap
         \left\{\mathrm{Fr}_n(v) \geq \frac{n}{4v}\right\}
        \cap
        \left\{\mathrm{Fr}_n(\pa(v)) \geq \frac{n}{4\pa(v)}\right\}\right)\\
        &
        \hspace{2cm}
        \geq 
        \p\left( \left\{\frac{n^2}{16v\pa(v)} \ge S \left( \frac{n}{v}\right)^2 \right\}  \cap
         \left\{\mathrm{Fr}_n(v) \geq \frac{n}{4v}\right\}
        \cap
        \left\{\mathrm{Fr}_n(\pa(v)) \geq \frac{n}{4\pa(v)}\right\}\right).
    \end{aligned}
    \end{equation*}
    The event $\left\{\tfrac{n^2}{16v\pa(v)} \ge S \left( n/v\right)^2 \right\}$ is satisfied for $\mathrm{pa}(v) \leq v/(16S)$.
    Using that $\pa(v)$ is uniformly distributed on $[v-1]$ and is independent from $\mathrm{Fr}_n(v)$, the law of total probability yields
    
    \begin{equation*}
    \begin{aligned}
        \p\Big(\mathrm{Fr}_n(\pa(v)) &\mathrm{Fr}_n(v) \ge S \left(\frac{n}{v}\right)^2 \Big) \\
        &\hspace{0.1cm}\geq 
        \frac{1}{v-1}\sum_{u\leq v/(16S)}
        \p\left( \left\{\mathrm{Fr}_n(v) \geq \frac{n}{4v}\right\}
        \cap
        \left\{\mathrm{Fr}_n(u) \geq \frac{n}{4u}\right\}\ \Big| \ \pa(v)=u \right)\\
        &\hspace{0.1cm}=
        \frac{\p\left(\mathrm{Fr}_n(v) \geq \tfrac{n}{4v}\right)}{v-1}\sum_{u\leq v/(16S)}
        \p\left( 
        \left\{\mathrm{Fr}_n(u) \geq \frac{n}{4u}\right\} \Big|  \left\{\mathrm{Fr}_n(v) \geq \frac{n}{4v}\right\} \cap \left\{\pa(v)=u\right\} \right),
    \end{aligned}
    \end{equation*} 
    where the last inequality holds since $\mathrm{Fr}_n(v)$ and $\mathrm{pa}(v)$ are independent.
    Furthermore, the random variable
    $\mathrm{Fr}_n(u)  \mid  \left\{\mathrm{Fr}_n(v) \geq \tfrac{n}{4v}\right\} \cap \left\{\pa(v)=u\right\}$ stochastically dominates $\mathrm{Fr}_n(u)$,  so that
    \begin{equation*}
    \begin{aligned}
        \p\left(\mathrm{Fr}_n(\pa(v)) \mathrm{Fr}_n(v) \ge S\left(\frac{n}{v}\right)^2 \right)
        &\geq 
        \frac{\p\left(\mathrm{Fr}_n(v) \geq \tfrac{n}{4v}\right)}{v-1}\sum_{u\leq v/(16S)}
        \p\left( 
        \mathrm{Fr}_n(u) \geq \frac{n}{4u} \right).
    \end{aligned}
    \end{equation*}
    Using Lemma \ref{lem:beta-bin-fringe}(3), we get that $\p\left(\mathrm{Fr}_n(v) \geq \frac{n}{4v}\right)\geq \tfrac{1}{2}$ and $\p\left( \mathrm{Fr}_n(u) \geq \tfrac{n}{4u} \right)\geq \tfrac{1}{2} $, so

    $$\p\left(\mathrm{Fr}_n(v) \mathrm{Fr}_n(\pa(v))\ge S\left(\tfrac{n}{v}\right)^2\right)  
    \geq
    \frac{\tfrac{1}{2}}{v-1} \lfloor v/(16S) \rfloor \frac{1}{2}
    \geq
    \frac{1}{128S},$$
    where we used the assumption $S \leq \frac{v}{16}$ for the last inequality. Inserting this into \eqref{eq:Jordan2_lower_tail_LB_decompo} concludes the proof of the lemma.
\end{proof}

\subsection{Many vertices with large Jordan-2 centrality is unlikely}\label{subsec:fringe}
We now proceed to controlling the \emph{global} events to bridge the Jordan-2 centrality to the Jordan-2 ordering. We start with the upper bound on the upper tail, bounding the second event in \eqref{eq:hs2 split} in the outline: the event that many vertices have unusually large Jordan-2 centrality. 

By the upper bound $\phi_n^\ssup{2}(u)\le \mathrm{Fr}_n(u)\cdot \mathrm{Fr}_n(\pa(u))$, see \eqref{eq:upper-bound-jordan2}, it is enough to bound the number of vertices for which this product exceeds a threshold of order $(n/t)^2$. 
The following lemma is the main result of this subsection.

\begin{restatable}{lemma}{mixedlarge}
\label{lem:mixed large} 
    There exists a constant $C>0$ such that for all $n\in\N$ and $t \ge 1$, 
    \begin{equation}
        \p\left(\big|\big\{v\in[n]\setminus{1}: \mathrm{Fr}_n\left(\pa(v)\right)\mathrm{Fr}_n(v) > (n/t)^2\big\}\big|>Ct\right) \le \exp\left(-t^{1/4}\right).
    \end{equation}    
\end{restatable}
Unlike the analysis of the lower tail of the Jordan-2 centrality of a fixed vertex, Lemma \ref{lem:mixed large} requires understanding of global counts of vertices whose subtree sizes lie in prescribed ranges, together with information about the fringe-tree sizes of their children. Counting these is difficult in the arrival-order description of $T_n$ given in Section \ref{sec:model}. 

To overcome this, we use an equivalent sampling of the random recursive tree on the Ulam--Harris embedding. To motivate this construction, we first state a construction close to the arrival-order model and record its properties. We then proceed to the fringe-splitting construction, which provides the conditional independence needed for our counting arguments. Both constructions are well-known, see for example \cite[Section 2.1]{addarioberry2024leafstrippinguniformattachment} or \cite{goldschmidt2016short, Baumler2026, neveu1986arbres} and references therein. Afterwards, we prove two auxiliary lemmas that also rely on the fringe-splitting construction, and eventually prove Lemma \ref{lem:mixed large}. 

\medskip 

To start the constructions, we first define the \emph{Ulam--Harris} (UH) description of a tree. Consider the set of vertices
\begin{equation}
    \cU := \bigcup_{j=0}^\infty \cU_j:=\bigcup_{j=0}^\infty \N^j, \qquad \text{where}\quad \cU_0=\{\emptyset\},
\end{equation}
whose elements we call Ulam--Harris nodes. The root is the empty sequence $\emptyset$, and each node $u\in\cU$ has ordered child slots $\{(u,i)\}_{i\ge 1}$. A node is said to be occupied once a vertex of the growing tree has been placed at that location. For $k\in \N$, we also say that $u$ is the parent of $(u,k)$ and that $\left((u,i)\right)_{i\in \{1,\ldots,k-1\}}$ are the left siblings of $(u,k)$.
Both constructions produce a UH tree with additional node labels for the arrival times. In the first construction, labels are assigned dynamically as the tree grows; in the second, they are drawn only after the underlying tree on $n$ vertices has been generated. In both cases, the labeled unordered tree obtained by forgetting the Ulam--Harris ordering has the same law as the random recursive tree. 

\medskip 
\noindent\emph{Construction 1: Ulam--Harris embedding.\ }
In the first construction, we embed the recursive tree into the Ulam–Harris tree respecting the Ulam--Harris left-to-right order.
The process starts with the root $\emptyset$, labeled $1$. 
At each step $t\ge2$, a new vertex arrives and chooses one of the currently occupied nodes uniformly at random as its parent.
If the chosen parent is the Ulam–Harris node $u$, the new vertex is placed at its first available child slot $(u,k)$, where $k$ is the smallest index such that $(u,k)$ is not yet occupied.
The node is then labeled with the arrival time $t$ and marked as occupied.

\smallskip 
This construction yields a labeled ordered tree $A_n\subseteq\cU$ with $|A_n|=n$, where labels increase along ancestral lines, and among siblings the label order coincides with their lexicographic order in the Ulam--Harris tree. For each occupied node $u\in A_n$, we write 
\[
\mathrm{Fr}_n(u):= \big|\{v\in A_n : v\text{ has $u$ as an ancestor in $\cU$}\}\big|
\]
for the fringe-tree size of $u\in \cU$ (so we abuse notation slightly, as $\mathrm{Fr}_n$ can act both on $\cU$ and on $V_n=[n]$). 
We record three properties:
\begin{enumerate}[leftmargin=*, itemsep=0em]
\item Forgetting the Ulam–Harris labels recovers the random recursive tree defined above, in which each vertex $t$ attaches uniformly to a vertex in $[t-1]$. 
    \item Conditional on the occupied set $A_n=a_n$ (without arrival time labels), each admissible labeling -- labels increase along simple root-to-leaf paths, and among siblings, the label order is consistent with the left-to-right order in the Ulam--Harris tree -- is equally likely. Indeed, the UH-embedding is uniquely determined by the sequence of parent choices $(\mathrm{pa}(v))_{v\in[2,n]}$, and each such sequence has probability $\prod_{v=2}^n1/(v-1)$. Thus all sequences resulting in the occupied set $A_n$ are equally likely.
    \item Conditional on any realization of the tree such that the fringe tree sizes of $u$ and $(u,1),\ldots, (u,k)$ equal $m$ and $m_1,\ldots, m_k$, respectively, the size of the fringe tree of $(u,k+1)$ is uniformly distributed on $\{1,\ldots, m-(1+m_1+\ldots+m_k)\}$ if $1+m_1+\ldots+m_k<m$, and equals $0$ otherwise. This follows from a standard P\'olya-urn argument that tracks the evolution of the subtree sizes. 
\end{enumerate}
Observations 2 and 3 motivate the next construction.

\medskip 
\noindent\emph{Construction 2: Fringe-size splitting.\ }
An equivalent way to generate the tree is to work directly with fringe-tree sizes, rather than attaching vertices one by one. We start with the root $\emptyset$, whose fringe tree contains all $n$ vertices. The remaining $n-1$ vertices are allocated recursively along the Ulam--Harris tree by successive uniform splits. 

\smallskip 
Given a node $u$ with $\mathrm{Fr}_n(u)=m>1$, we reveal the fringe tree sizes of its children from left to right. Set $R_1=m-1$; for its leftmost child $(u,1)$, let $\mathrm{Fr}_n\left((u,1)\right)=X_1\sim\mathrm{Unif}\{1,\ldots, R_1\}$. For $k\ge 1$, work conditionally on $R_{k}$ and $X_{k}$ for $k\ge 1$, and continue while $R_{k+1}\ge1$:
\begin{equation}\label{eq:uniform-splitting}
R_{k+1}=R_k-X_k, \qquad\text{and}\qquad \mathrm{Fr}_n\left((u,k+1)\right)=X_{k+1}\sim\mathrm{Unif}\{1,\ldots, R_{k+1}\}.
\end{equation}
This construction provides the occupied set $A_n$ in the Ulam--Harris tree, but does not assign arrival-time labels. Conditional on the occupied set, we let the arrival-time labeling be uniform among all admissible labelings, matching Construction 1. By Observations 2 and 3 above, the resulting tree has the same distribution as the uniform random recursive tree (omitting the Ulam--Harris ordering).

\smallskip 
The subtree sizes can be sampled in any admissible order---that is, the fringe-tree size of each vertex $(u,k)$ in the Ulam--Harris tree is sampled after its parent $u$ and left siblings, $\left((u,i)\right)_{i\in \{1,\ldots,k-1\}}$, have been assigned fringe-tree sizes. The recursive allocation may therefore be carried out in depth-first, breadth-first, or any other admissible order without changing the distribution. 

\medskip 
\noindent\emph{Auxiliary lemmas.}\ 
We use the independence in this second construction to prove two auxiliary lemmas on the fringe-tree-size distribution, which then leads to the proof of Lemma \ref{lem:mixed large}. We introduce notation for the number of vertices with fringe-tree size in an interval $(a,b]$, where $a,b \in \N\cup \{+\infty\}$. Let
\begin{equation}
N_n^\mathrm{Fr}(a,b]:=\big|\big\{ u \in A_n: \mathrm{Fr}_n(u)\in (a,b]\big\}\big| = \big|\big\{ v\in[n]: \mathrm{Fr}_n(v)\in (a,b]\big\}\big|.
\end{equation}
Informally, the size of the fringe-tree of vertex $v$ is of order $n/v$; thus the number of vertices with fringe-tree size in $[\ell+1, 2\ell]$ should be of order $n/\ell$. The next lemma bounds the upper tail. 
We also write $A_n^\prime = A_n \setminus \{\emptyset\}$ for the set of occupied nodes that are not the root.

\begin{lemma}[Fringe-tree size counts]\label{lem:fringe tree count}
    Let $\ell, n\in\N$, and assume $r\ge 20\tfrac{n}{\ell}$. Then, 
    \[
    \p\left(N_n^\mathrm{Fr}(\ell, 2\ell] \ge r\right)\le 27e^{-0.04r}.
    \]
    \begin{proof}
        The statement is clear for $\ell \geq n$, since $N_n^\mathrm{Fr}(\ell, 2\ell]=0$ for $\ell \geq n$. Thus, assume that $\ell < n$ for the rest of the proof.
        We work under the fringe-tree splitting construction. Recall that $A_n^\prime$ denotes the set of occupied Ulam--Harris nodes in the tree of size $n$ that are not the root.
        Let 
        \begin{align*}
            B_1&:=\big\{(u,i)\in A_n^\prime: \mathrm{Fr}_n\left((u,i)\right)\in(\ell, 2\ell]\big\}, \\
            B_2&:=\big\{(u,i)\in A_n^\prime: \mathrm{Fr}_n(u)\ge \ell+1, \ \mathrm{Fr}_n\left((u,i)\right)\in[\ell/2, \ell]\big\}.
        \end{align*}
        The lemma aims to establish an upper bound on the size of the set $B_1$. The set $B_2$ will assist with this. Note that $B_1\cup B_2$ is the set of nodes that are not the root, whose fringe size lies in $(\ell/2, 2\ell]$, and whose parents have fringe size at least $\ell+1$. We bound 
        \begin{align*}
            \p\left(N_n^\mathrm{Fr}(\ell, 2\ell] \ge r\right)& \leq \sum_{t=r-1}^\infty\p\left(|B_1|+|B_2|=t, |B_1|\ge r-1\right)\\ &
            \leq \sum_{t=r-1}^\infty\big(\p\left(|B_1|\ge 4t/5, |B_1|+|B_2|=t\right)+\p\left(|B_2| > t/5, |B_1|+|B_2|=t\right) \big).
        \end{align*}
        The fringe trees attached to vertices in $B_2$ are disjoint, as a node with fringe-tree size at most $\ell$ can never have a descendant with fringe-tree size strictly larger than $\ell$. Thus, $|B_2|\cdot \ell/2 \le n$, which is equivalent to $|B_2|\le 2n/\ell$. As $t\ge r-1\ge 20n/\ell-1 > 10n/\ell$ by assumption, the second probability in the sum is equal to 0 for all $t \geq r-1$. Thus, 
        \[
        \p\left(N_n^\mathrm{Fr}(\ell, 2\ell] \ge r\right) \le \sum_{t=r-1}^\infty\p\left(|B_1|\ge 4t/5, |B_1|+|B_2|=t\right).
        \]
        
        For $1\leq i \leq n$, let $\{ u^i_1,\ldots,u^i_{n^i}\}= [n]^i$ in lexicographic ordering  and let $\{ v_1,\ldots,v_{K} \} = \bigcup_{i=1}^n [n]^i$ with $K=\sum_{i=1}^n n^i$, so that the level of $v_j$, i.e., the unique value $i\in [n]$ s.t. $v_j \in [n]^i$, is non-decreasing in $j$, and so that each block $[n]^i$ is in the lexicographic ordering.
        The sets $B_1$ and $B_2$ can be constructed iteratively via the fringe-tree splitting construction, by sampling fringe-tree sizes of the vertices $v_1,\ldots,v_K$. Let $B_1^j$ and $B_2^j$ be the sets
        \begin{align*}
            B_1(j)& := \big\{(u,i)\in \{v_1,\ldots,v_j\} : \mathrm{Fr}_n\left((u,i)\right)\in(\ell, 2\ell]\big\}, \\
            B_2(j)& := \big\{(u,i)\in \{v_1,\ldots,v_j\} : \mathrm{Fr}_n(u)\ge \ell+1, \ \mathrm{Fr}_n\left((u,i)\right)\in[\ell/2, \ell]\big\}
        \end{align*}

        Let $v_j=(u,i)$ and suppose that we reveal the uniform random variable $X_i = \mathrm{Fr}_n((u,i))$ in~\eqref{eq:uniform-splitting} determining the fringe-tree size of $(u,i)$ in two stages: first we reveal whether $X_i\in[\ell/2, 2\ell]$, then conditional on this event we sample whether $(u,i)$ belongs to $B_1$ or $B_2$; independently of the past, we have that $X_i\in[\ell/2, \ell]$ with probability at least $1/3$ as $X_i$ is uniform on $[\ell/2, (2\ell)\wedge R_i]$. As a result, we can couple $|B_1(j)|, |B_2(j)|$ with a sequence of i.i.d. Bernoulli(1/3)-distributed random variables $Y_1,Y_2,\ldots$ so that
        \begin{equation*}
            |B_2(j)| \geq \sum_{\ell=1}^{|B_1(j)| + |B_2(j)|} Y_\ell 
        \end{equation*}
        for all $j=1,\ldots,K$. Since $|B_1(K)|=|B_1|$ and $|B_2(K)|=|B_2|$, this implies that
        \begin{align*}
            \p\left(N_n^\mathrm{Fr}(\ell, 2\ell] \ge r\right)
            &\leq \sum_{t=r-1}^\infty\p\left(|B_1|\ge \tfrac{4}{5}t, |B_1|+|B_2|=t\right)\\
            &= \sum_{t=r-1}^\infty\p\left(|B_2|\le \tfrac{1}5t, |B_1|+|B_2|=t\right)\\ &\le \sum_{t=r-1}^\infty\p\left(\mathrm{Bin}(t, 1/3)\le t/5\right)\\& \le \sum_{t=r-1}^\infty\exp\left(-t\Lambda_{1/3}(0.2)\right)\le \sum_{t=r-1}^\infty e^{-0.04t} \le 27 e^{-0.04r},
            \end{align*}
            where 
        \begin{equation}\label{eq:large-dev}
            \Lambda_p(x):=x\log(x/p)+(1-x)\log\left((1-x)/(1-p)\right)
        \end{equation} is the large-deviation rate function of a Bernoulli-$p$ distribution.
    \end{proof}
\end{lemma}
The next lemma bounds the number of children with a large fringe tree  relative to their parent. 
\begin{lemma}[Children with large fringe trees] \label{lem:stochastic_domination}
Let $n\in\N, u \in \cU$, and $\alpha\in(0,1)$. Then, for all $k\ge 3\log_2(1/\alpha)$ and $\ell \in \N$
    \[
    \p\left(\big|\{i\in\N : \mathrm{Fr}_n\left((u,i)\right) \ge \alpha \mathrm{Fr}_n(u)\}\big|\,>\, k \, \mid \, \mathrm{Fr}_n(u) =\ell\right)\ \le \ 0.95^k.
    \]
    \begin{proof} 
        We analyze the splitting as in \eqref{eq:uniform-splitting}.
        We track for each child of $u$ if its fringe-tree size exceeds slightly less than its expectation: let
        \[
        Y_j=\Ind{X_j\ge R_j/2}, \qquad Y^{(k)}=\sum_{j=1}^k Y_j.
        \]
        Observe that $\p(Y_j=1 \mid Y_1,\ldots,Y_{j-1})\ge 1/2$ as $X_j$ is uniform on $\{1,\ldots, R_j\}$. Moreover, each time $Y_j=1$, the remaining mass halves, i.e.,  $R_{j+1}\le R_j/2$. Hence, after $m$ halvings the remaining mass is at most $R_12^{-m}$. In particular, if the number of halvings among the first $k$ children exceeds $\log_2(1/\alpha)$, then the remaining mass before child $k+1$ is less than $\alpha \mathrm{Fr}_n(u)$, and it is impossible to have a $(k+1)$-st child of size at least $\alpha \mathrm{Fr}_n(u)$, i.e., $R_{k+1} < \alpha \mathrm{Fr}_n(u)$.
        So, if there are at least $k+1$ children of $u$ with fringe-tree size at least $\alpha \mathrm{Fr}_n(u)$, we must have that $Y^{(k)}\le \log_2(1/\alpha)$. Since $Y^{(k)}$ stochastically dominates a binomial random variable with $p=1/2$, we find for $k\ge 3\log_2(1/\alpha)$,
        \begin{multline*}
        \p\left(\big|\{i\in\N : \mathrm{Fr}_n\left((u,i)\right) \ge \alpha \mathrm{Fr}_n(u)\}\big|>k\, \mid \, \mathrm{Fr}_n(u) =\ell \right) \le \p\big(Y^{(k)}\le \log_2(1/\alpha)\big)
        \\
        \le \p\left(\mathrm{Bin}(k, 1/2)\le k/3\right)\le\exp\left(-k\Lambda_{1/2}(\tfrac{1}{3})\right)\le 0.95^k,
        \end{multline*}
        where $\Lambda_{1/2}$ is the large deviations rate function of a Bernoulli(1/2)-distribution, see \eqref{eq:large-dev}. 
    \end{proof}
\end{lemma}

Using Lemmas \ref{lem:fringe tree count} and \ref{lem:stochastic_domination}, we prove the main lemma of this section, that we recall below.

\mixedlarge*

    \begin{proof}
        We first split the event inside the probability by considering dyadic intervals for the potential fringe-tree size of $u$: 
        \[
        \begin{aligned}
        \big\{(u,i)\in A_n^\prime: \mathrm{Fr}_n(u)&\mathrm{Fr}_n\left((u,i)\right) > (n/t)^2\big\}\\
        &\subseteq
        \bigcup_{j=0}^{\lfloor\log_2t\rfloor}\big\{(u,i)\in A_n^\prime: \mathrm{Fr}_n(u)\in\big(2^j\tfrac{n}{t}, 2^{j+1}\tfrac{n}{t}\big], \mathrm{Fr}_n((u,i))\ge 2^{-(j+1)}\tfrac{n}{t}\big\}.
        \end{aligned}
        \]
        Assuming that for each $j\in\{0,\ldots, \lfloor\log_2t\rfloor\}$, 
        \[
        \big|\big\{(u,i)\in A_n^\prime: \mathrm{Fr}_n(u)\in\big(2^j\tfrac{n}{t}, 2^{j+1}\tfrac{n}{t}\big], \mathrm{Fr}_n\left((u,i)\right)\ge 2^{-(j+1)}\tfrac{n}{t}\big\}\big| \le (C/5)\cdot t\cdot 0.8^j, 
        \]
        we have 
        \[
        \big|\big\{(u,i)\in A_n^\prime: \mathrm{Fr}_n(u)\mathrm{Fr}_n\left((u,i)\right) > (n/t)^2\big\}\big| \le \sum_{j=0}^{\lfloor\log_2t\rfloor }(C/5)\cdot t\cdot 0.8^j \le Ct.
        \]
        By a union bound, and using that $2^{-(j+1)}\tfrac{n}{t}\ge 2^{-2(j+1)}\mathrm{Fr}_n(u)$ if $\mathrm{Fr}_n(u)\in (2^j \tfrac{n}{t}, 2^{j+1}\tfrac{n}{t}]$, 
        \[
        \begin{aligned}
        &\p\left(\big|\big\{(u,i)\in A_n^\prime: \mathrm{Fr}_n(u)\mathrm{Fr}_n\left((u,i)\right) > (n/t)^2\big\}\big|>Ct\right)\\ 
        &\le 
        \sum_{j=0}^{\lfloor \log_2 t\rfloor}\p\left(\big|\big\{(u,i)\in A_n^\prime: \mathrm{Fr}_n(u)\in\left(2^j\tfrac{n}{t}, 2^{j+1}\tfrac{n}{t}\right], \mathrm{Fr}_n\left((u,i)\right)\ge 2^{-(j+1)}\tfrac{n}{t}\big\}\big| > (C/5) 0.8^jt\right) \\
         &\le \sum_{j=0}^{\lfloor \log_2 t\rfloor}\p\left(\big|\big\{(u,i)\in A_n^\prime: \mathrm{Fr}_n(u)\in\left(2^j\tfrac{n}{t}, 2^{j+1}\tfrac{n}{t}\right], \mathrm{Fr}_n\left((u,i)\right)\ge 2^{-2(j+1)}\mathrm{Fr}_n(u)\big\}\big| > (C/5)  0.8^jt\right)
        .\end{aligned}
        \]
        We bound the summands individually using Lemmas \ref{lem:fringe tree count} and \ref{lem:stochastic_domination} and the fringe-tree splitting construction. Define
        \begin{equation}\label{eq:rj}
        J_j=N_n^\mathrm{Fr}\big(2^j\tfrac{n}{t}, 2^{j+1}\tfrac{n}t\big], \qquad r_j:=\lceil\sqrt{C}0.6^jt\rceil.
        \end{equation}

    Similarly to the proof of Lemma \ref{lem:fringe tree count}, for $0 \leq i \leq n$, let $\{ u^i_1,\ldots,u^i_{n^i}\}= [n]^i$ in lexicographic ordering  and let $\{v_1,\ldots,v_{K}\} = \bigcup_{i=1}^n [n]^i$ with $K=\sum_{i=0}^{n} n^i$ be so that the level of $v_j$, i.e., the unique value $i\in [n]$ s.t. $v_j \in [n]^i$, is non-decreasing in $j$, and so that each block $[n]^i$ is in the lexicographic ordering.
        For $1 \leq \ell \leq K$, define the sets $F_1(\ell)$ and $F_2(\ell)$ by
        \begin{align*}
            F_1(\ell)& := \left\{ u \in \{v_1,\ldots,v_\ell\} : \mathrm{Fr}_n \left(u\right)\in \left( 2^j \frac{n}{t}, 2^{j+1} \frac{n}{t} \right] \right\}, \\
            F_2(\ell)& := \left\{(u,i)\in \mathcal{U} : u \in \{v_1,\ldots,v_\ell\} , \mathrm{Fr}_n \left(u\right)\in \left( 2^j \frac{n}{t}, 2^{j+1} \frac{n}{t} \right], \mathrm{Fr}_n\left((u,i)\right) \geq 2^{-2(j+1)} \mathrm{Fr}_n(u) \right\}.
        \end{align*}
        Let $\cF_\ell$ be the $\sigma$-algebra $\cF_\ell=\sigma\left( \mathrm{Fr}_n((u,i)) : u \in \{\emptyset,v_1,\ldots,v_\ell\}, i\in \N \right)$.
        Let $u = v_{\ell} \in \cU$ with fringe-tree size  $\mathrm{Fr}_n(u)\in\big(2^j\tfrac{n}{t}+1, 2^{j+1}\tfrac{n}{t}\big]$. Conditional on $\cF_{\ell-1}$, the sizes of the fringe trees $\left(\mathrm{Fr}_n((u,i))\right)_{i\geq 1}$ have the same distribution as the sizes of these fringe trees, when conditioned on $\mathrm{Fr}_n(u)$. (Note that $\mathrm{Fr}_n(u)$ is measurable with respect to $\cF_{\ell-1}$). Lemma \ref{lem:stochastic_domination} with $\alpha=2^{-2(j+1)}$ implies that its number of children with fringe-tree size at least  $2^{-2(j+1)}\mathrm{Fr}_n(u)$ has an exponentially decaying tail: for every $m\ge 0$,
        \begin{multline*}
            \p\left(\big|\{i: \mathrm{Fr}_n\left((u,i)\right) \ge 2^{-2(j+1)}\mathrm{Fr}_n(u)\}\big|>6(j+1)+m \ \big|\ \cF_{\ell-1} \right)
            \\
            =
            \p\left(\big|\{i: \mathrm{Fr}_n\left((u,i)\right) \ge 2^{-2(j+1)}\mathrm{Fr}_n(u)\}\big|>6(j+1)+m \ \big|\ \mathrm{Fr}_n(u) \right)
            \leq 0.95^{6(j+1)+m} \le 0.95^m
        \end{multline*}
        Hence, conditional on $\cF_{\ell-1}$, the number of children of $u$ with $\mathrm{Fr}_n\left((u,i)\right) \ge 2^{-2(j+1)}\mathrm{Fr}_n(u)$
        is stochastically dominated by a shifted geometric random variable $6(j+1)+Y$, where $Y\sim\mathrm{Geo}(0.05)$. Using this inductively, we see that there exist i.i.d. random variables $Y_1, Y_2,\ldots$ with $Y_i \sim \mathrm{Geo}(0.05)$ so that
        \begin{equation*}
            |F_2(\ell)| \leq \sum_{k=1}^{|F_1(\ell)|} \left( 6(j+1) + Y_k \right).
        \end{equation*}
        for all $1 \leq \ell \leq K$. Thus, we obtain that 
        \begin{align}
            &\p\left(\big|\big\{(u,i)\in A_n^\prime: \mathrm{Fr}_n(u)\mathrm{Fr}_n\left((u,i)\right) > (n/t)^2\big\}\big|>Ct\right)\nonumber\\
            &\qquad\le
            \sum_{j=0}^{\lfloor \log_2 t\rfloor}\p\bigg(\sum_{k=1}^{J_j} \left(Y_{k}+6(j+1)\right)>(C/5) 0.8^j t\bigg)\nonumber \\
            &\qquad\le
            \sum_{j=0}^{\lfloor \log_2 t\rfloor}\Bigg(\p\left(J_j>r_j\right)+ \p\bigg(\sum_{k=1}^{r_j} \left(Y_{k}+6(j+1)\right)>(C/5) 0.8^j t\bigg)\Bigg) 
            .\label{eq:lem-mixed-pr1}
        \end{align}
        To bound the first summands, we use Lemma \ref{lem:fringe tree count} for $\ell=2^j\tfrac{n}{t}$ and $r=r_j$. This requires that 
        \[r_j = \lceil\sqrt{C}t0.6^j\rceil\ge 20\cdot  2^{-j}t,\]
        which is satisfied for all $j$ provided that $C$ is a sufficiently large constant. As $\log(0.6)/\log(2)=-0.73...>-0.75$, we obtain for $C$ large enough
        \begin{align}
            \sum_{j=0}^{\lfloor \log_2 t\rfloor}\p\left(J_j>r_j\right)&\le 27\sum_{j=0}^{\lfloor \log_2 t\rfloor}\exp\left(-0.04 \sqrt{C}0.6^jt\right)\nonumber\\&\le 27(1+\lfloor\log_2t\rfloor) \exp\left(-0.04\sqrt{C}t^{1+\log(0.6)/\log(2)}\right)\le \tfrac{1}{2}\exp\left(-t^{1/4}\right).\label{eq:lem-mixed-pr2}
        \end{align}
        We proceed to the second term within the sum in \eqref{eq:lem-mixed-pr1}. As $Y_{k}+6(j+1)\le 7(j+1)Y_{k}$, $r_j \leq 2 \sqrt{C}0.6^j t$, and $\E\left[ \exp(0.04 Y_k) \right]<\infty$, we get that 
        \begin{multline*}
         \p\bigg(\sum_{k=1}^{r_j} \left(Y_{k}+6(j+1)\right)>(C/5) 0.8^j t\bigg) \le \p\bigg(\sum_{k=1}^{r_j} Y_{k}>r_j\frac{C0.8^j}{35r_j(j+1)} t\bigg).
         \\
         \leq
         \left(\E\big[\exp\left(0.04Y_{k}\right)\big]\exp\left(-\frac{0.04 C 0.8^j }{35 \cdot 2 \sqrt{C} 0.6^j (j+1) t } t \right)\right)^{r_j} \leq \left( \exp\left(- 1 \right) \right)^{r_j} \leq \exp\left(- \sqrt{C} t 0.6^j \right) .
        \end{multline*}
        Combining this bound with \eqref{eq:lem-mixed-pr1} and \eqref{eq:lem-mixed-pr2}, we obtain for $t\ge1$ and $C$ sufficiently large,
        \begin{multline*}
        \p\left( \left|\left\{(u,i)\in A_n^\prime: \mathrm{Fr}_n(u)\mathrm{Fr}_n\left((u,i)\right) > (n/t)^2\right\}\right|>Ct \right)\\ 
        \le 
        \tfrac{1}{2}\exp\left(-t^{1/4}\right)+ \sum_{j=0}^{\lfloor\log_2t\rfloor} \exp\left(- \sqrt{C} t 0.6^j \right)
        \le 
        \exp\left(-t^{1/4}\right). \qedhere
        \end{multline*}
    \end{proof}

We finish the section with an immediate corollary,  which we use in Section~\ref{sec:Jordan}.

\begin{corollary}\label{coro:fringe tree counts large} There exists a constant $C\ge 1$ such that for all  $t > 0$ and $n\in\N$
	\begin{equation*}
		\p \left( \big| \big\{ v\in [n]  : \mathrm{Fr}_{n}(v)  > n/t \big\} \big| >  Ct \right) \leq  \exp \left( 1 -t^{1/4} \right).
	\end{equation*}
\begin{proof}
The statement is clear for $t\leq 1$. So let $t> 1$.
Every $v\in \{u\in [n]:\mathrm{Fr}_n(u)>n/t\}$ other than the root has a parent $u\in \cU$ and a child-index $i$ such that $(u,i)\in A_n^\prime$ and $\mathrm{Fr}_n(u)\mathrm{Fr}_n((u,i))>\,(n/t)^2$. Thus
\[
|\{v\in [n]:\mathrm{Fr}_n(v)>n/t\}\big|\le 1 + \big|\{(u,i)\in A_n^\prime:\ \mathrm{Fr}_n(u)\mathrm{Fr}_n((u,i)) > (n/t)^2\}\big|.
\]
Lemma \ref{lem:mixed large} implies the probability that the right-hand term exceeds $Ct+1$ is at most $\exp(-t^{1/4})$, which yields the claim by increasing the constant $C$.
\end{proof}
\end{corollary}

\subsection{Upper tail: upper bound}\label{sec:Upper_tail_UB_J2}
We combine the lemmas from the Sections \ref{subsec:jordan2c-lower} and \ref{subsec:fringe} to prove the first part of Theorem \ref{thm:main}.

\begin{lemma}\label{lem:Jordan2_uppertail_UB}
    There exists a constant $C>0$ such that for any $n\ge 1$, any $v\in [n]$, and any $S\ge 1$,
    \begin{align}
        \p \left( \hs_{2}(v) \geq S v \right) & \leq C\frac{ 1+\log S}{S^2}.
    \end{align}
\end{lemma}

\begin{proof}
    We follow the reasoning from the outline at the beginning of Section \ref{sec:NBjordan}.  If $\hs_2(v)\ge Sv$, there are at least $Sv$ vertices with Jordan-2 centrality at least as large as $\phi_n^\ssup{2}(v)$. So,
\begin{align}
    \p\left(\hs_2(v)\ge Sv\right) \leq \p\left(\big|\big\{u\in [n]: \phi_n^\ssup{2}(u)\ge \phi_n^{\ssup{2}}(v)\big\}\big|\ge Sv\right).
\end{align}
We consider two cases: either $\phi_n^\ssup{2}(v)$ is unlikely small, or there are many vertices with large Jordan-2 centrality. Let $C_1$ be the constant $C$ from Lemma \ref{lem:mixed large}. We bound
\begin{align*}
    \p\left(\hs_2(v)\ge Sv\right)&\le \p\left(\phi_n^\ssup{2}(v)\le \left(\frac{2C_1n}{Sv}\right)^2\right) +  \p\left(\Big|\Big\{u\in[n]: \phi_n^\ssup2(u)> \left(\frac{2C_1n}{Sv}\right)^2\Big\}\Big| > Sv\right) 
    .
\end{align*}
 We use that $\phi_n^\ssup{2}(v)\le \mathrm{Fr}_n(v)\mathrm{Fr}_n(\pa(v))$ by \eqref{eq:upper-bound-jordan2} for $v\ge2$. Hence, 
\begin{align*}
    & \p\left(\hs_2(v)\ge Sv\right)
    \\
    &\le \p\left(\phi_n^\ssup{2}(v)\le \left(\frac{2C_1n}{Sv}\right)^2\right) +  \p\left(\Big|\Big\{u\in[n]\setminus\{1\}: \mathrm{Fr}_n(u)\mathrm{Fr}_n(\pa(u))> \left(\frac{2C_1n}{Sv}\right)^2\Big\}\Big| > Sv-1\right) 
    \\
    &\le \p\left(\phi_n^\ssup{2}(v)\le \left(\frac{2C_1n}{Sv}\right)^2\right) +  \p\left(\Big|\Big\{u\in[n]\setminus\{1\}: \mathrm{Fr}_n(u)\mathrm{Fr}_n(\pa(u))> \left(\frac{2C_1n}{Sv}\right)^2\Big\}\Big| > C_1 \frac{Sv}{2C_1}\right)
\end{align*}
where the last inequality holds for $S\geq 2$.
Applying inequality \eqref{eq:lem2.1 upper bound} from Lemma \ref{lem:boosted small} with $t=S/(2C_1)$ to the first term and Lemma \ref{lem:mixed large} with $t=Sv/(2C_1)$ to the second term, we get that for $S\geq 2C_1 \vee 2$,
\begin{equation*}
    \p\left(\hs_2(v)\ge Sv\right) \le C_2\cdot  4C_1^2\frac{1+\log\left( S/(2C_1)\right)}{S^2} + \exp\left(-(2C_1)^{-1/4}S^{1/4}\right), 
\end{equation*}
where $C_2$ is the constant $C$ from Lemma \ref{lem:boosted small}.
As a result, there exists a constant $C>0$ such that for all $S\geq 2C_1 \vee 2$, 
\[
\p\left(\hs_2(v)\ge Sv\right)\le C\frac{1+\log S}{S^2}.
\]
Increasing the constant $C$ also implies that the same inequality holds for all $S\geq 1$.
\end{proof}

\subsection{Upper tail: lower bound}\label{sec:Upper_tail_LB_J2}

We proceed with the analysis of the lower bound of the upper tail, i.e., the first inequality in \eqref{jordan2:overest_combined}.

\begin{lemma}\label{lem:Jordan2_upper_tail_lower_bound}
        For every $\eps>0$ there exist positive constants $c=c(\eps), \delta=\delta(\eps)$ such that for all $n\geq 3$, all $v\in[ n/4]$, and all $S\in[1, \delta \left(n/v\right)^{1-\eps}]$
    \begin{equation}\label{eq:lem2.15 (1)}
        \mathbb{P}\left(\hs_2(v)\geq Sv\right) \geq  c \frac{1+\log(S)}{S^2}.
    \end{equation}
    Furthermore, there exist positive constants $S_0,c^\prime, \delta$ such that for all $n\geq 3$, all $v \leq n/4$ and all $S\in[S_0, \delta \tfrac{n}{v}]$,
    \begin{equation}\label{eq:lem2.15 (2)}
        \mathbb{P}\left(\hs_2(v)\geq Sv\right) \geq  c^\prime \frac{1}{S^2}.
    \end{equation}
\end{lemma}

Before going to the proof of this lemma, we state and prove an auxiliary lemma on the global event that few vertices have a typical fringe-tree size. It constitutes a counterpart of Lemma \ref{lem:mixed large}. 
\begin{lemma}\label{claim:S over 64v}
        There exists a  constant $C>0$ such that for any $n\geq 1$, and $t \leq n/16$,
        \begin{equation}\label{eq:S over 64v}\begin{aligned}
        \p\Big(\Big|\Big\{u \in [n] &: \phi_n^\ssup{2}(u) > \Big( \frac{n}{64t}\Big)^2 \Big\}\Big|\le t\Big)\\ &\le
            \p\left(\Big|\left\{u \in [n] : \phi_n(u) >  \frac{n}{64t} \right\}\Big|\le t\right) \leq C \exp \Big( -  \frac{t}{8} \Big) .\end{aligned}
        \end{equation}
    \end{lemma}

    \begin{proof}
    As $\phi_n^\ssup{2}(u)\ge \phi_n(u)^2$ by \eqref{eq:NNBJordanDef}, the first inequality in \eqref{eq:S over 64v} follows. We turn to the inequality on the second line in \eqref{eq:S over 64v}.

    It suffices to consider the case when $t \in \N$, which we will assume for the rest of the proof. Since $t\le n/16$, we have $\left\{ 8t , \ldots, 16t \right\} \subset [n]$.
    For a vertex $u \in \left\{ 8t , \ldots, 16t \right\}$ with $\mathrm{Fr}_n(u) \leq \frac{n}{2}$ we have $\mathrm{Fr}_n^{>16t}(u) \leq \mathrm{Fr}_n(u) = \phi_n(u)$. Hence, 

    \begin{align}
         \notag\p&\left(\Big|\left\{u \in [n] : \phi_n(u) >  \frac{n}{64t} \right\}\Big|\le t\right)\\
         &\notag\hspace{2cm}\leq \p\left(\Big|\left\{u \in  \left\{ 8t , \ldots, 16t \right\} : \phi_n(u) > \frac{n}{64t}\right\}\Big|\le t\right) \\
        &\hspace{2cm} \notag
        \le  \p\left(\exists u\in  \left\{ 8t , \ldots, 16t \right\} : \mathrm{Fr}_n(u)\geq \frac{n}{2}\right) 
        \\
        & \notag\hspace{2cm}
        \hspace{1.2cm}
        +  \p\left(\Big|\left\{u\in  \left\{ 8t , \ldots, 16t \right\} : \mathrm{Fr}_n(u) > \frac{n}{64t}\right\}\Big|\le t\right) \\
        & \label{eq:union bound term 1}
        \hspace{2cm}\le  \p\left(\exists u\in  \left\{ 8t , \ldots, 16t \right\} : \mathrm{Fr}_n(u)\geq \frac{n}{2}\right) 
        \\
        & \label{eq:union bound term 2}
        \hspace{2cm}\hspace{1.2cm}
        +  \p\left(\Big|\left\{u\in  \left\{ 8t, \ldots, 16t \right\} : \mathrm{Fr}_n^{>16t}(u) > \frac{n}{64t}\right\}\Big|\le t\right),
    \end{align}
    where $\mathrm{Fr}_n^{>16t}(u)$ denotes the fringe tree witnessed after time $16t$, see Definition \ref{def:fringe}.
    The first summand in \eqref{eq:union bound term 1} can be bounded by a union bound and Lemma \ref{lem:proba_large_fringe}:
    \begin{multline}
        \p\left(\exists u\in  \left\{ 8t , \ldots, 16t \right\} : \mathrm{Fr}_n(u)\geq \frac{n}{2}\right) 
        \leq \sum_{u =  8t  }^{16t} \p \left( \mathrm{Fr}_n(u)\geq \frac{n}{2} \right)
        \\ \label{large fringe}
        =
        \sum_{u =  8t  }^{16t} \p \left( \mathrm{Fr}_n(u)\geq \frac{u}{2} \frac{n}{u} \right)
        \overset{\eqref{eq:proba_fringe-large}}{\leq} \sum_{u =  8t  }^{16t} 2 e^2 \exp\left( - \frac{u}{2} \right) \leq \frac{2 \exp \left( \frac{5}{2} - 4t \right) }{\sqrt{e}-1} \leq 38 \exp \left( -4t \right).
    \end{multline}
    For the second summand, \eqref{eq:union bound term 2}, observe that for each $u \in \left\{ 8t , \ldots , 16t \right\}$, we have
    \begin{align*}
        \p \left( \mathrm{Fr}_n^{>16t}(u) >\frac{n}{64 t} \right)
        = 
        1 -
        \p \left( \mathrm{Fr}_n \left( 16t \right) \leq  \frac{n}{4 \cdot 16t} \right)
        \overset{\eqref{eq:fringe-ub}}{\geq} 1 - \frac{2}{4} = \frac{1}{2}.
    \end{align*}
    Linearity of expectation implies that
    \begin{align*}
        \E \left[ \Big| \left\{ u \in \left\{ 8t , \ldots , 16t \right\} : \mathrm{Fr}_n^{>16t}(u) >  \frac{n}{64 t} \right\}\Big| \right] \geq 8t \frac{1}{2} = 4t.
    \end{align*}
    Since the terms $\mathrm{Fr}_n^{>16t}(u),  u \in \left\{ 8t, \ldots , 16t \right\}$, are negatively associated by Lemma \ref{lem:beta-bin-fringe}(2), we can apply the Chernoff bound for negatively associated Bernoulli random variables (see \cite[Theorem 4.5]{mitzenmacher2017probability} and \cite[Proposition 5]{dubhashi1996balls}) and get that 
    \begin{equation*}
        \p\left(\Big|\left\{u\in  \left\{ 8t , \ldots, 16t  \right\} : \mathrm{Fr}_n^{>16t}(u) >  \frac{n}{64t} \right\}\Big|\le t\right)
        \leq 
        \exp \left( - t/8 \right).\qedhere
    \end{equation*}
    \end{proof}

We are now ready to prove Lemma \ref{lem:Jordan2_upper_tail_lower_bound}.

\begin{proof}[Proof of Lemma \ref{lem:Jordan2_upper_tail_lower_bound}]
    Let $\eps>0$. By \eqref{eq:lem2.1 lower bound} there exists a constant $c^\star=c^\star(\eps)$ such that
    \begin{align*}
        \p \left( \phi_n^{\ssup{2}}(v)
    \leq
    \left( \frac{n}{t v} \right)^2 \right) \geq c^\star \frac{\log t}{t^2}
    \end{align*}
    for all $t \in \big( \sqrt{2}, \tfrac{n}{v \sqrt{3}} \wedge (n/v)^{1-\eps}\big)$. 
    Let $n\ge3$, $v\in[n/4]$ and $S\in\big[1, \tfrac{1}{1024\sqrt{3}} (n/v)^{1-\eps} \big]$. We start by noting that to have $\hs_2(v)\ge Sv$, it suffices that at least $Sv$ vertices have Jordan-$2$ centrality larger than $\phi_n^{(2)}(v)$. So,
    \begin{align}
        \p\left(\hs_2(v) \ge Sv\right)
        &\notag \geq 
        \p\left(\big|\big\{u\in [n]: \phi_n^\ssup{2}(u) >\phi^\ssup{2}_n(v)\big\}\big|\ge Sv\right) \\
        &\notag \geq
        \p\left(\left\{\phi_n^\ssup{2}(v)\leq \left(\tfrac{1}{64S}\tfrac{n}{v}\right)^2\right\} 
        \cap
        \left\{\big|\big\{u\in [n]: \phi_n^\ssup{2}(u) >\left(\tfrac{1}{64S}\tfrac{n}{v}\right)^2\big\}\big|\ge Sv\right\}\right)\\
        & \label{eq:Jordan2_upper_tail_LB_decompo} \geq \p\left(\phi_n^\ssup{2}(v)\leq \left(\tfrac{1}{64S}\tfrac{n}{v}\right)^2\right) 
        -
        \p\left(\big|\big\{u\in [n]: \phi_n^\ssup{2}(u) > \left(\tfrac{1}{64S}\tfrac{n}{v}\right)^2 \big\}\big|< Sv\right).
    \end{align}
    The first term of \eqref{eq:Jordan2_upper_tail_LB_decompo} is lower bounded using inequality \eqref{eq:lem2.1 lower bound} from Lemma \ref{lem:boosted small} evaluated at $t=64S$. The second term is bounded from above by Lemma \ref{claim:S over 64v} $t=Sv$ (our assumed upper bound on $S$ implies that  $t=Sv\le n/16$). As a result, there exist constants $C,c>0$ such that 
\begin{equation}\nonumber
    \begin{aligned}
        \p\left(\hs_2(v) \ge Sv\right)
        & \geq c^\star \frac{1+\log S}{S^2} - C\exp\big(-vS/8\big) \geq c \frac{1+\log S}{S^2},
    \end{aligned}
    \end{equation}
    where the last inequality holds for some constant $c>0$ small enough and all $S\geq M$, where $M$ is a large enough constant. Let $\delta = \tfrac{1}{M 1024 \sqrt{3}}$. This shows inequality \eqref{eq:lem2.15 (1)} for $M \leq S \leq \tfrac{1}{1024 \sqrt{3}} (n/v)^{1-\eps}$. When $1 \leq S \leq M$, then there is nothing to show when $\delta (n/v)^{1-\eps} < 1$. But when $1\leq \delta (n/v)^{1-\eps}$, we have that
    \begin{align*}
        \p\left(\hs_2(v) \ge Sv\right)
        \geq
        \p\left(\hs_2(v) \ge M v \right) \geq c \frac{1+\log M}{M^2} > 0,
    \end{align*}
    since $M \leq \tfrac{1}{1024 \sqrt{3}} (n/v)^{1-\eps}$. This shows \eqref{eq:lem2.15 (1)}, by changing the constant $c$.
    
    The proof of \eqref{eq:lem2.15 (2)} follows analogously, since
    \begin{align*}
        \p\left(\hs_2(v) \ge Sv\right)
        &\notag \geq 
        \p\left(\big|\big\{u\in [n]: \phi_n^\ssup{2}(u) >\phi^\ssup{2}_n(v)\big\}\big|\ge Sv\right) \\
        &\notag \geq \p\left(\phi_n^\ssup{2}(v)\leq \left(\tfrac{1}{64S}\tfrac{n}{v}\right)^2\right) 
        -
        \p\left(\big|\big\{u\in [n]: \phi_n^\ssup{2}(u) > \left(\tfrac{1}{64S}\tfrac{n}{v}\right)^2 \big\}\big|< Sv\right)
        \\
        &
        \overset{\eqref{eq:lem2.1 lower bound version2}}{\geq} c \frac{1}{(64S)^2} - C \exp \left( -vS/8 \right) \geq c^\prime \frac{1}{S^2},
    \end{align*}
    where $c$ is the constant from inequality \eqref{eq:lem2.1 lower bound version2}. The second-to-last inequality holds for $1 \leq 64 S \leq \tfrac{1}{\sqrt{3}}(n/v)$, and the last inequality holds for all $S\geq S_0$, where $S_0$ is a large enough constant. \qedhere 
\end{proof}

\subsection{Lower tail}\label{sec:Lower_tail_UB-J2}
We turn our attention to the lower tail, which analyses the probability that the Jordan-2 ordering underestimates the arrival time of a vertex. As for the upper tail,  we split the event $\left\{ \hs_2(v)\leq \tfrac{1}{S}v \right\}$ into a local event on the Jordan-2 centrality of $v$ (analysed in Section \ref{subsec:jordan2c-upper}), and a global event, for which we have established the prerequisites along the way when controlling the upper tail. 
The following lemma constitutes the upper bound in \eqref{jordan2:underest_combined} in Theorem~\ref{thm:main}.

\begin{lemma}[Upper bound]\label{lem:Jordan2_lower_tail}
    For the Jordan-$2$ ordering $\hs_2$, there exists a positive constant $C$ such that for any $n\geq 1$, any vertex $v \in[n/4]$ and any $S\geq 1$, 

    \begin{equation}\label{eq:lem:jordan2lowertaillowerbound}
        \mathbb{P}\Big(\hs_2(v)\leq \frac{1}{S}v\Big) \leq  \frac{C}{S^2}.
    \end{equation}
\end{lemma}

\begin{proof}
We give two (slightly) different proofs, depending whether $S < v^{3/4}$ or $S \geq v^{3/4}$.
We start the proof for the case where $ S < v^{3/4}$.
Without loss of generality, let $S$ be such that $\frac{v}{S} \in \N$, and assume that $S\ge 32$; inequality \eqref{eq:lem:jordan2lowertaillowerbound} clearly holds for $S\leq 32$ and $C$ large enough.
If $\hs_2(v)\leq \tfrac{1}{S} v$, there are no more than $\tfrac{1}{S}v$ vertices with Jordan-2 centrality \emph{strictly} larger than $\phi_n^\ssup{2}(v)$. So,

    \begin{align*}
        \p\Big(\hs_2(v)\le \tfrac{1}{S}v \Big) &
        \leq \p\Big(\big|\big\{u\in[n]: \phi_n^\ssup{2}(u) >
        \phi_n^\ssup{2}(v)\big\}\big|\le \tfrac{1}{S}v\Big)
        \\
        &
        \leq \p\Big(\big|\big\{u\in[n]: \phi_n^{\ssup{2}}(u) >
        \left( S\tfrac{n}{64v} \right)^2 \big\}\big|\le \tfrac{1}{S}v\Big) + \p\Big(\phi_n^\ssup{2}(v)\ge \big(S\tfrac{n}{64v}\big)^2\Big).
    \end{align*}
    The first probability on the right-hand side is bounded from above by Lemma \ref{claim:S over 64v} for $t=v/S$, which is at most $n/32$, by the assumption $S\ge 32$. The second probability is bounded from above using Lemma \ref{lem:Fringe_fringe_parent_large} with $S_{\ref{lem:Fringe_fringe_parent_large}}=(S/64)^2$. 
    Combined, we obtain that there exist constants $C^\prime, C$ so that
    \[
    \p\Big(\hs_2(v)\le \tfrac{1}{S}v \Big)
        \leq C^\prime \exp\Big(-\frac{v}{8S}\Big) + \frac{64^2C^\prime}{S^2} \leq 
        C^\prime \exp \Big(-\frac{S^{1/3}}{8}\Big) + \frac{64^2C^\prime}{S^2} \leq \frac{C}{S^2},
    \]
    where we used the assumption $S < v^{3/4}$ for the second inequality.

\smallskip
    Next, we consider the case where $v \geq S\geq v^{3/4}$. For the proof of \eqref{eq:lem:jordan2lowertaillowerbound}, it suffices to consider large values of $v$, say $v^{3/4} \geq 32$. Indeed, when $v^{3/4} \leq 32$, then \eqref{eq:lem:jordan2lowertaillowerbound} is trivially satisfied for $C\geq 32^{8/3}$, since the left-hand side of \eqref{eq:lem:jordan2lowertaillowerbound} equals $0$ for $S > 32^{4/3} \geq v$.
    Using the elementary inequality
    \begin{equation*}
        \phi_n^{\ssup{2}}(v) \leq n \phi_n(v) \leq n \mathrm{Fr}_n(v), 
    \end{equation*}
    we see that
    \begin{align*}
        \left\{ \hs_2(v) \leq \frac{v}{S} \right\}  & \subseteq \left\{ \phi_n^{\ssup{2}}(v) \geq \Big(\frac{n}{64v^{1/4}}\Big)^2 \right\} \cup \left\{ \Big| \Big\{ u \in [n] : \phi_n^{\ssup{2}}(u) \geq \Big(\frac{n}{64v^{1/4}}\Big)^2 \Big\} \Big| \leq \frac{v}{S} \right\}
        \\
        &
        \subseteq \left\{ \mathrm{Fr}_n(v) \geq \frac{n}{2^{12}\sqrt{v}} \right\} \cup \left\{ \Big| \Big\{ u \in [n] : \phi_n^{\ssup{2}}(u) \geq \Big(\frac{n}{64v^{1/4}}\Big)^2 \Big\} \Big| \leq v^{1/4} \right\} .
    \end{align*}
    To bound the probability of the first event on the right-hand side, we apply inequality \eqref{eq:proba_fringe-large} from Lemma \ref{lem:proba_large_fringe}. For the second event, we use Lemma \ref{claim:S over 64v} for $t=v^{1/4}$, which applies for $v^{1/4}\le n/16$, which holds since $v^{3/4}\ge 32$. We obtain
    \begin{align*}
        \p \left( \hs_2(v) \leq \frac{v}{S} \right)
        &
        \leq
        \p \left( \mathrm{Fr}_n(v) \geq \frac{\sqrt{v}}{2^{12}} \frac{n}{v} \right) + \p \left( \Big| \Big\{ u \in [n] : \phi_n^{\ssup{2}}(u) \geq \Big(\frac{n}{64v^{1/4}}\Big)^2 \Big\} \Big| \leq v^{1/4} \right)
        \\
        &
        \leq 2 e^2 \exp\left( - \frac{\sqrt{v}}{2^{12}} \right) + C \exp \left( - \frac{v}{8 v^{3/4}} \right) \leq \frac{C^\prime}{v^{2}} \leq \frac{C^\prime}{S^2},
    \end{align*}
    for some constant $C'$, where we used the assumption $S \leq v$ for the last inequality.
    
    Lastly, inequality \eqref{eq:lem:jordan2lowertaillowerbound} is clear for $S>v$, since the left-hand side of inequality \eqref{eq:lem:jordan2lowertaillowerbound} equals $0$ for $S>v$.
\end{proof}

\begin{remark}\label{rem:stretched-exp}
    The proof of \eqref{eq:lem:jordan2lowertaillowerbound} for the case $S \geq v^{3/4}$ actually shows that $\p\left( \hs_2(v) \leq v/S \right)$ is stretched-exponentially small in $S$. A slightly modified proof also shows this for any $S\geq v^{1/2+\eps}$, for any $\eps>0$. Thus, the restriction to $S \leq \delta \sqrt{v}$ in the lower bound in \eqref{jordan2:underest_combined} is essentially sharp.
\end{remark}

To conclude the analysis of the Jordan-2 ordering, we prove a lower bound on the lower tail establishing \eqref{jordan2:underest_combined}, complementing Lemma \ref{lem:Jordan2_lower_tail}.

\begin{lemma}[Lower bound]\label{lem:Jordan2_lower_tail_lower_bound}
        For the Jordan-$2$ ordering $\hs_2$, there exist positive constants $C$, $\delta$ such that for any $n\geq 1$, any vertex $v \in[n]$ and any $S\in[1, \delta\sqrt{v}]$ 

    \begin{equation}
        \mathbb{P}\Big(\hs_2(v)\leq \frac{1}{S}v\Big) \geq  \frac{C}{S^2}.
    \end{equation}
\end{lemma}



\begin{proof}
    Let us fix some $n\geq1$, $v\in  [n]$ and $S\geq 1$. We start by noting that to have $\hs_2(v)\le \tfrac{1}{S}v$, it suffices that no more than $v/S$ vertices have Jordan-$2$ centrality \emph{at least} $\phi_n^{(2)}(v)$. So,

    $$\p\left(\hs_2(v)\le \frac{1}{S}v\right) 
        \geq
        \p\left(\big|\big\{u \in [n]: \phi_n^\ssup{2}(u) \ge\phi^\ssup{2}_n(v)\big\}\big|\le v/S\right).$$
    Moreover, for any constant positive constant $C^\star$
    \begin{equation}\label{eq:Jordan_2_LT_LB_decompo}
    \begin{aligned}
        \p\left(\hs_2(v)\le \frac{1}{S}v\right) 
        &\ge
        \p\left(\left\{\phi_n^\ssup{2}(v) \ge \left(2C^\star S\tfrac{n}{v}\right)^2\right\} \cap \Big\{ \big|\big\{u \in [n] : \phi^\ssup{2}_n(u) \ge\phi^\ssup{2}_n(v)\big\}\big|\le v/S \Big\}\right)\\
        &\ge
        \p\left(\phi_n^\ssup{2}(v) \ge \left(2C^\star S\tfrac{n}{v}\right)^2\right) - \p\left(\big|\big\{u\in [n]: \phi^\ssup{2}_n(u) \ge \left(2C^\star S\tfrac{n}{v}\right)^2\big\}\big|> \frac{1}{S}v\right).
    \end{aligned}
    \end{equation}
    Choose $C^\star \in (0,\infty)$ large enough, as in Lemma \ref{lem:mixed large}, so that
    \begin{equation}\label{eq:lem2.10 ref}
        \p\left(\big|\big\{u\in [n] \setminus {1}: \mathrm{Fr}_n(u)\mathrm{Fr}_n(\pa(u)) \ge \left(\tfrac{n}{t}\right)^2\big\}\big|> C^\star t \right) \leq \exp \left( - t^{1/4} \right)
    \end{equation}
    for all $t\geq 1$. Let $\delta \leq \frac{1}{2 C^\star} $, so that $v/S \geq 2$ and $\tfrac{v}{2 C^\star S} \geq 1$. Applying the previous inequality for $t=\tfrac{v}{2 C^\star S}$, we get that for $S \leq \delta v$,
    \begin{align*}
    &  \p\left(\big|\big\{u\in [n]: \phi^\ssup{2}_n(u) \ge \left(2 C^\star S\tfrac{n}{v}\right)^2\big\}\big|> \frac{1}{S}v\right)
    \\
    & \hspace{3cm}
        \leq 
        \p\left(\big|\big\{u\in [n] \setminus {1}: \mathrm{Fr}_n(u)\mathrm{Fr}_n(\pa(u)) \ge \left(2 C^\star S \tfrac{n}{v}\right)^2\big\}\big|> \frac{1}{S}v - 1\right)
        \\
        &\hspace{3cm}
        \leq \p\left(\big|\big\{u\in [n] \setminus {1}: \mathrm{Fr}_n(u)\mathrm{Fr}_n(\pa(u)) \ge \left(2 C^\star S\tfrac{n}{v}\right)^2\big\}\big|> C^\star \frac{v}{2C^\star S} \right)\\ &
        \hspace{3cm}
        \leq \exp \left( - \left( \tfrac{v}{2 C^\star S} \right)^{1/4} \right) ,
    \end{align*}
    where the first inequality holds since $\phi_n^\ssup{2}(u)\leq \mathrm{Fr}_n(u)\mathrm{Fr}_n(\pa(u))$.
    Plugging the above into \eqref{eq:Jordan_2_LT_LB_decompo} and using Lemma \ref{lem:J2_LT_LB_main_term} at $(2C^\star S)^2$,
    \begin{equation*}
    \begin{aligned}
        \p\left(\hs_2(v)\le \frac{1}{S}v\right)
        &\geq
        \frac{1}{128\cdot (2C^\star S)^2}  - \frac{C}{v} - \exp \left( - \left( \tfrac{v}{2C^\star S} \right)^{1/4} \right)\geq \frac{c'}{S^2}
    \end{aligned}
    \end{equation*}
    for some constants $C>0$ and $c^\prime >0$ and all $S \in \left[1, \delta \sqrt{v} \right]$, for a sufficiently small constant $\delta>0$.
\end{proof}

\subsection{Proof of Theorem \ref{thm:main}}

\begin{proof}
    The proof of the theorem follows immediately from combining Lemmas \ref{lem:Jordan2_uppertail_UB}, 
\ref{lem:Jordan2_upper_tail_lower_bound},
 \ref{lem:Jordan2_lower_tail}, and \ref{lem:Jordan2_lower_tail_lower_bound}.
\end{proof}

\section{Performance of the Jordan ordering}\label{sec:Jordan}
The proof of Theorem~\ref{thm:jordan} follows the same philosophy as for Jordan-2, outlined at the beginning of Section \ref{sec:NBjordan}. The auxiliary lemmas developed in Section \ref{sec:NBjordan} allow us to carry out the analysis with substantially less work.  In Section \ref{sec:jordanuppertails}, we analyse the upper tail, i.e., the probability of the event $\{\hs_J(v) \geq Sv \}$, proving~\eqref{jordan:overest_combined}. In Section \ref{sec:jordanlowertails}, we analyse the lower tail, i.e., the probability of the event $\{\hs_J(v) \leq v/S \}$, proving  \eqref{jordan:underest_combined}.

\subsection{Upper tail}\label{sec:jordanuppertails}

In this section, we study the probability of the event $\left\{ \hs_J(v) \geq S v \right\}$ for large values of $S$. Lemma~\ref{lem:JordanLowerBound} gives a lower bound on this probability, whereas Lemma \ref{lem:JordanUpperBound} gives an upper bound. Together, these two lemmas directly imply the  two inequalities in \eqref{jordan:overest_combined} in Theorem \ref{thm:jordan}.

\begin{lemma}[Jordan ordering: lower bound for the upper tail]\label{lem:JordanLowerBound}
    Consider the Jordan ordering $\hs_J$. 
    There exists positive constants $c$ and $\delta$ such that for all $n\ge1$, any  vertex $v\in [n]$, and any $S\in[1, \delta n/v]$, 
    \begin{equation}\label{eq:jordan lower bound}
        \p \left( \hs_{J}(v) \geq S v \right) \geq \frac{c}{S}.
    \end{equation}
\end{lemma}

\begin{proof}
	We start by noting that it suffices to prove the result for $n \geq n_0$ and $S\geq S_0$ large enough. Let $v\in\{2,\ldots,n\}$ and $S \in[1, \frac{n}{64v}]$. Define the two events 
	\begin{align*}
		\cA & = \left\{ \mathrm{Fr}_{n}(v) \leq \frac{n}{64 Sv} \right\}, \\
		\cB & =  
		\left\{ \left| \left\{ u \in [n] : \phi_n(u) > \frac{n}{64Sv} \right\} \right| > Sv \right\} .
	\end{align*}
	On the event $\cA$, the fringe-tree size of $v$ is smaller than $n/2$, so $\phi_n(v) = \mathrm{Fr}_n(v) \leq \frac{n}{64Sv}$. Moreover, on the event $\cB$, there are at least $Sv$ vertices $u$ with $\phi_n(u)> \frac{n}{64Sv}$. Thus, on the event $\cA \cap \cB$ one has that
	\begin{equation*}
		\left| \left\{ u\in [n] : \phi_{n}(u) > \phi_{n}(v) \right\} \right|\, > \,  Sv,
	\end{equation*}
	which implies that
	\begin{equation*}
		\p \left( \hs_J(v)\ge Sv \right) \ \geq\ \p \left( \cA \cap \cB \right) \ \ge\ \p \left( \cA \right) - \p \left( \cB^c \right).
	\end{equation*}
    To bound $\p(\cA)$, we apply the lower bound on the lower tail of the fringe-tree size in \eqref{eq:fringe-lb} in Lemma \ref{lem:beta-bin-fringe}(3), which requires $S \le n/(64v)$, which holds by our assumptions on $S$. We obtain
	
	\begin{equation*}
		\p \left( \cA \right) = \p\left(\mathrm{Fr}_n(v)\le \frac{n}{64 Sv}\right) \geq \frac{1}{8\cdot 64S}=\frac{1}{512 S}.
	\end{equation*}
	In order to bound $\p(\cB^c)$ from above, we use Lemma \ref{claim:S over 64v} and get that for $S\leq \frac{n}{16 v}$,
	\begin{align*}
		\p\left( \cB^c \right) = \p \left(\left|\left\{ u \in [n] : \phi_n(u) > \frac{n}{64Sv} \right\}\right| \leq Sv\right) \leq C \exp(-Sv/8) \leq C \exp(-S/8) ,
	\end{align*}
	where $C$ is the constant from inequality \eqref{eq:S over 64v}.
	Thus, we see that
	\begin{align*} 
		\p \left( \hs_{J}(v) \geq S v \right) \geq \p \left( \cA \right) - \p \left( \cB^c \right)
		\geq
		\frac{1}{512S} - C e^{-S/8} \geq \frac{c_\star}{S} ,
	\end{align*}
	for some constant $c_\star>0$, and all $S$ large enough, say $S\geq S_0 > 1$. It remains to prove the result for $1\leq S\leq S_0$. Let $\delta = \tfrac{1}{(S_0 +1) 64}$. For $v > \delta n$, there is nothing to show since $\delta \tfrac{n}{v} < 1$ and we only consider $1\leq S\leq \delta \tfrac{n}{v}$. If $\delta \tfrac{n}{v} \geq 1$ (and thus $S_0 \leq \tfrac{n}{64v})$, then for $v\in \{2,\ldots, \lfloor \delta n \rfloor \}$, and $1 \leq S \leq S_0$, we have
    \begin{equation*}
        \p\left( \hs_J(v) \geq Sv \right) \geq \p\left( \hs_J(v) \geq S_0 v \right) \geq \frac{c_\star}{S_0} \geq \frac{c_\star / S_0}{S},
    \end{equation*}
    showing inequality \eqref{eq:jordan lower bound} with $c=c_\star / S_0$ and $\delta= \tfrac{1}{(S_0+1) 64}$, for all $v \in \{2,\ldots,\lfloor \delta n \rfloor \}$. For $v=1$, observe that the vertices $1$ and $2$ are indistinguishable for any label-invariant estimator, so that for $1\leq S \leq \delta \frac{n}{v}$ we get that
    \begin{align*}
        \p\left( \hs_J(1) \geq S \cdot 1 \right) = \p\left( \hs_J(2) \geq S/2 \cdot 2 \right) \geq \frac{c}{S/2} \geq \frac{c}{S},
    \end{align*}
    where we applied inequality \eqref{eq:jordan lower bound}, which was previously established for $v=2$.
\end{proof}

We proceed to an upper bound on the upper tail for the Jordan ordering, for which we employ a similar strategy as in the proof of Theorem \ref{thm:main}, which gave an upper bound on the upper tail for the Jordan-2 ordering.

\begin{lemma}[Jordan ordering: upper bound for the upper tail]\label{lem:JordanUpperBound}
    For the Jordan ordering $\hs_J$, there exists a positive constant $C$ such that for any $n\ge 1$, any vertex $v\in[n]$, and any $S\ge 1$,
    \begin{equation}\label{eq:devi}
        \p \left( \hs_{J}(v)\geq S v \right) \leq \frac{C}{S}.
    \end{equation}
\end{lemma}
\begin{proof}
     Fix some $n\in\N$, $v\in[n]$ and $S\ge 1$. Since $\hs_J$ is a permutation of $[n]$, if $S>n/v$ then the probability in \eqref{eq:devi} equals 0. Further, it suffices to consider $S$ large enough and $v \leq \frac{n}{2}$, by taking the constant $C$ in \eqref{eq:devi} large enough.
     
    We first assume $v\ge 2$. Note that if $\hs_J(v)\ge Sv$ then at least $Sv$ vertices have Jordan centrality at least $\phi_n(v)$. Thus, for some $M=M(S,v)>0$, we decompose
    $$
    \begin{aligned}
    \{\hs_J(v)\ge Sv\} &\subseteq \Big\{ \big|\big\{u \in [n] : \phi_{n}(u) \geq \phi_{n}(v)\big\} \big| \geq Sv \Big\} \\
    &\subseteq \left\{\phi_n(v) \le \frac{n}{Mv}\right\}\cup \left\{ \left|\left\{u \in [n] : \phi_{n}(u) \geq \frac{n}{Mv}\right\} \right| \geq Sv \right\}.
    \end{aligned}
    $$
    By Lemma \ref{lem:minima}, $\phi_n(v)\ge \min(\mathrm{Fr}_n^{>v}(v), \mathrm{Fr}_n^{>v}(1))$, which we use for the first event in the union. Further, by the exchangeability in Lemma \ref{lem:beta-bin-fringe}(1), $\mathrm{Fr}_n^{>v}(v)$ and $\mathrm{Fr}_n^{>v}(1)$ have the same distribution as $\mathrm{Fr}_n(v)$.
    For the second event we use the upper bound $\phi_n(u)\le \mathrm{Fr}_n(u)$, see \eqref{eq:jordan-upper}. 
    Therefore, 
    \[
    \begin{aligned}
    \p\left(\hs_J(v)\ge Sv\right)\ & 
    \leq
    \p \Big(\phi_n(v) \le \frac{n}{Mv}\Big) + \p\left(  \left|\left\{u \in [n] : \phi_{n}(u) \geq \frac{n}{Mv}\right\} \right| \geq Sv \right)\\
    &
    \le\  2\p\left(\mathrm{Fr}_n(v) \le \frac{n}{Mv}\right)+ \p\left( \left|\left\{u \in [n] : \mathrm{Fr}_n(u) \geq \frac{n}{Mv}\right\} \right| \geq Sv \right).
    \end{aligned}
    \]
    Inequality \eqref{eq:fringe-ub} from Lemma \ref{lem:beta-bin-fringe}(3) (which actually holds for all $M>0$) provides an upper bound for the first term, implying that 
    \[
    \p\left(\hs_J(v)\ge Sv\right) \le  \frac{4}{M}+ \p\left( \left|\left\{u \in [n] : \mathrm{Fr}_n(u) \geq \frac{n}{Mv}\right\} \right| \geq Sv \right).
    \]
    We bound the remaining term using Corollary \ref{coro:fringe tree counts large}. Let $C$ be the constant from Corollary~\ref{coro:fringe tree counts large}, let $\widetilde{C}>C$, and let $M\coloneqq \frac{S}{C}$. Then Corollary \ref{coro:fringe tree counts large} implies that 
    \begin{multline*}
    \p\left(\hs_J(v)\ge Sv\right) \le  \frac{4}{S/\widetilde{C}} + \p\bigg( \Big|\Big\{u \in [n] : \mathrm{Fr}_{n}(u) \geq \frac{n}{Sv/\widetilde{C}}\Big\} \Big| \geq \widetilde{C} \frac{Sv}{\widetilde{C}} \bigg)
    \\
    \le \frac{4 \widetilde{C}}{S}  + \exp\left(-\widetilde{C}^{-1/4}\left(Sv\right)^{1/4}\right) \leq \frac{8 \widetilde{C}}{S},
    \end{multline*}
    where the last inequality holds for $S$ large enough. To prove \eqref{eq:devi} for $v=1$, note that vertices $1$ and $2$ are indistinguishable for any label-invariant estimator. Thus for $S$ large enough, 
    \[
    \p\left(\hs_J(1)\ge S\cdot 1\right) = \p\left(\hs_J(2)\ge (S/2)\cdot 2\right)\le \frac{8\widetilde{C}}{S/2}=\frac{16\widetilde{C}}{S},
    \]
    so that the result \eqref{eq:devi} follows by taking the constant $C$ in \eqref{eq:devi} large enough.
\end{proof}

\subsection{Lower tail}\label{sec:jordanlowertails}

In this section, we provide upper and lower bounds on  $\p\big( \hs_J(v)\leq v/S \big)$.

\begin{lemma}\label{lem:JordanLowerBoundUpper}
    For the Jordan ordering $\hs_J$, there exist positive constants $C$ and $c$ such that for any $n\geq 1$, any vertex $v\in [n]$ and any $S\geq 1$, 

    \begin{equation}
        \mathbb{P}\left(\hs_J(v)\leq \frac{v}{S}\right) \leq C \left( e^{- c v/S} +  e^{-S/64} \right).
    \end{equation}
\end{lemma}

\begin{proof}
The statement is clear for $4v \leq S$, since the relevant probability equals $0$. Further, it suffices to consider the case when $v/S \in \N$, which we will assume for the rest of the proof. Lastly, the statement is clear for $S\in[1,32]$ by taking $C$ large enough.
    So let $n\in \N$, $v\in \{1,\ldots,n\}$ and $4v >S\ge 32$. First, note that if $\hs_J(v)\le v/S$ then no more than $v/S$ vertices have Jordan centrality exceeding $\phi_n(v)$. Hence

    \begin{align*}
        \p\left(\hs_J(v)\le \frac{1}{S}v\right) 
        &\leq \p\left(\big|\big\{u\in [n]: \phi_n(u) >\phi_n(v)\big\}\big|\le \frac{v}{S}\right) \\
        &\le \p\left(\left|\left\{u : \phi_n(u) > S \frac{n}{64v} \right\}\right|\le \frac{v}{S}\right) + \p \left( \phi_n(v) \geq S \frac{n}{64 v}  \right) .
    \end{align*}
    The first summand can be bounded using Lemma \ref{claim:S over 64v} for $t=v/S$, which applies as $v/S\le n/16$ because $S\ge 32$. For the second summand we use that $\phi_n(v)\le \mathrm{Fr}_n(v)$, so that
    \begin{equation}\nonumber
        \p\left(\phi_n(v) \ge S \frac{n}{64v}\right) \leq \p\left(\mathrm{Fr}_n(v)\ge \frac{S}{64} \frac{n}{v}\right) \le  C^\prime \exp\left(-\frac{S}{64} \right)
    \end{equation}
    with $C^\prime = 2 e^2$, by inequality \eqref{eq:proba_fringe-large}. Thus, we get that
    \begin{equation*}
        \p\Big(\hs_J(v)\le \frac{1}{S}v\Big) 
        \le C \exp \Big( - \frac{v}{8S} \Big)+ C^\prime \exp\Big(-\frac{S}{64} \Big). \qedhere
    \end{equation*}
\end{proof}

\begin{lemma}\label{lem:JordanLowerBoundLower}
    For the Jordan ordering $\hs_J$, there exist positive constants $C, c^\prime$ and $c$ such that for any $n\geq 1$, any vertex $v\in[n/2]$ and any $S\in[1,c^\prime v]$

    \begin{equation}\label{eq:Jordan lower bound lemm 3.4}
        \mathbb{P}\left(\hs_J(v)\leq \frac{v}{S}\right) \geq c e^{-CS}.
    \end{equation}
\end{lemma}

\begin{proof}
    It suffices to consider the case where $v/S \in \N$, which we will assume for the rest of the proof. Moreover, we assume that $v\ge3$, as the range for $S$ can be made empty for $v\in\{1,2\}$ by choosing $c'$ small.
    We start with the case where $S \leq \sqrt[10]{v}$. 
    Let $M\ge 1$ be a constant such that for all $t>0$,
    \begin{equation}\label{corocoro}
        \p \left( \big| \big\{ v\in [n]  : \mathrm{Fr}_{n}(v)  > n/t \big\} \big| >  Mt \right) \leq \exp \left( 1-t^{1/4} \right).
    \end{equation}
    Such a constant exists by Corollary \ref{coro:fringe tree counts large}.
    The most likely way to underestimate the arrival time of $v$ is when its fringe tree is large. Similar to the proof of Lemma \ref{lem:Jordan2_lower_tail_lower_bound}, we bound
  
    \begin{align}
       \notag \mathbb{P}\left(\hs_J(v)\leq \frac{v}{S}\right) &\geq 
\p\left(\big|\big\{u\in [n]: \phi_n(u) \ge\phi_n(v)\big\}\big|\le \frac{v}{S} \right) \\
         \notag&\ge 
        \p\left( \Big\{\mathrm{Fr}_n(v)\in \left[MS\frac{n}v, \frac{n}{2}\right]\Big\} \cap \Big\{\big|\big\{u\in [n]: \phi_n(u) \ge MS\frac{n}{v}\big\}\big|\le v/S\Big\}  \right) \\
        &\ge \p\left(\mathrm{Fr}_n(v)\ge MS\frac{n}{v}\right) - \p\left(\mathrm{Fr}_n(v)\ge n/2\right) \label{eq:jordan-lower-lower1}\\&\hspace{3cm}- \p\left(\big|\big\{u\in [n]: \phi_n(u) \ge MS\frac{n}{v}\big\}\big|> v/S\right).\label{eq:jordan-lower-lower2}
        \end{align}
    To bound the probabilities in \eqref{eq:jordan-lower-lower1}, we employ \eqref{eq:proba-parent-large-lower} and \eqref{eq:proba_fringe-large} from Lemma \ref{lem:proba_large_fringe}, which imply that
    \begin{align*}
        \p\left(\mathrm{Fr}_n(v)\ge MS\frac{n}{v}\right) \geq 
        \frac{1}{2} \exp \left(-8MS \right), \qquad \text{ and } \qquad
        \p\left(\mathrm{Fr}_n(v)\ge n/2\right) \leq 2 e^2 \exp \left( - \frac{v}{2} \right) .
    \end{align*}
    Further, the remaining probability on \eqref{eq:jordan-lower-lower2} is bounded from above by \eqref{corocoro} (see also Corollary~\ref{coro:fringe tree counts large}), using additionally that $\phi_n(v)\le \mathrm{Fr}_n(v)$:
    \begin{multline*}
        \p\left(\left|\left\{u\in [n]: \phi_n(u) \ge MS\frac{n}{v}\right\}\right|> \frac{v}{S}\right)
        \\
        \leq \p\left(\left|\left\{u\in [n]: \mathrm{Fr}_n(u) \ge \frac{n}{v/(MS)}\right\}\right|> M \frac{v}{MS}\right) 
        \leq \exp\left(1 - \left( \frac{v}{MS} \right)^{1/4} \right) .
    \end{multline*}
    Combining the above inequalities, we get that 
    \begin{equation*}
    \begin{aligned}
        \mathbb{P}\left(\hs_J(v)\leq \frac{v}{S}\right) 
        &\ge\frac{1}{2} \exp \left( -8MS \right)  - 2 e^2 \exp\left( - \frac{v}{2} \right)- \exp\left(1 - \left( \frac{v}{MS} \right)^{1/4} \right).
        \end{aligned}
    \end{equation*}
    Since we assumed that $S\leq \sqrt[10]{v}$, we see that there exist positive constants $c^\prime, c, C$ so that
    \begin{equation*}
        \mathbb{P}\left(\hs_J(v)\leq \frac{v}{S}\right) \geq c e^{-CS}.
    \end{equation*}
    for any $1 \leq S \leq c^\prime v \wedge \sqrt[10]{v}$.

\smallskip 
    Next, we consider the case where $\tfrac{v}{2} \geq S> \sqrt[10]{v}$. We intersect with the event that $\pa(v)=1$ and also condition on the size of the fringe tree of $v$. We first prove for each $k\in [2 RSn/v, n/2]$, where $R$ is a large constant, the inclusion 
    \begin{equation}\label{eq:inclusion}
        \Big\{\hs_J(v)\le \frac{v}S\Big\}\, \supseteq\,\left\{ \mathrm{pa}(v) = 1, \mathrm{Fr}_n(v) = k \right\} \cap \left\{ \left| \left\{u \notin (T_n,1)_{v \downarrow}: \mathrm{Fr}_n(u) \geq k \right\} \right| \leq  \frac{Rn}{k} \right\}.
    \end{equation} We start from the right-hand side.
    If the fringe tree of $v$ has size equal to $k$, then all its children have fringe tree size strictly smaller than $k$, allowing to change the counted set in the last event to include all $u\in[n]$ without changing the event. Since for $k\le n/2$ we have $k=\mathrm{Fr}_n(v)=\phi_n(v)$,  we have 
    \begin{align*}
    &\left\{ \mathrm{Fr}_n(v)=k, \left| \left\{u \notin (T_n,1)_{v \downarrow}: \mathrm{Fr}_n(u) \geq k \right\} \right| \leq  \frac{Rn}{k} \right\}
    \\
    &
    \hspace{4cm}
    \subseteq
    \left\{  \phi_n(v)=k, \left| \left\{u \in[n], \mathrm{Fr}_n(u) \geq \phi_n(v) \right\} \right| \leq  \frac{Rn}{k} +1 \right\}\\
    &
    \hspace{4cm}
    \subseteq 
    \left\{  \phi_n(v)=k, \left| \left\{u \in[n]: \phi_n(u) \geq \phi_n(v) \right\} \right| \leq  \frac{Rn}{k} +1 \right\},
    \end{align*}
    where we used that $\phi_n(u)\le \mathrm{Fr}_n(u)$ for the last inclusion. The inclusion \eqref{eq:inclusion} follows since $Rn/k +1 \le v/S$ for $k\in [2RSn/v, n/2]$ and $S \leq \tfrac{v}{2}$. Since the events on the right-hand side in \eqref{eq:inclusion} are disjoint in $k$, we obtain that
    \begin{align}
    \notag
    \p\big(\hs_J(v)\le v/S\big) \ge \sum_{k= \lceil 2RSn/v\rceil}^{n/2}\bigg( &\p \Big( \left| \left\{u \notin (T_n,1)_{v \downarrow}: \mathrm{Fr}_n(u) \geq k \right\} \right| \leq \frac{Rn}{k}\, \Big|\,  \mathrm{pa}(v) = 1, \mathrm{Fr}_n(v) = k \Big) \\ 
    & \label{align:three terms}
    \hspace{15pt}\cdot \p\big( \mathrm{pa}(v) = 1\big)\cdot\p\big( \mathrm{Fr}_n(v) = k \big)\bigg),
    \end{align}
    using that the size of the fringe tree of $v$ and the parent of $v$ are independent. Conditioned on the event $\left\{ \mathrm{pa}(v) = 1, \mathrm{Fr}_n(v) = k \right\}$, the tree $T_n \setminus (T_n,1)_{v \downarrow}$ is distributed like a uniform attachment tree of size $n-k$. By Markov's inequality, 
    \begin{align*}
       \p \Big( \Big| \Big\{u \notin (T_n,1)_{v \downarrow}&: \mathrm{Fr}_n(u) \geq k \Big\} \Big| \leq \frac{Rn}{k}\, \Big|\,  \mathrm{pa}(v) = 1, \mathrm{Fr}_n(v) = k \Big) \\
        &=
        1-\p \Big( \left| \left\{u\in[n-k]: \mathrm{Fr}_{n-k}(u) \ge k \right\} \right| > \frac{Rn}{k}\Big) \\
        &
        \geq 1 - \frac{k}{Rn} \E \left[ \left| \left\{u\in[n-k]: \mathrm{Fr}_{n-k}(u) \ge k \right\} \right| \right]
        \\
        &= 1- \frac{k}{Rn}\sum_{u\in[n-k]}\p\Big(\mathrm{Fr}_{n-k}(u)\ge k\Big) \\
        &\ge 1- \frac{k}{Rn}\sum_{u\in[n]}\p\Big(\mathrm{Fr}_{n}(u)\ge k\Big).
    \end{align*}
    By \eqref{eq:proba_fringe-large} in Lemma \ref{lem:proba_large_fringe} there exists an absolute constant $C>0$ such that the last sum in the above display is at most $Cn/k$. Fixing $R=2C$ implies that the probability is at least $1/2$. Using that $\p(\pa(v)=1)=1/(v-1)$, and inserting the previous inequalities into \eqref{align:three terms}, we see that
    \[
    \p\big(\hs_J(v)\le v/S\big) \ge \sum_{k= \lceil 2RSn/v\rceil}^{n/2} \frac{1}{2} \frac{1}{v-1} \p \left( \mathrm{Fr}_n(v)=k \right) = \frac{1}{2(v-1)}\p\Big(\mathrm{Fr}_n(v)\in \big[2RSn/v, n/2\big] \Big).
    \]
    Using \eqref{eq:proba-parent-large-lower} and \eqref{eq:proba_fringe-large} in Lemma \ref{lem:proba_large_fringe}, it follows that there exist constants $c, c'>0$ such that for all $S\in[\sqrt[10]{v}, c'v]$,
    \[
    \p\big(\hs_J(v)\le v/S\big) \ge \frac{1}{2(v-1)}\left( \frac{1}{2} \exp\left( -16 RS \right) -2 e^2 \exp \left( - \frac{v}{2} \right) \right) \geq c e^{-CS}. \qedhere
    \]
\end{proof}

\subsection{Proof of Theorem \ref{thm:jordan}}
\begin{proof}
    Lemmas \ref{lem:JordanLowerBound} and \ref{lem:JordanUpperBound} establish the upper tail \eqref{jordan:overest_combined}. Lemma \ref{lem:JordanLowerBoundLower} shows the first inequality in \eqref{jordan:underest_combined}. Lemma \ref{lem:JordanLowerBoundUpper} shows the second inequality in \eqref{jordan:underest_combined}, since $v/S\ge S$ for $S \leq \sqrt{v}$. Thus,
    \begin{equation*}
        \p\Big(\hs_J(v)\le \frac{1}{S}v\Big)\leq C \left( e^{- c v/S} +  e^{-S/64} \right) \leq C \left( e^{- c S} +  e^{-S/64} \right) .\qedhere
    \end{equation*}
\end{proof}

\section{Lower bounds on the performance of  arbitrary estimators}\label{sec:lower bounds}

This section proves Proposition \ref{prop:lower_bound}, providing lower bounds on 
\[
\p\left(|\hs(v, T_n)-v|\ge Sv\right)
\]
for arbitrary label-invariant arrival-time estimators $\hs$. We include $T_n$ as an argument of $\hs$ to stress that $\hs(\,\cdot\, , T_n)$ is a random permutation whose distribution depends on the realization of the random tree $T_n$.

The proof rests on isolating, for each vertex $v$, a subset $I_v(T_n)$ of vertices that are indistinguishable from $v$ given the information contained in the tree. Every label-invariant estimator has to assign to half of these vertices an estimated arrival time that is at a distance of at least $|I_v(T_n)|/4$ from $v$, and with probability $1/4$, $v$ is among those vertices. Our lower bound will follow by bounding the size of the set $I_v(T_n)$ from below. Next, we formalize this reasoning, starting with a definition of $I_v$.



Let $(\rt_n)_{n\ge1}$ be a deterministic sequence of recursive trees (i.e., with labels increasing along every non-backtracking path started at the root) such that for all $n\ge 1$, $\rt_n$ is the subtree of $\rt_{n+1}$ induced on the vertices with labels in $[n]$.
Recall that $\pa(i)$ is the parent of vertex $i$, fix some $v \in \{2,\ldots,n\}$ and define  the set
\begin{equation}
I_v(\rt_n):=\big\{ i\in [v/2,v]\ : \ i\text{ is a leaf in }\rt_v, \ \mathrm{pa}(i) < v/2\big\}.\label{eq:Iv}
\end{equation}

\noindent In words, $I_v(\rt_n)$ is the set of all vertices in $\rt_n$ which arrived between time $v/2$ and $v$, that are still leaves at time $v$, and which are children of vertices that arrived before time $v/2$. The set $I_v(\rt_n)$ satisfies $I_v(\rt_n) = I_v(\rt_v)$ for all $n\geq v$, and  contains $v$ if and only if $\pa(v) < v/2$. The key observation is that if $I_v(\rt_n)=\mathcal{I}$, then any permutation of labels in $\mathcal{I}$ yields a relabeled tree $\rt_n^\tau$ that is also recursive (here we slightly abuse notation and extend $\tau$ to act as the identity on $\{1,\ldots,n\}\setminus\mathcal{I}$).
An illustration is given in Figure \ref{fig:permutation_tree} and the proof of this fact is given in the following lemma.

\begin{lemma}\label{lem:permutation_still_recursive}
    For any recursive tree $\rt$ of size at least $v$ and any permutation $\tau$ of $I_v(\rt)$, the tree $\rt^{\tau}$ is recursive. Moreover, $I_v(\rt)=I_v(\rt^{\tau})$.
\end{lemma}

\begin{proof}
    Let $\rt$ be a recursive tree of size $n\ge v$. The proof of the lemma is trivial if $I_v(\rt)$ has only one element. Suppose $|I_v(\rt)|\ge2$ and let $i,j$ be two distinct vertices in $I_v(\rt)$. Since every permutation on $I_v(\rt)$ can be written as a product of transpositions on $I_v(\rt)$, it suffices to prove the claim for transpositions. Let $\tau=(i,j)$. Every simple path from the root to a leaf that does not pass through $i$ or $j$ remains increasing in $\rt^\tau$. Consider the two simple paths $\pi^\ssup{i}$ and $\pi^\ssup{j}$ from vertex $1$ passing through $i$ and $j$, respectively, and ending at a leaf of $\rt$. We have to show that the labels along these paths remain increasing if we swap $i$ and $j$ and that $i$ and $j$ belong to $I_v(\rt^{\tau})$. 

    Because $\rt$ is recursive, these paths are of the form $\pi^\ssup{i}=(1,i_1,\dots ,i_k,i,i_{k+1},\dots,i_K)$ and $\pi^\ssup{j}=(1,j_1,\dots ,j_\ell,j,j_{\ell+1},\dots,j_L)$ with increasing labels along both paths.
    By definition of $I_v(\rt)$, we have that $i_k < v/2 \leq i\leq v < i_{k+1}$ and $j_\ell< v/2 \leq j \leq v < j_{\ell+1}$.
    As a consequence, the two simple paths in $\rt^{\tau}$ starting at $1$, going through $i$ and $j$, respectively, and finishing at the same leaves become $(1,i_1,\dots ,i_k,j,i_{k+1},\dots,i_K)$ and $(1,j_1,\dots ,j_\ell,i,j_{\ell+1},\dots,j_L)$. Hence these paths remain increasing and both $i$ and $j$ belong to $I_v(\rt^{\tau})$. Thus $\rt^{\tau}$ is recursive and $I_v(\rt)=I_v(\rt^{\tau})$. 
    \qedhere

\end{proof}

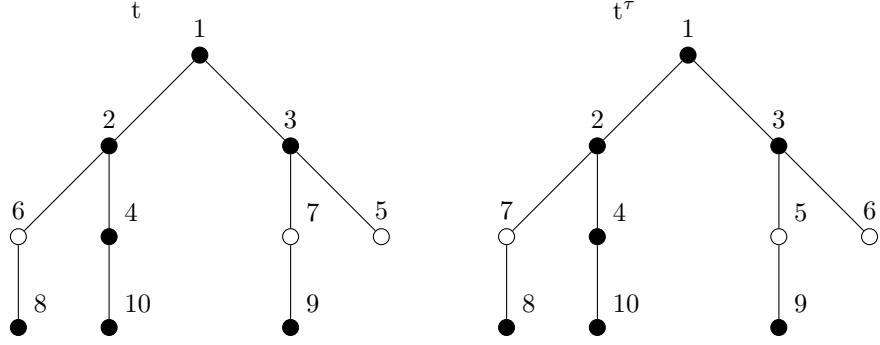
\begin{figure}
		
		\centering
	
	\begin{tikzpicture}[scale=1.2
		]
		\vertex[draw=none] (T) at (-0.7,0.5) {$\rt$};
		
		\vertex[fill=black, label=above:{$1$}] (A) at (0,0) {};
		\vertex[fill=black, label=above:{$2$}] (B) at (-1,-1) {};
		\vertex[fill=black, label=above:{$3$}] (C) at (1,-1) {};
		\vertex[fill=white, label=above:{$6$}] (D) at (-2,-2) {};
		\vertex[fill=black, label=above right:{$4$}] (E) at (-1,-2) {};
		\vertex[fill=white, label=above right:{$7$}] (F) at (1,-2) {};
		\vertex[fill=black, label=above right:{$8$}] (G) at (-2,-3) {};
		\vertex[fill=black, label=above right:{$10$}] (H) at (-1,-3) {};
		\vertex[fill=black, label=above right:{$9$}] (J) at (1,-3) {};
        \vertex[fill=white, label=above:{$5$}] (K) at (2,-2) {};
		
		\path (A) edge (B);
		\path (A) edge (C);
		\path (D) edge (B);
		\path (E) edge (B);
		\path (F) edge (C);
		\path (G) edge (D);
		\path (H) edge (E);
		\path (J) edge (F);
        \path (C) edge (K);

		\vertex[draw=none] (T2) at (5.38+-0.7,0.5) {$\rt^\tau$};
		
		\vertex[fill=black, label=above:{$1$}] (A2) at (5.38+0,0) {};
		\vertex[fill=black, label=above:{$2$}] (B2) at (5.38+-1,-1) {};
		\vertex[fill=black, label=above:{$3$}] (C2) at (5.38+1,-1) {};
		\vertex[fill=white, label=above:{$7$}] (D2) at (5.38+-2,-2) {};
		\vertex[fill=black, label=above right:{$4$}] (E2) at (5.38+-1,-2) {};
		\vertex[fill=white, label=above right:{$5$}] (F2) at (5.38+1,-2) {};
		\vertex[fill=black, label=above right:{$8$}] (G2) at (5.38+-2,-3) {};
		\vertex[fill=black, label=above right:{$10$}] (H2) at (5.38+-1,-3) {};
		\vertex[fill=black, label=above right:{$9$}] (J2) at (5.38+1,-3) {};
        \vertex[fill=white, label=above:{$6$}] (K2) at (5.38+2,-2) {};
		
		\path (A2) edge (B2);
		\path (A2) edge (C2);
		\path (D2) edge (B2);
		\path (E2) edge (B2);
		\path (F2) edge (C2);
		\path (G2) edge (D2);
		\path (H2) edge (E2);
		\path (J2) edge (F2);
        \path (C2) edge (K2);
		
	\end{tikzpicture}
	
	\caption{The indistinguishable vertices $I_7(\rt)$ in a recursive tree. For a recursive tree $\rt$ of size $10$, we display in white $I_7(\rt)=\{5,6,7\}$ as well as $\rt^{\tau}$ for $\tau$ a permutation of $I_7(\rt)$. 
    }
    \label{fig:permutation_tree}
\end{figure}

Lemma~\ref{lem:permutation_still_recursive} shows that relabeling the vertices in $I_v(\rt_n)$ preserves recursivity of the tree. 
The next lemma formalizes the key consequence for estimators: within each set $I_v(\rt_n)$, the estimated labels by $\hs$ have the same distribution. Later on, we apply the following lemma for $I=I_v$ when $v$ is sufficiently large, and consider $I=\{v-1,v\}$ when $v$ is of constant order.  

\begin{lemma}\label{lem:IorderedRandom}
Fix some $v \geq 1$ and let $I$ be a function acting in the space of  recursive trees and returning a subset of vertices of $[v]$, such that for any recursive tree $\rt$ of size $n\geq v$,

    \begin{itemize}
        \item for any permutation $\tau$ of $I(\rt)$, the tree $\rt^{\tau}$ is recursive and $I(\rt^{\tau})=I(\rt)$;
        \item for $\rt_v$, the subtree of $\rt$ spanned by vertices in $[v]$, $I(\rt_v)=I(\rt)$. 
    \end{itemize}
    Then, for any label-invariant estimator $\hs$ and any set $\mathcal{I}$ in the image of $I$ with $i,v \in \mathcal{I}$,

    $$
     \hs\left(v,T_n\right)\mid I(T_n)=\mathcal{I} \quad 
     \overset{\mathcal{L}}{=} \quad 
     \hs\left(i,T_n\right)\mid  I(T_n)=\mathcal{I}.
     $$

\end{lemma}
\begin{proof}
    Let $\hs$ be a label-invariant estimator. From \cite[Theorem 4]{CrXu21} we know that in the RRT model, for every recursive tree $\rt$ on $n$ vertices and every permutation $\gamma$ such that $\rt^{\gamma}$ is also recursive, one has $\p(T_n=t)=\p(T_n=t^{\gamma})$.
    So, for a fixed recursive tree $\rt$ of size at least $v$ and a permutation $\tau$ preserving recursivity on this tree

    $$\p\left( T_n=\rt  \right)=\p\left( T_n=\rt^{\tau^{-1}}  \right)=\p\left( T_n^{\tau}=\rt  \right).$$
    Hence, for every fixed set $\mathcal{I}\subseteq[v]$ and any permutation $\tau$ of $\mathcal{I}$, 
    \begin{equation}\label{eq:tree_distrib_equality2}
    T_n\mid I(T_v)=\mathcal I \ \overset{\mathcal L}{=} \ T_n^\tau\mid I(T_v)=\mathcal I,
    \end{equation}
    where we condition on $I(T_v)$ on both sides because $I(T_v)=I(T_v^{\tau})$.
    For a fixed tree $\rt$ and a permutation $\tau$ of $\mathcal{I}$, because $\hs$ is label-invariant (see \eqref{eq:def_label-invariant}),
    $$\hs\left(i,\rt\right)\quad 
    \overset{\mathcal{L}}{=} \quad \hs\left(\tau(i),\rt^{\tau}\right),$$
    where the equality would be deterministic if the label-invariant estimator $\hs$ were deterministic given a tree. Fix some non-empty subset $\mathcal{I}\subseteq [v]$ consistent with $I(\rt)=\mathcal{I}$ for some tree $\rt$, and $\tau$ a permutation of $\mathcal{I}$. Plugging the random  tree $T_n$ into the above equality, we get for $n\ge v$,

    $$\hs\left(i,T_n\right)\mid I(T_v)=\mathcal{I} \quad 
    \overset{\mathcal{L}}{=} \quad \hs\left(\tau(i),T^{\tau}_n\right)\mid I(T_v)
    =\mathcal{I}.$$
    Since we assumed that $v,i\in\mathcal{I}$, the transposition $(i,v)$ is a permutation of $\mathcal{I}$. Choosing $\tau=(i,v)$ yields for $n\ge v$, 

    $$\hs\left(i,T_n\right)\mid I(T_v)=\mathcal{I} \quad 
    \overset{\mathcal{L}}{=} \quad \hs\left(v,T_n^\tau\right)\mid I(T_v)=\mathcal{I}\quad 
    \overset{\mathcal{L}}{=} \quad \hs\left(v,T_n\right)\mid I(T_v)=\mathcal{I},$$
    where the second equality follows from \eqref{eq:tree_distrib_equality2}. For all $n\ge v$, we have $I(T_
    n)=I(T_v)$ which concludes the proof of the lemma. \qedhere
    

\end{proof}

Equipped with this technical lemma, we are now ready to prove Proposition \ref{prop:lower_bound}

\begin{proof}[Proof of Proposition \ref{prop:lower_bound}]
    Let $\hs$ be a label-invariant estimator of the arrival times in the random recursive tree $T_n$ of size $n\ge v$. Let $\delta = \tfrac{1}{56}$. We start by proving that for $v\in[3,n]$ and $v\geq 1/(2\delta) = 28$
    \[
    \p\left(|\hs(v, T_n)-v|\ge \delta v\right) \ge \frac{1}{112}-\frac{\delta}{4},
    \]
    which implies Proposition \eqref{prop:lower_bound}
for $v \geq 28.$    We condition on the set of indistinguishable vertices $I_v(T_n)$. That is,  
    \begin{equation}\label{eq:condition-1}
\p\left(|\hs(v, T_n)-v|\ge \delta v\right)
\ge\!\!
\sum_{\substack{\mathcal{I} \subseteq [n]:  \\ |\mathcal{I}|\ge 4\delta v ,\\ v \in \mathcal{I}}} \!\!
\p\left(|\hs(v,T_n)-v|\ge \frac{|I_v(T_n)|}{4}\ \Big|\ I_v(T_n) = \mathcal{I}\right) \p\big( I_v(T_n) = \mathcal{I}\big) .
    \end{equation}
    We first argue that the conditional probability is bounded from below by $1/4$. By the definition of $I_v$ in \eqref{eq:Iv} and Lemma \ref{lem:permutation_still_recursive}, $I_v$ satisfies the hypothesis of Lemma \ref{lem:IorderedRandom}. So the estimated arrival times $\hs(u, T_n)$ are identically distributed for $u\in I_v(T_n)$. Therefore, for any set $\mathcal{I}$ with $|\mathcal{I}|\geq 4 \delta v$,
\begin{align*}
    &\p\left(|\hs(v,T_n)-v|\ge \frac{|I_v(T_n)|}{4}\ \big|\ I_v(T_n) = \mathcal{I}\right)
    \\
    & \hspace{4cm} = \frac{1}{|\mathcal{I}|}\sum_{u\in \mathcal{I} } \p\left(|\hs(u,T_n)-v|\ge \frac{|I_v(T_n)|}{4} \ \big|\ I_v(T_n) = \mathcal{I}\right)
    \\
    & \hspace{4cm} = \frac{1}{|\mathcal{I}|}\E\Bigg[\sum_{u\in I_v(T_n)} \mathbbm{1}_{\left\{|\hs(u,T_n)-v|\ge \frac{|I_v(T_n)|}{4} \right\}} \ \Big|\ \ I_v(T_n) = \mathcal{I}  \Bigg] .
\end{align*}

Since $\hs(\, \cdot\, , T_n)$ is a bijection from $[n]$ to $[n]$, 
at least $|I_v(T_n)|/4$ of the elements in $I_v(T_n)$ must satisfy $|\hs(u,T_n)-v| \geq |I_v(T_n)|/4$ when $|I_v(T_n)| \geq 2$. Since $|I_v(T_n)| \geq 4\delta v \geq 2$, the sum within the conditional expectation is bounded from below by $|I_v(T_n)|/4 = |\mathcal{I}|/4$, and the entire expression is bounded from below by $1/4$. 
Substituting this bound into \eqref{eq:condition-1} we get that
    \begin{equation}\label{eq:proba_lower_bound_order}
        \p\left(|\hs(v,T_n)-v|\ge \delta v\right) \ \ge \!\!
        \sum_{\substack{\mathcal{I} \subset [n]:  \\ |\mathcal{I}|\ge 4\delta v ,\\ v \in \mathcal{I}}} \!\!
\tfrac{1}{4} \p\big( I_v(T_n) = \mathcal{I}\big)\
=\
\tfrac{1}{4}\p\big(|I_v(T_n)|\ge 4\delta v, v \in I_v(T_n)\big).
    \end{equation}




    
\noindent To conclude, we bound the probability on the right-hand side using a first-moment method. We define the set $I_v^\prime (T_{v-1})$ by

\begin{equation*}
I_v^\prime(T_{v-1}):=\big\{ i\in [v/2,v-1]\ : \ i\text{ is a leaf in }T_{v-1}, \ \mathrm{pa}(i) < v/2\big\}.
\end{equation*}

We compute the expected size of this set to be
\begin{align*}
    \E\big[ |I_v^\prime(T_{v-1})|   \big] &=   \sum_{k=\lceil \tfrac{v}{2}\rceil}^{v-1} \p\left( k\text{ is a leaf in }T_{v-1}, \ \mathrm{pa}(k)< v/2 \right)\\
    &=
    \sum_{k=\lceil \tfrac{v}{2}\rceil}^{v-1} \frac{\lceil \tfrac{v}{2}\rceil -1}{k-1}\prod_{j=k+1}^{v-1} \frac{j-2}{j-1}
    =
    \big( v-\lceil \tfrac{v}{2}\rceil \big) \frac{\lceil \tfrac{v}{2}\rceil-1}{v-2}
    \\
    &
    \geq
    \left( \frac{v}{2} - \frac{1}{2} \right) \frac{ \tfrac{v}{2}-1}{v-2} 
=
\frac{v}{4}- \frac{1}{4}
\geq 
\frac{v}{7}
    ,
\end{align*}
where the last inequality holds since $v \geq 3$.

\noindent By truncating the first moment, and using the deterministic bound $|I_v^\prime(T_{v-1})|\le v$, we see that
\begin{equation*}
\begin{aligned}
\frac{v}{7}\le \E\left[ |I_v^\prime(T_{v-1})| \right] &= \E\big[|I_v^\prime(T_{v-1})|\Ind{|I_v^\prime(T_{v-1})|< 4\delta v}\big]+\E\big[|I_v^\prime(T_{v-1})|\Ind{|I_v^\prime(T_{v-1})|\ge4\delta v}\big]\\
&\leq 4\delta v + v\cdot \p\left( |I_v^\prime(T_{v-1})|\geq 4\delta v \right).
\end{aligned}
\end{equation*} 
Rearranging gives that for $v\ge 3$,
$$\p\left( |I_v^\prime(T_{v-1})|\geq 4 \delta v \right)\, \geq  \,\frac{1}{7}-4\delta  .$$
The set $I_v^\prime (T_{v-1})$ and $\mathrm{pa}(v)$ are independent. Further, if $\mathrm{pa}(v) < v/2$, then $I_v(T_n) = I_v(T_v) = I_v^\prime(T_{v-1}) \cup \{v\}$, implying that
\begin{align*}
    &\p \left( |I_v(T_n)| \geq 4 \delta v, v \in I_v(T_n) \right) = \p \left( |I_v(T_n)| \geq 4 \delta v, \mathrm{pa}(v) < \frac{v}{2} \right)
    \\
    & \hspace{25mm}
    \geq
     \p \left( |I_v^\prime(T_{v-1})| \geq 4 \delta v, \mathrm{pa}(v) < \frac{v}{2} \right)
    \\
    & \hspace{25mm}
    =
    \p \left( |I_v^\prime(T_{v-1})| \geq 4 \delta v\right) \p \left( \mathrm{pa}(v) < \frac{v}{2} \right)
    \geq
    \left( \frac{1}{7} - 4\delta \right)  \frac{\tfrac{v}{2} -1}{v-1} \geq \frac{1}{28} - \delta .
\end{align*}
Plugging this result into \eqref{eq:proba_lower_bound_order}, yields
\begin{equation*}
\p\big(|\hs(v,T_n)-v|\ge \delta v\big)
\ge\frac{1}{112}-\frac{\delta}{4} .
\end{equation*}
Since $\delta =1/56$, this concludes the proof of the proposition for $v\geq 28$, as we required $v\ge 1/(2\delta)$ at the beginning of the proof.

\medskip

For any $v\in[3,28]$ and any recursive tree $\rt$ of size at least $v$, we define the set

$$\Tilde{I}_v(\rt_n):=\begin{cases}
    \{v-1,v\} & \text{ if }v-1 \text{ and }v \text{ are leaves in } \rt_v, \\
    \emptyset & \text{ otherwise.}
\end{cases}$$
It is direct to check that $\Tilde{I}_v$ satisfies the hypothesis of Lemma \ref{lem:IorderedRandom}. Hence 

$$\hs(v,T_n) \ \big| \  \Tilde{I}_v(T_n)\neq \emptyset \quad
\overset{\mathcal{L}}{=}
\quad \hs(v-1,T_n) \ \big| \  \Tilde{I}_v(T_n)\neq \emptyset,$$
and in particular $\p\big( \hs(v)=v \ \big| \ \Tilde{I}_v(T_n)\neq \emptyset\big)\leq 1/2$. So 

$$\p\big(|\hs(v,T_n)-v|\ge 1 \big)
\geq
\p\big(\Tilde{I}_v(T_n)\neq \emptyset\big)
\p\left(|\hs(v,T_n)-v|\ge 1 \ \big|\ \Tilde{I}_v(T_n)\neq \emptyset\right) \geq \frac{1}{2}\p\big(\Tilde{I}_v(T_n)\neq \emptyset\big).$$
Checking that $\p\big(\Tilde{I}_v(T_n)\neq \emptyset\big)=\tfrac{v-2}{v-1}\geq \tfrac{1}{3}$ and that $1\geq v/28$ yields

$$\p\big(|\hs(v,T_n)-v|\ge v/28 \big)
\geq \frac{1}{6},$$
concluding the proof of the proposition for $v\in [3,28]$.

\medskip

For $v\in\{1,2\}$, we use that vertices $1$ and $2$ are indistinguishable, hence $\hs(1,T_n)\overset{\mathcal{L}}{=}\hs(2,T_n)$ and so $\p\big(|\hs(v,T_n)-v|\ge 1\big)\geq 1/2$, proving the proposition for $v=1,2$.
\end{proof}

\appendix 

\section{Distribution of fringe-tree sizes}\label{sec:appendix}
In this appendix, we give the proofs of Lemmas \ref{lem:beta-bin-fringe}, \ref{lem:appendix_fringe_size}, \ref{lem:proba_large_fringe}, and \ref{lem:Fringe_fringe_parent_large}. (We give the proof of Lemma~\ref{lem:proba_large_fringe} before the proof of Lemma \ref{lem:Fringe_fringe_parent_large}, since the proof of Lemma \ref{lem:Fringe_fringe_parent_large} depends on Lemma \ref{lem:proba_large_fringe}).

\betabinfringe*

    \begin{proof}
    For the first item, we observe that the fringe trees witnessed after time $m$ are all disjoint.
        Their evolution after time $m$ is modeled by a P\'olya urn \cite{eggenberger1923statistik} with $m$ colors, starting at time $m$ with one ball of each color.
        At each time $n > m$, a ball is picked uniformly at random from the urn, and the picked ball is replaced by two balls of the same color as the picked ball.
        The number of balls of color $i\in[m]$ at time $n$ describes the witnessed fringe tree size of a vertex: indeed, each arriving vertex joins a fringe tree with probability proportional to the size of the tree. 

        The number of balls of each color at time $n\ge m$ is an exchangeable family of random variables, yielding point 1.
        The family follows a Dirichlet-multinomial distribution, implying that the witnessed fringe-tree sizes are negatively associated, proving the second item. 

        For the third item of Lemma \ref{lem:beta-bin-fringe}, we use the exchangeability, and observe that $\mathrm{Fr}_n^{>m}(m)=\mathrm{Fr}_n(m)$. The size of a single fringe tree evolves according to a beta-binomial random variable plus one. That is, 
        \begin{equation}\label{eq:pmf-fringe-tree}
            \p\left(\mathrm{Fr}_n(m)=k\right) 
= (m-1)\frac{(n-m)!}{(n-1)!}\frac{(n-k-1)!}{(n-m-k+1)!}, \quad k \in \{1,\ldots,n-m\}.
        \end{equation}
        We start with the upper bound on the lower tail, i.e., inequality \eqref{eq:fringe-ub}. When $m=1$, $\mathrm{Fr}_n(1)=n$ and the bound is trivial. So we can assume that $n\ge m\ge2$ for the rest of the proof of \eqref{eq:fringe-ub}.
Observe that for $k \in \{1,\ldots,n-m+1\}$,
\[
\frac{\p\left(\mathrm{Fr}_n(m)=k+1\right)}{\p\left(\mathrm{Fr}_n(m)=k\right)} = \frac{(n-k-2)!}{(n-m-k)!} \frac{(n-m-k+1)!}{(n-k-1)!} = \frac{n-m-k+1}{n-k-1}=1 - \frac{m-2}{n-k-1} \leq 1.
\]
So the function $k\mapsto \p(\mathrm{Fr}_n(v)=k)$ is non-increasing for $m\ge 2$. Hence, 
\[
\p\left(\mathrm{Fr}_n(m)\le \frac{n}{M m}\right) \le \frac{n}{M m}\p\left(\mathrm{Fr}_n( m)=1\right)=\frac{n}{M m}(m-1)\frac{(n-m)!}{(n-1)!}\frac{(n-2)!}{(n-m)!}=\frac{n}{Mm}\frac{m-1}{n-1}.
\]
The right-hand side is at most $2/M$, since $n\geq 2$.
We proceed to the lower bound on the lower tail, inequality \eqref{eq:fringe-lb}, assuming $n\ge m\ge 2$, $m\le n/M$, and $M\ge 4$. Set
\begin{equation}\label{eq:lower-tail-useful}
k=\left\lfloor\frac{n}{Mm}\right\rfloor,\qquad\text{and observe}\qquad m-2\leq \frac{n}{4} \le \frac{n}{2}\left(1-\frac{1}{Mm}\right)\le \frac{n-k}{2}.
\end{equation}
Since the function $j \mapsto \p(\mathrm{Fr}_n(m)=j)$ is non-increasing for $m\geq 2$, we have 
\[
\begin{aligned}
\p\big(\mathrm{Fr}_n(m)\le k\big) \ \ge \ k  \cdot \p\big(\mathrm{Fr}_n(m)=k\big) \ & = \ k \cdot \p\big(\mathrm{Fr}_n(m)=1\big) \prod_{j=1}^{k-1} \frac{\p\left(\mathrm{Fr}_n(m)=j+1\right)}{\p\left(\mathrm{Fr}_n(m)=j\right)} 
\\
&
= 
\left\lfloor\frac{n}{Mm}\right\rfloor\frac{m-1}{n-1}\prod_{j=1}^{k-1}\left(1-\frac{m-2}{n-j-1}\right)\\&\ge \frac{1}{4M}\prod_{j=1}^{k-1}\left(1-\frac{m-2}{n-j-1}\right),
\end{aligned}
\]
where we used the formula for $\tfrac{\p\left(\mathrm{Fr}_n(m)=j+1\right)}{\p\left(\mathrm{Fr}_n(m)=j\right)}$ proven above.
By \eqref{eq:lower-tail-useful}, each term $\tfrac{m-2}{n-j-1}$ in the above product product is at most $1/2$. We use that $1-x\ge \exp(-2x)$ for all $x\in[0, 1/2]$. This yields 
\[
\p\left(\mathrm{Fr}_n(v)\le k\right)\ge \frac{1}{4M}\exp\bigg(-2(m-2)\sum_{j=1}^{k-1}\frac{1}{n-j-1}\bigg) =  \frac{1}{4M}\exp\bigg(-2(m-2)\sum_{j=2}^k\frac{1}{n-j}\bigg).
\]
Recall $k=\lfloor n/(Mm)\rfloor$. We bound the sum in the exponential from above by $k/(n-k)$, which is at most $1/(Mm-1)$. We obtain for $M\ge 4$ and $m\ge 2$,
\[
\p\left(\mathrm{Fr}_n(m)\le \frac{n}{Mm}\right) \ge  \frac{1}{4M}\exp\bigg(-2(m-2)\frac{1}{Mm-1}\bigg)
\geq
\frac{1}{4M}\exp\bigg(-\frac{2(m-2)}{4m-1}\bigg)
\ge \frac{1}{4\sqrt{e} M}\ge \frac{1}{8M}.
\]
This finishes the proof.
\end{proof}



\appendixfringesize*

\begin{proof}
    We work on the event  $\cE=\{\pa(v)\text{ is a leaf in }T_{v-1}, \mathrm{pa}(v) \geq 2\}$, which has probability $1/2$ for $v\geq 3$. Indeed, using that $\pa(v)$ is uniformly distributed on $[1,v-1]$ and independent from $T_{v-1}$, we see that

    \begin{align*}
        \p\left( \cE \right)
        = &\ 
        \frac{1}{v-1}\sum_{i=2}^{v-1}\p\left( i \text{ is a leaf in }T_{v-1} \right) 
        =
        \frac{1}{v-1}\left[1+\sum_{i=2}^{v-2}\prod_{j=i+1}^{v-1}\frac{j-2}{j-1}
        \right]\\
        =& \
        \frac{1}{v-1}\left[1+\sum_{i=2}^{v-2}\frac{i-1}{v-2}
        \right]
        =
        \frac{1}{v-1}\sum_{i=2}^{v-1}\frac{i-1}{v-2}
        =
        \frac{1}{(v-1)(v-2)} \sum_{i=1}^{v-2} i
        =\frac{1}{2}.
    \end{align*}
    Since, on the event $\cE$, $\mathrm{Fr}_n^{>v}(\mathrm{pa}(v)) + \mathrm{Fr}_n^{>v}(v)$ behaves like the number of red balls in a P\'olya urn, started with $2$ red balls and $v-2$ many blue balls at time $v$, we get that
    \begin{align*}
        &\mathrm{Fr}_n^{>v}(\mathrm{pa}(v)) + \mathrm{Fr}_n^{>v}(v) - 2 \sim \mathrm{BetaBin}\left( n - v , 2, v-2 \right) , \text{ and }
        \\
        &
        \mathrm{Fr}_n^{>v}(v)  \sim \mathrm{Unif} \left\{ 1 , \ldots , \mathrm{Fr}_n^{>v}(\mathrm{pa}(v)) + \mathrm{Fr}_n^{>v}(v) - 1 \right\}.
    \end{align*}
    We record a useful claim about the Beta-Binomial distribution, which we will prove below.

    \begin{claim}\label{claim:c_B}
    There exists a positive constant $c_{\mathrm{B}} > 0$ so that for all $n \geq 3, v \in \{3,\ldots,\lfloor \frac{n}{4} \rfloor \}$, and $k \in \{0,\ldots,\lfloor \frac{n}{v} \rfloor\}$, 
    \begin{equation}\label{eq:claim:c_B}
        \p \left( \mathrm{BetaBin}(n-v,2,v-2) = k \right) \geq c_{\mathrm{B}} k \frac{v^2}{n^2} .
    \end{equation}
    \end{claim}
    On the event $\cE$ one has 
    \begin{equation*}
    	\mathrm{Fr}_n(\mathrm{pa}(v)) = \mathrm{Fr}_n^{>v}(\mathrm{pa}(v)) + \mathrm{Fr}_n^{>v}(v).
    \end{equation*}
	Using the law of total probability, we can thus calculate
	\begin{align}\label{eq:log dep}
		\notag \p \left( \mathrm{Fr}_n(v) \mathrm{Fr}_n(\mathrm{pa}(v)) \leq s \left(\frac{n}{v}\right)^2 \right)
		&\geq
		\p \left( \mathrm{Fr}_n(v) \mathrm{Fr}_n(\mathrm{pa}(v)) \leq s \left(\frac{n}{v}\right)^2 \ \Big|\  \cE \right) \p \left(\cE\right)
		\\
		& \notag
		=
		\p \left( \mathrm{Fr}_n^{>v}(v) \left(\mathrm{Fr}_n^{>v}(\mathrm{pa}(v)) + \mathrm{Fr}_n^{>v}(v)\right) \leq s \left(\frac{n}{v}\right)^2\ \Big|\ \cE \right) \frac{1}{2} \\
		& \notag
		\geq
		\sum_{k=2}^{n-v}
		\p \left( \mathrm{Fr}_n^{>v}(v) k \leq s \left(\frac{n}{v}\right)^2\ \Big|\ \cE , \mathrm{Fr}_n^{>v}(\mathrm{pa}(v)) + \mathrm{Fr}_n^{>v}(v) = k \right) 
		\\
		&
		  \hspace{2cm} \cdot
		\p \left( \mathrm{Fr}_n^{>v}(\mathrm{pa}(v)) + \mathrm{Fr}_n^{>v}(v) = k\ \big|\ \cE \right) \frac{1}{2}. 
	\end{align}
   The events $\left\{\mathrm{Fr}_n^{>v}(\mathrm{pa}(v)) + \mathrm{Fr}_n^{>v}(v) = k \right\}$ and $\cE$ are independent, so that Claim \ref{claim:c_B} implies that for all $2 \leq k \leq \frac{n}{v}$,
	\begin{align*}
		 \p \left( \mathrm{Fr}_n^{>v}(\mathrm{pa}(v)) + \mathrm{Fr}_n^{>v}(v) = k \ \big|\  \cE \right)
		& 
		=
		\p \left( \mathrm{Fr}_n^{>v}(\mathrm{pa}(v)) + \mathrm{Fr}_n^{>v}(v) = k  \right)
		\\
		& 
		=
		\p \left( \mathrm{BetaBin}(n-v,2,v-2) = k - 2  \right) \geq c_{\mathrm{B}} (k-2) \frac{v^2}{n^2} ,
	\end{align*}
    where we used \eqref{eq:claim:c_B} for the last inequality.
	Further, conditioned on $\left\{\mathrm{Fr}_n^{>v}(\mathrm{pa}(v)) + \mathrm{Fr}_n^{>v}(v) = k , \cE\right\}$, $\mathrm{Fr}_n^{>v}(v)$ is still uniformly distributed on $\{1,\ldots,k-1\}$. Thus, for all $k=1,\ldots,\lfloor s \left(\frac{n}{v}\right)^2 \rfloor$,
	\begin{align*}
		& \p \left( \mathrm{Fr}_n^{>v}(v) k \leq s \left(\frac{n}{v}\right)^2 \ \Big|\  \cE , \mathrm{Fr}_n^{>v}(\mathrm{pa}(v)) + \mathrm{Fr}_n^{>v}(v) = k \right) 
		\\
		&
		\hspace{3cm} = \p \left( \mathrm{Fr}_n^{>v}(v) \leq \frac{s}{k} \left(\frac{n}{v}\right)^2 \ \Big|\  \cE , \mathrm{Fr}_n^{>v}(\mathrm{pa}(v)) + \mathrm{Fr}_n^{>v}(v) = k \right) 
		\\
		&
		\hspace{3cm}
		=
		\frac{\lfloor \frac{s}{k} \left(\frac{n}{v}\right)^2 \rfloor}{k-1} \geq \frac{1}{2} \frac{s}{k^2} \left(\frac{n}{v}\right)^2 .
	\end{align*}
	Inserting the two previous inequalities into \eqref{eq:log dep}, we get that
	\begin{align}
		\p \left( \mathrm{Fr}_n(v) \mathrm{Fr}_n(\mathrm{pa}(v)) \leq s \left(\frac{n}{v}\right)^2 \right) &\geq \frac{1}{2} \sum_{k=3}^{ \lfloor s \left(\frac{n}{v}\right)^2 \wedge \frac{n}{v} \rfloor}  c_{\mathrm{B}} (k-2) \frac{v^2}{n^2}   \frac{1}{2} \frac{s}{k^2} \left(\frac{n}{v}\right)^2 \notag
		\\
		&=
		\frac{s c_{\mathrm{B}} }{4} \sum_{k=3}^{ \lfloor s \left(\frac{n}{v}\right)^2 \wedge \frac{n}{v} \rfloor} \frac{k-2}{k^2} 
		\geq
		c^\star s  \log \left(  s \left(\frac{n}{v}\right)^2 \wedge \frac{n}{v} \right),\label{eq:almost there}
	\end{align}
    where $c^\star > 0$ is a positive constant.
	Using that $s \geq \left(\frac{v}{n}\right)^{2-\eps}$, we see that 
	\begin{align*}
		\log \left(  s \left(\frac{n}{v}\right)^2 \wedge \frac{n}{v} \right) 
		&
		= 
		\mathbbm{1}_{\left(\frac{v}{n}\right)^{2-\eps} \leq s \leq \frac{v}{n}}
		\log \left(  s \left(\frac{n}{v}\right)^2 \right) 
		+
		\mathbbm{1}_{s > \frac{v}{n}}
		\log \left( \frac{n}{v} \right)
		\\
		&
		 \geq 
		\mathbbm{1}_{\left(\frac{v}{n}\right)^{2-\eps} \leq s \leq \frac{v}{n}}
		\log \left(  \left(\frac{n}{v}\right)^\eps \right) 
		+
		\mathbbm{1}_{s > \frac{v}{n}}
		\log \left( \frac{1}{s} \right)
		\\
		&
		\geq 
		\mathbbm{1}_{\left(\frac{v}{n}\right)^{2-\eps} \leq s \leq \frac{v}{n}} \frac{\eps}{2-\eps}
		\log \left( \frac{1}{s} \right) 
		+
		\mathbbm{1}_{s > \frac{v}{n}}
		\log \left( \frac{1}{s} \right)
		\geq
		\frac{\eps}{2} \log \left( \frac{1}{s} \right),
	\end{align*}
	and inserting this into \eqref{eq:almost there} finishes the proof of \eqref{eq:lem2.6 (1)}.

    To prove \eqref{eq:lem2.6 (2)}, it suffices to consider $s$ with $3 (v/n)^2 \leq s \leq v/n$. From \eqref{eq:almost there}, we see that
    \begin{equation*}
		\p \left( \mathrm{Fr}_n(v) \mathrm{Fr}_n(\mathrm{pa}(v)) \leq s \left(\frac{n}{v}\right)^2 \right) 
        \geq 
		\frac{s c_{\mathrm{B}} }{4} \sum_{k=3}^{ \lfloor s \left(\frac{n}{v}\right)^2 \wedge \frac{n}{v} \rfloor} \frac{k-2}{k^2} 
        \geq 
		\frac{s c_{\mathrm{B}} }{12} \sum_{k=3}^{ \lfloor s \left(\frac{n}{v}\right)^2 \rfloor} \frac{1}{k} \geq \frac{c_{\mathrm{B}} }{36} s .\qedhere
	\end{equation*}
    \end{proof}

     \begin{proof}[Proof of Claim \ref{claim:c_B}]
    	Writing $B(x,y)=\frac{(x-1)!(y-1)!}{(x+y-1)!}$ for the Beta function, we get that
    	\begin{equation}\nonumber
    		\begin{aligned}
    			\p \left( \mathrm{BetaBin}(n-v,2,v-2) = k \right) & =
    			\binom{n-v}{k}\frac{B(k+2, n-(k+2))}{B(2, v-2)} \\
    			&=
    			\frac{(n-v)!}{k!(n-v-k)!} \frac{(k+1)!(n-k-3)!}{(n-1)!} \frac{(v-1)!}{(v-3)!1!}
    			\\
    			&
    			=
    			\frac{(k+1)(v-1)(v-2)}{(n-1)(n-2)} \frac{(n-v)!}{(n-v-k)!} \frac{(n-3-k)!}{(n-3)!}.
    		\end{aligned}
    	\end{equation}
    	For the first term in the above product, we have, since $v\geq 3$, that
    	\begin{equation*}
    		\frac{(k+1)(v-1)(v-2)}{(n-1)(n-2)} \geq  \frac{k(v-1)(v-2)}{n^2} \geq \frac{k v^2}{6 n^2}.
    	\end{equation*}
        To bound the remaining terms, assume first that $v \geq k$.
    	Using that $\frac{2v}{n}\leq \frac{1}{2}$ and the elementary inequality $1-x\geq \exp(-2x)$ for $x \in \left[ 0, \frac{1}{2} \right]$, we get that
    	\begin{multline*}
    		\frac{(n-v)!}{(n-v-k)!} \frac{(n-3-k)!}{(n-3)!} = \prod_{j=0}^{k-1} \frac{n-v-j}{n-3-j} \geq \prod_{j=0}^{k-1} \frac{n-v-k}{n} \geq \left( \frac{n-2v}{n} \right)^k \\ = \left( 1 - \frac{2v}{n} \right)^k \geq \exp \left( - \frac{4kv}{n} \right) \geq \exp\left( -4 \right),
    	\end{multline*}
    	where we used that $k \leq \frac{n}{v}$ for the last inequality.
    	Similarly, for $v < k$ we get that
    	\begin{multline*}
    		\frac{(n-v)!}{(n-v-k)!} \frac{(n-3-k)!}{(n-3)!} 
            =
            \frac{(n-v)!}{(n-3)!} \frac{(n-3-k)!}{(n-v-k)!} 
            = 
            \prod_{j=0}^{v-4} \frac{n-3-k-j}{n-3-j} 
            \\
    		\geq
    		\prod_{j=0}^{v-4} \frac{n-v-k}{n}  \geq \left( \frac{n-2k}{n} \right)^v
    		= \left( 1 - \frac{2k}{n} \right)^v \geq \exp \left( - \frac{4kv}{n} \right) \geq \exp\left( -4 \right).
    	\end{multline*}
    	Combining the three previous inequalities, we get that
    	\begin{equation*}
    		\p \left( \mathrm{BetaBin}(n-v,2,v-2) = k \right) \geq \frac{e^{-4}}{6} k \frac{v^2}{n^2} ,
    	\end{equation*}
    	proving the claim.
    \end{proof}
 
\fringetreedistrib*

\begin{proof}
    Fix some $n\geq3$, $v\geq 2$ and $1\leq k \leq n-v+1$. Then, we use that, for $v\geq 2$,
    \begin{equation*}
        \p\left(\mathrm{Fr}_n(v)=k\right) 
        = (v-1) \frac{(n-v)!}{(n-1)!}\frac{(n-k-1)!}{(n-v-k+1)!} = \frac{v-1}{n-1}\prod_{j=0}^{k-2}\frac{n-v-j}{n-2-j}\leq \frac{v-1}{n-1}\left( \frac{n-v}{n-2}\right)^{k-1},
    \end{equation*}
    as the factors are decreasing in $j$.
    Using that $\tfrac{v-1}{n-1}\leq 2v/n$, that $(\tfrac{n}{n-2})^{k-1}\leq (\tfrac{n}{n-2})^{n-2} \leq e^2$, that $\log(1-v/n)\leq -v/n$, and that $\exp(-(k-1)v/n)\leq \exp(1-kv/n)$, we get
    \begin{align}
        &\notag \frac{v-1}{n-1}\le 2v/n, \qquad \text{ and }
        \\
        &\label{eq:k}
        \left( \frac{n-v}{n-2}\right)^{k-1}=\left( 1-\frac{v-2}{n-2}\right)^{k-1} \leq \exp\left(-(k-1)\frac{v-2}{n-2}\right) \leq \exp\bigg(2-k\frac{v}{n}\bigg),
    \end{align}
    where the last inequality holds since $1 \leq k \leq n-v+1$ and $v \geq 2$: indeed, since the two functions $k \mapsto \tfrac{-(k-1)(v-1)}{n-1}$ and $k \mapsto 2-\tfrac{kv}{n}$ are linear in $k$, it suffices to prove the inequality for $k=1$ and for $k=n-v+1$. The statement for $k=1$ is clear, since $v\leq n$. For $k=n-v+1$, the statement is equivalent to 
    \begin{equation}\label{eq:h}
        (n-v+1) \frac{v}{n}- (n-v)\frac{v-2}{n-2}-2 = \frac{-3nv+2v^{2}-2v+4n}{n\left(n-2\right)} \leq 0.
    \end{equation}
    The function $h(v) \coloneqq -3nv+2v^{2}-2v+4n$ is convex, so that the maximum in the domain $[2,n]$ is attained at one of the boundary values. One readily checks that $h(2),h(n) \leq 0$, so that $h(v)\leq 0$ for all $v\in [2,n]$, implying \eqref{eq:h} and thus also \eqref{eq:k}.
    Because \eqref{eq:proba_fringe-k} is also trivially true for $k>n-v+1$, these bounds yield \eqref{eq:proba_fringe-k}. \smallskip 
    
    For the proof of \eqref{eq:proba_fringe-large}, it suffices to consider $t > 2$, and thus also for $v\geq 3$, since the probability of the event on the left-hand side of \eqref{eq:proba_fringe-large} equals zero for $v=2, t>2$. Further, the statement is clear when $t\tfrac{n}{v} > n-v+1$. Using the previously established bound on $\p(\mathrm{Fr}_n(v)=k)$, we obtain for $v\ge 3$,
    \begin{multline}
        \p\left(\mathrm{Fr}_n(v)\geq t\frac{n}{v} \right) 
        \leq 
        \frac{v-1}{n-1}\sum_{k\geq t\tfrac{n}{v} }\bigg(\frac{n-v}{n-2}\bigg)^{k-1} 
        \le \frac{v-1}{n-1}\frac{n-2}{v-2}\bigg(\frac{n-v}{n-2}\bigg)^{t(n/v)-1}
        \\
        \le \frac{v-1}{v-2} \bigg(\frac{n-v}{n-2}\bigg)^{t(n/v)-1}
        \le 2 \bigg(\frac{n-v}{n-2}\bigg)^{t(n/v)-1}
        \le 2e^2\exp\left(-t\right).
    \end{multline}
    In the last inequality, we used that $(\tfrac{n-v}{n-2})^{t(n/v)-1} \leq \exp(2-t\tfrac{n}{v}\tfrac{v}{n})$ for $t\tfrac{n}{v} \leq n-v+1$, as proven in inequality \eqref{eq:k}.
    This proves \eqref{eq:proba_fringe-large} for $v\ge 3$.
    
    \smallskip 
   We next prove \eqref{eq:proba-parent-large-lower}. We use the de-Finetti representation of the Beta-Binomial distribution and see that
    \begin{align}
       \p \left( \mathrm{Fr}_n(v) > t \frac{n}{v} \right) 
       & \notag
       =
       \int_{0}^{1} \p \left( \mathrm{Bin}(n-v,p) + 1 > t \frac{n}{v} \right)(v-1)(1-p)^{v-2} \md p
       \\
       & \label{eq:definetti}
       \geq
       \int_{\frac{4t}{v}}^{1} \p \left( \mathrm{Bin}\left( \Big\lceil\frac{n}{2}\Big\rceil ,p\right)  \geq t \frac{n}{v} \right)(v-1)(1-p)^{v-2} \md p .
    \end{align}
    First, observe that for a Binomial random variable $X_p$ with parameters $\lceil \frac{n}{2} \rceil$ and $p \geq \frac{4t}{v}$, one has
    \begin{align*}
        \mathrm{Var}(X_p) \leq \E \left[ X_p \right] \quad \text{ and } \quad \E \left[ X_p \right] = \Big\lceil \frac{n}{2} \Big\rceil p \geq \frac{n}{2} \frac{4 t}{v} = 2t \frac{n}{v} \geq 4,
    \end{align*}
    where the last inequality holds since $v \leq \tfrac{n}{2}, t \geq 1$.
    Combining these facts with Cantelli's inequality, we get that
    \begin{multline*}
    \p \left( \mathrm{Bin}\left( \Big\lceil\frac{n}{2}\Big\rceil ,p\right) \geq t \frac{n}{v} \right)
    \geq
        \p \left( X_p \geq \frac{\E\left[X_p\right]}{2} \right) \geq 1 - \frac{\mathrm{Var}(X_p)}{\mathrm{Var}(X_p) + \E\left[X_p\right]^2/4 }
        \\
        =
        \frac{\E\left[ X_p \right]^2/4}{\mathrm{Var}(X_p) + \E\left[ X_p \right]^2/4}
        \geq 
        \frac{\E\left[ X_p \right]^2/4}{\E\left[ X_p \right] + \E\left[ X_p \right]^2/4}
        =
        \frac{\E\left[ X_p \right]/4}{1+\E\left[ X_p \right]/4}
        \geq
        \frac{1}{2} .
    \end{multline*}
    Inserting this bound into \eqref{eq:definetti} yields that
    \begin{align*}
       \p \left( \mathrm{Fr}_n(v) > t \frac{n}{v} \right) 
       &\geq
       \int_{\frac{4t}{v}}^{1} \p \left( \mathrm{Bin}\left( \Big\lceil\frac{n}{2}\Big\rceil ,p\right)  \geq t \frac{n}{v} \right)(v-1)(1-p)^{v-2} \md p  \\
       &\geq
       \frac{1}{2}
       \int_{\frac{4t}{v}}^{1}(v-1)(1-p)^{v-2} \md p \\
       &=
       \frac{1}{2} \left( 1- \frac{4t}{v} \right)^{v-1}
       \geq \frac{1}{2} \exp\left( - \frac{8t}{v} \right)^{v-1}
       \geq \frac{1}{2} \exp\left( - 8t \right),
    \end{align*}
    where we used the elementary inequality $1-x \geq e^{-2x}$ for all $x \in \left[0,\frac{1}{2} \right]$ in the last line. This finishes the proof of \eqref{eq:proba-parent-large-lower}.

\smallskip 
    For the proof of inequality \eqref{eq:proba-parent-large2}, we start by decomposing on the event $\{\mathrm{Fr}_n(v)\geq t n/(2v)\}$ and its complement, i.e.,
    \begin{align*}
        \p\left( \mathrm{Fr}_n(\pa(v))\geq t \frac{n}{v}  \right)
        \leq
        \p\left( \mathrm{Fr}_n(v)\geq \frac{t}{2} \frac{n}{v}  \right)
        +
        \p\left( \mathrm{Fr}_n(\pa(v))\geq t \frac{n}{v},  \ \mathrm{Fr}_n(v)\leq t \frac{n}{2v} \right).
    \end{align*}
    For the first summand in the above sum, we use inequality \eqref{eq:proba_fringe-large} with $t/2$. Using further that $\pa(v)$ is uniformly distributed on $[v-1]$ on the second summand, we obtain
    \begin{equation}
        \label{eq:split law of total prob}\p\left( \mathrm{Fr}_n(\pa(v))\geq t \frac{n}{v}  \right)
    \leq
    2e^{-t/2 + 2}
    +
    \frac{1}{v-1}\sum_{i=1}^{v-1} \p\left( \mathrm{Fr}_n(i)\geq t \frac{n}{v}, \  \mathrm{Fr}_n(v)\leq t \frac{n}{2v} \ \Big| \ \pa(v)=i \right).
    \end{equation}
    If $\pa(v)=i$ and $\mathrm{Fr}_n(v)\leq t \frac{n}{2v}$, then the subtree $(T,1)_{i\downarrow}\setminus(T,1)_{v\downarrow}$ must have at least $t \frac{n}{2v}$ elements for $\mathrm{Fr}_n(i)$ to exceed $t \frac{n}{v}$. So,

    \begin{align*}
        \p\left( \mathrm{Fr}_n(i)\geq t \frac{n}{v} , \mathrm{Fr}_n(v)\leq t \frac{n}{2v} \ \Big| \ \pa(v)=i \right)
        & \leq
        \p\left( \big|(T,1)_{i\downarrow}\setminus (T,1)_{v\downarrow}|\geq t \frac{n}{2v}  \ \Big| \ \pa(v)=i\right).
    \end{align*}
    One can couple a RRT and a RRT conditioned on $\{\pa(v)=i\}$  by grafting $v$ and its subtree at $i$. Doing so, it is direct that $|(T,1)_{i\downarrow}\setminus (T,1)_{v\downarrow}| $ conditioned on $ \pa(v)=i$ is stochastically dominated by $\mathrm{Fr}_n(i)$. So,

    \begin{multline*}
        \p\left( \mathrm{Fr}_n(i)\geq t \frac{n}{v}, \mathrm{Fr}_n(v)\leq t \frac{n}{2v} \ \Big| \  \pa(v)=i \right)
        \\
    \leq
    \p\left( |(T,1)_{i\downarrow}\setminus (T,1)_{v\downarrow}| \geq  t \frac{n}{2v} \ \Big| \ \pa(v)=i \right)
    \leq
    \p\left( \mathrm{Fr}_n(i)\geq t \frac{n}{2v}   \right).
    \end{multline*}
    Inserting this inequality into \eqref{eq:split law of total prob} and using \eqref{eq:proba_fringe-large}, we get
    \begin{align*}
        \p\left( \mathrm{Fr}_n(\pa(v))\geq t \frac{n}{v}  \right) &
        \leq
    2e^{-t/2 + 2}
    +
    \frac{1}{v-1}\sum_{i=1}^{v-1} \p\left( \mathrm{Fr}_n(i)\geq t \frac{n}{2v}  \right)
    \\
    &
    =
    2e^{-t/2 + 2}
    +
    \frac{1}{v-1}\sum_{i=1}^{v-1} \p\left( \mathrm{Fr}_n(i)\geq i \frac{t}{2v} \frac{n}{i}  \right)
    \\
    & 
        \leq
        2e^{-t/2+2}
        +
        \frac{1}{v-1} \sum_{i=1}^{v-1} 2\exp\left(-i \frac{t}{2v}+2\right)
        \\
    & 
        \leq
        2e^{-t/2+2}
        +
        \frac{2 e^2}{v-1} \sum_{i=1}^{\infty} \exp\left(-i \frac{t}{2v}\right)
        \\
        &
        =
        2e^{-t/2+2}
        +
        \frac{2 e^2}{v-1} \frac{1}{\exp(t/(2v))-1}
        \leq 
        2e^{-t/2+2}
        +
        \frac{2 e^2}{v-1} \frac{2v}{t},
    \end{align*}
    where we used the elementary inequality $\exp(t/(2v))-1 \geq t/(2v)$ for the last inequality.
    The bound \eqref{eq:proba-parent-large2} follows. Inequality \eqref{eq:proba-parent-large} follows from an application of \eqref{eq:proba-parent-large2} with $t=\frac{v}{2}$.
\end{proof}

\Fringefringeparentlarge*

\begin{proof}
    We start by treating the cases $v=1,2$. Using that $\phi_n^\ssup{2}(v)\le \mathrm{Fr}_n(\pa(v))\mathrm{Fr}_n(v)$ at $v=2$ yields

    $$\p\Big(\phi_n^\ssup{2}(v) \ge S\left(\frac{n}{v}\right)^2 \Big) \le
    \p\Big( \mathrm{Fr}_n(1)\mathrm{Fr}_n(2) \geq S\frac{n^2}{4}\Big)
    \overset{ \eqref{eq:proba_fringe-large}}{\leq}
    2e^2 e^{-S/2}
    \leq
    \frac{C}{S}$$
    for some positive constant $C$. For $v=1$ it suffices to use that $\phi_n^\ssup{2}(1)$ and $\phi_n^\ssup{2}(2)$ are identically distributed, leading to 

    $$\p\Big(\phi_n^\ssup{2}(1) \ge S\left(\frac{n}{1}\right)^2 \Big)
    =
    \p\Big(\phi_n^\ssup{2}(2) \ge 4S\left(\frac{n}{2}\right)^2 \Big)
    \leq \frac{C}{4S},$$
    concluding the proof of the lemma for $v=1,2$. For $v\in [3,n]$, we write for any $u\in [v-1]$
    \begin{equation*}
        \mathrm{Fr}_n^{\neq v}(u) \coloneqq \mathrm{Fr}_n(u) - \mathrm{Fr}_n(v) \mathbbm{1}_{v \in (T_n,1)_{u\downarrow}}
    \end{equation*}
    for the number of vertices in the fringe tree of $u$ that are not contained in the fringe tree of $v$. In particular, for $\mathrm{pa}(v)$ one has that
    \begin{equation*}
        \mathrm{Fr}_n(\mathrm{pa}(v)) = \mathrm{Fr}_n^{\neq v}(\mathrm{pa}(v)) + \mathrm{Fr}_n(v).
    \end{equation*}
    This directly implies that 
    \begin{equation*}
        \phi_n^\ssup{2}(v)\le \mathrm{Fr}_n(\pa(v))\mathrm{Fr}_n(v) = \mathrm{Fr}_n^{\neq v}(\mathrm{pa}(v)) \mathrm{Fr}_n(v) + \mathrm{Fr}_n(v)^2 .
    \end{equation*}
    Thus we get that 
    \begin{align}\label{eq:Two terms S/2}
        \p \left( \phi_n^{\ssup{2}}(v) \geq S \left( \frac{n}{v} \right)^2 \right) 
        \leq 
        \p \left( \mathrm{Fr}_n^{\neq v}(\mathrm{pa}(v)) \mathrm{Fr}_n(v) \geq \frac{S}{2} \left( \frac{n}{v} \right)^2 \right) + \p \left( \mathrm{Fr}_n(v)^2 \geq \frac{S}{2} \left( \frac{n}{v} \right)^2 \right).
    \end{align}
    By \eqref{eq:proba_fringe-large}, there exists a constant $C^\prime < \infty$ so that $\E\left[ \mathrm{Fr}_n(v)^2 \right] \leq C^\prime (n/v)^2$ for all $n\in \N, v \in [n]$. Markov's inequality thus implies that
    \begin{equation}\label{eq:Two terms S/2 term 1}
        \p \left( \mathrm{Fr}_n(v)^2 \geq \frac{S}{2} \left( \frac{n}{v} \right)^2 \right) \leq
        \frac{\E \left[ \mathrm{Fr}_n(v)^2 \right]}{\frac{S}{2} \left( \frac{n}{v} \right)^2} \leq \frac{2 C^\prime}{S} .
    \end{equation}
    For the first term in \eqref{eq:Two terms S/2}, we use the law of total probability to get that.
    \begin{align}\label{parent switch}
        & \p \left( \mathrm{Fr}_n^{\neq v}(\mathrm{pa}(v)) \mathrm{Fr}_n(v) \geq \frac{S}{2} \left( \frac{n}{v} \right)^2 \right) = \frac{1}{v-1} \sum_{u=1}^{v-1}  \p \left( \mathrm{Fr}_n^{\neq v}(u) \mathrm{Fr}_n(v) \geq \frac{S}{2} \left( \frac{n}{v} \right)^2\ \Big|\ \mathrm{pa}(v)=u \right). 
    \end{align}
    For fixed $u\in [v-1]$, the random variable $\mathrm{Fr}_n^{\neq v}(u) \mathrm{Fr}_n(v)$ is independent of $\mathrm{pa}(v)$. Indeed, both $\mathrm{Fr}_n(v)$ and $\mathrm{Fr}_n^{\neq v}(u)$ are determined by $\left( \mathrm{pa}(w) ; w \in [n] \setminus \{v\} \right)$. Further, the random variables $\mathrm{Fr}_n^{\neq v}(u) $ and $\mathrm{Fr}_n(v)$ are negatively associated. Indeed, conditioned on $\mathrm{Fr}_n(v)=k$, the tree $T_n \setminus (T_n,1)_{v \downarrow}$ is just distributed like a uniform attachment tree of size $n-k$. Thus, we see that
    \begin{align*}
        & \p \left( \mathrm{Fr}_n^{\neq v}(u) \mathrm{Fr}_n(v) \geq \frac{S}{2} \left( \frac{n}{v} \right)^2\ \Big|\ \mathrm{pa}(v)=u \right) 
        \\
        &
        = \p \left( \mathrm{Fr}_n^{\neq v}(u) \mathrm{Fr}_n(v) \geq \frac{S}{2} \left( \frac{n}{v} \right)^2 \right)
         \\
        &
        \leq
        \p \left( \mathrm{Fr}_n^{\neq v}(u) \geq \frac{S}{2} \frac{n}{v} \right) + \sum_{\ell=0}^{\infty} \p \left(  \mathrm{Fr}_n(v) \geq 2^{\ell} \frac{n}{v}, \mathrm{Fr}_n^{\neq v}(u) \geq \frac{S}{4} 2^{-\ell}  \frac{n}{v} \right)
         \\
        &
        \leq
        \p \left( \mathrm{Fr}_n^{\neq v}(u) \geq \frac{S}{2} \frac{n}{v} \right) + \sum_{\ell=0}^{\infty} \p \left(  \mathrm{Fr}_n(v) \geq 2^{\ell} \frac{n}{v}\right) \p \left( \mathrm{Fr}_n^{\neq v}(u) \geq \frac{S}{4} 2^{-\ell}  \frac{n}{v} \right)
         \\
        &
        \overset{ \eqref{eq:proba_fringe-large}}{\leq}
        \p \left( \mathrm{Fr}_n^{\neq v}(u) \geq \frac{S}{2} \frac{n}{v} \right) + \sum_{\ell=0}^{\infty} 2 e^2 \exp\left(-2^\ell\right) \p \left( \mathrm{Fr}_n^{\neq v}(u) \geq \frac{S}{4} 2^{-\ell}  \frac{n}{v} \right).
    \end{align*}
    Inserting this into \eqref{parent switch} and using again that $\mathrm{Fr}_n^{\neq v}(u)$ and $\mathrm{pa}(v)$ are independent, we get that
    \begin{align*}
        & \p \left( \mathrm{Fr}_n^{\neq v}(\mathrm{pa}(v)) \mathrm{Fr}_n(v) \geq \frac{S}{2} \left( \frac{n}{v} \right)^2 \right) 
        \\
        &
        \leq
        \frac{1}{v-1} \sum_{u=1}^{v-1} \left( \p \left( \mathrm{Fr}_n^{\neq v}(u) \geq \frac{S}{2} \frac{n}{v} \right) + \sum_{\ell=0}^{\infty} 2 e^2 \exp\left(-2^\ell\right) \p \left( \mathrm{Fr}_n^{\neq v}(u) \geq \frac{S}{4} 2^{-\ell}  \frac{n}{v} \right)  \right)
        \\
        &
        =
        \frac{1}{v-1} \sum_{u=1}^{v-1}  \p \left( \mathrm{Fr}_n^{\neq v}(u) \geq \frac{S}{2} \frac{n}{v} \ \Big|\  \mathrm{pa}(v)=u \right)
        \\
        &
        \hspace{2cm}
        + 
        \sum_{\ell=0}^{\infty} 2 e^2 \exp\left(-2^\ell\right) \frac{1}{v-1} \sum_{u=1}^{v-1}  \p \left( \mathrm{Fr}_n^{\neq v}(u) \geq \frac{S}{4} 2^{-\ell}  \frac{n}{v} \ \Big|\  \mathrm{pa}(v)=u \right)  
        \\
        & 
        =
        \p \left( \mathrm{Fr}_n^{\neq v}(\mathrm{pa}(v)) \geq \frac{S}{2} \frac{n}{v} \right)
        +
        \sum_{\ell=0}^{\infty} 2 e^2 \exp\left(-2^\ell\right)  \p \left( \mathrm{Fr}_n^{\neq v}(\mathrm{pa}(v)) \geq \frac{S}{4} 2^{-\ell}  \frac{n}{v} \right)
        \\
        & 
        \leq
        \p \left( \mathrm{Fr}_n(\mathrm{pa}(v)) \geq \frac{S}{2} \frac{n}{v} \right)
        +
        \sum_{\ell=0}^{\infty} 2 e^2 \exp\left(-2^\ell\right)  \p \left( \mathrm{Fr}_n(\mathrm{pa}(v)) \geq \frac{S}{4} 2^{-\ell}  \frac{n}{v} \right) .
    \end{align*}
    By Lemma \ref{lem:proba_large_fringe}, there exists a constant $C<\infty$ so that $\p \left( \mathrm{Fr}_n(\mathrm{pa}(v)) \geq tn/v \right) \leq C/t$ for all $t>0$. Further, let $C^\star>0$ be a constant so that $\exp \left(-x\right) \leq C^\star x^{-2}$ for all $x\geq 1$. Then we get that
    \begin{align*}
        \p \left( \mathrm{Fr}_n^{\neq v}(\mathrm{pa}(v)) \mathrm{Fr}_n(v) \geq \frac{S}{2} \left( \frac{n}{v} \right)^2 \right) 
        &\leq
        \frac{C}{S/2}
        +
        \sum_{\ell=0}^{\infty} 2 e^2 \exp\left(-2^\ell\right)  \frac{C}{\frac{S}{4} 2^{-\ell}}
        \\
        &\leq
        \frac{2C}{S}
        +
        \sum_{\ell=0}^{\infty} 2 e^2 C^\star 2^{-2\ell}  \frac{4C}{S 2^{-\ell}} \\
        &=
        \frac{2C}{S}
        +
        \frac{8 e^2 C^\star C}{S}
        \sum_{\ell=0}^{\infty} 2^{-\ell}
        =
        \frac{2C}{S}
        +
        \frac{16 e^2 C^\star C}{S}.
    \end{align*}
    Inserting this inequality and inequality \eqref{eq:Two terms S/2 term 1} into \eqref{eq:Two terms S/2} shows that
    \begin{equation*}
        \p \left( \phi_n^{\ssup{2}}(v) \geq S \left( \frac{n}{v} \right)^2 \right)  \leq \frac{2C^\prime}{S} + \frac{2C}{S}
        +
        \frac{16 e^2 C^\star C}{S},
    \end{equation*}
    finishing the proof.
\end{proof}

\noindent\textbf{Acknowledgements.\ }This material is partly based upon work supported by the National Science Foundation under Grant No.\ DMS-1928930, while SB and JJ were in residence at the Simons Laufer Mathematical Sciences Institute in Berkeley, California, during the spring semester of 2025. JJ additionally thanks Magdalen College, Oxford, for a Senior Demyship and appreciates his baby for patiently waiting until the night after the manuscript was completed to make an appearance. We thank Miklós Rácz for helpful comments.

{\small 
\bibliographystyle{abbrv}
\bibliography{references}




}

\end{document}